\documentclass[a4paper,11pt,DIV=11,
abstract=on
]{scrartcl}
\usepackage{mathtools}
\mathtoolsset{showonlyrefs}
\usepackage{amsmath, amsthm, amssymb}
\usepackage{fontenc}
\usepackage[utf8]{inputenc}
\usepackage{graphicx}
\usepackage{xargs}
\usepackage[subrefformat=parens,labelformat=parens]{subcaption}
\usepackage{algorithm, booktabs}
\usepackage[noend]{algpseudocode}
\usepackage{enumitem,listings,microtype,tikz,pgfplots}
\usepackage{hyperref}
\usepackage{environ,ifthen}
\usepackage{multirow}
\usepackage{verbatim}
 \usepackage{bm}
 \usepackage{bbm}

\newtheorem{Theorem}{Theorem}[section]

\newtheorem{Lemma}[Theorem]{Lemma}
\newtheorem{Remark}[Theorem]{Remark}

\newtheorem{Assumption}[Theorem]{Assumption}
\setkomafont{sectioning}{\rmfamily\bfseries}
\setkomafont{title}{\rmfamily}
\setkomafont{descriptionlabel}{\rmfamily\bfseries} 

\usetikzlibrary{spy}
\usetikzlibrary{calc}
\usetikzlibrary{external}

\usepackage{algorithm}
\usepackage{booktabs}

\usepackage[noend]{algpseudocode}

\usepackage{setspace}			
\usepackage{flafter}			

\usepackage{appendix}
\usepackage{hyperref}	
\hypersetup{colorlinks=true,linkcolor=black,urlcolor=blue,citecolor=black}		
\usepackage{placeins} 		

\usepackage{listings}
\usepackage{empheq}

\newcommand{\N}{\mathbb{N}}
\newcommand{\Z}{\mathbb{Z}}

\newcommand{\R}{\mathbb{R}}

\newcommand{\SP}{\mathbb{S}}

\newcommand{\NN}{\mathcal{N}}

\newcommand{\GG}{\mathcal{G}}
\newcommand{\II}{\mathcal{I}}
\newcommand{\HH}{\mathcal{H}}

\newcommand{\LL}{\mathcal{L}}
\renewcommand{\SS}{\mathcal{S}}

\newcommand{\E}{\mathbb{E}}
\renewcommand{\P}{\mathbb{P}}

\newcommand{\abs}[1]{\left| #1 \right|}
\newcommand{\norm}[2]{\left\| #1 \right\|_{#2}}

\newcommand{\del}{\partial}

\newcommand{\dx}[1][x]{\,\mathrm{d}#1}

\newcommand{\tT}{\mathrm{T}}

\newcommand{\e}{\mathrm{e}}

\newcommand{\imag}{\mathrm{i}}

\renewcommand\theta{\vartheta}

\DeclareMathOperator*{\argmin}{argmin}

\DeclareMathOperator{\conv}{conv}

\DeclareMathOperator{\tr}{tr}
\DeclareMathOperator{\rang}{rank}

\DeclareMathOperator{\lspan}{span}

\DeclareMathOperator{\Cov}{Cov}

\DeclareMathOperator{\SPD}{SPD}

\DeclareMathOperator{\Sym}{Sym}
\DeclareMathOperator{\modulo}{mod}
\allowdisplaybreaks[4]
\begin{document}

\title{Multivariate Myriad Filters based on Parameter Estimation of the Student-$t$ Distribution}
\author{Friederike Laus\footnotemark[1]
\and
Gabriele Steidl\footnotemark[1] \footnotemark[2]}

\maketitle

\footnotetext[1]{Department of Mathematics,
Technische Universit\"at Kaiserslautern,
Paul-Ehrlich-Str.~31, D-67663 Kaiserslautern, Germany,
\{friederike.laus,steidl\}@mathematik.uni-kl.de.
} 
\footnotetext[2]{Fraunhofer ITWM, Fraunhofer-Platz 1,
D-67663 Kaiserslautern, Germany}

\begin{abstract}
The contribution of this study is twofold: 
First, we  propose an efficient algorithm for the computation  
of the (weighted) maximum likelihood estimators for the parameters of the multivariate Student-$t$ distribution, 
which we call generalized multivariate myriad filter. 
Second, we use the generalized multivariate myriad filter in a nonlocal framework for the denoising of images corrupted by different kinds of noise. 
The resulting method is very flexible and can handle heavy-tailed noise such as Cauchy noise, 
as well as the other extreme, namely Gaussian noise.
Furthermore, we detail how the limiting case $\nu \rightarrow 0$
of the  projected normal distribution in two dimensions
can be used for the robust denoising of periodic data, 
in particular for images with circular data  corrupted by wrapped Cauchy noise.
\end{abstract}

\section{Introduction} \label{sec:intro}
Besides mean and median filter, myriad filter form an important class of nonlinear filters, 
in particular in robust signal and image processing. 
While in a multivariate setting, the mean filter can be defined componentwise, 
the generalization of the median to higher dimensions is not canonically, but often the geometric median is used, see, e.g.,~\cite{SST12}.
In this paper we make a first attempt to establish a multivariate myriad filter based on the multivariate Student-$t$ distribution.  

In one dimension, mean, median as well as myriad filters can
be derived as   maximum likelihood (ML) estimators of the location parameter from a  Gaussian, Laplacian respective Cauchy distribution.
Concerning a higher dimensional myriad filter, instead of a multivariate Cauchy distribution 
we propose to start with the family of more general Student-$t$ distributions, 
which possesses an additional degree of freedom parameter $\nu > 0$ 
that allows to control the robustness of the resulting filter. 
While the Cauchy distribution is  obtained as the special case $\nu=1$, 
the Student-$t$ distribution converges for $\nu \to \infty$ 
to the normal distribution, so that in the limit also mean filters are covered.

The multivariate Student $t$-distribution is frequently used in statistics~\cite{KN04}, 
whereas the  multivariate Cauchy distribution is far less common and in contrast 
to the one-dimensional case usually not considered separately from the Student-$t$ distribution. 
The parameter(s) of a multivariate Student $t$-distribution are usually estimated 
via the Maximum Likelihood (ML) method 
in combination with the EM algorithm~\cite{Byrne2017,Ch00,CH08,DLR77,McLK1997}. 
The EM algorithm for the Student-$t$ distribution has been derived, e.g. in~\cite{LLT89}, 
For an overview of estimation methods for the multivariate Student $t$-distribution, 
in particular the EM algorithm and its variants, we refer to~\cite{NK08} and the references therein.

Recently, the Student-$t$ distribution and the closely related Student-$t$ mixture models (SMM) have found interesting applications 
in various image processing tasks. One of the first papers which suggested a variational approach for denoising of images corrupted by Cauchy noise was 
\cite{ALP2002}. In~\cite{LMSS2018}, the authors proposed a unified framework for images corrupted by white noise that can handle (range constrained) Cauchy noise as well. 
Other recent approaches that consider also the task of deblurring include~\cite{DHWMZ2019,YYG2018}. Concerning mixture models, in~\cite{VS14} it has been shown that Student-$t$ mixture models are superior 
to Gaussian mixture models for mode\-ling image patches and the authors proposed an application in image compression. 
Further applications include robust image segmentation~\cite{BM18,NW12,SNG07} as well as robust registration~\cite{GNL09,ZZDZC14}.
In both cases, the SMM  is estimated using the EM algorithm derived in~\cite{PM00}.

In this paper, we propose an application of the Student-$t$ distribution to robust denoising of images corrupted by different kinds of noise. 
The initial motivation for this work were the recent papers~\cite{LPS18,MDHY18,SDZ15} for Cauchy noise removal. 
In~\cite{MDHY18,SDZ15} the authors proposed
a variational method consisting of a data term 
that resembles the noise statistics and a total variation regularization term.
Based on a ML approach the authors of~\cite{LPS18} introduced 
a generalized  myriad filter which estimates both the location and the scale parameter of the Cauchy distribution. 
They used this filter in a nonlocal approach, 
where for each pixel of the image they chose as samples those pixels  having a similar neighborhood 
and replaced the initial pixel by its filtered version. Such a pixelwise treatment assumes the pixels of an image to be independent, 
which is in practice a rather unrealistic assumption; in fact, 
in natural images they are usually locally highly correlated. 
Taking the local dependence structure into account may improve the results of image restoration methods. 
 For instance, in case of denoising images corrupted by additive Gaussian noise  this led to the state-of-the-art algorithm of Lebrun et al.~\cite{LBM13},
who proposed to restore the image  patchwise  based on a maximum a posteriori approach.

In the Gaussian setting, their approach is equivalent to
minimum mean square error estimation, and more general, the resulting estimator can be
seen as a particular instance of a best linear unbiased estimator (BLUE). For denoising
images corrupted by additive Cauchy noise, a similar approach addressed in this paper requires to define a
multivariate myriad filter. In this paper, we derive a generalized multivariate myriad filter
(GMMF) based on ML estimation for the family of Student-t distributions of which the
Cauchy distribution forms a special case.
In the limiting case $\nu=0$ and $d=2$ dimensions this further  provides an algorithm for denoising   images corrupted by heavy-tailed noise
with image values on the circle $\mathbb S^1$,  such as phase-valued images appearing in inferometric synthetic aperture radar InSAR.

After finishing this paper, we became aware that our algorithm for estimating the parameters of the Student-$t$ distribution 
can be considered as Jacobi variant of a sophisticated version of the 
EM algorithm which was heuristically proposed by Kent et al. in \cite{KTV94} and analyzed by van Dyk in \cite{vanDyk1995}.
The approach in the present paper is different and does not require the EM framework with special hidden variables.

The paper is organized as follows: 
In Section~\ref{Sec:Student_t}, we introduce the Student-$t$ distribution and the projected normal distribution.
Their likelihood functions are given in Section~\ref{Sec:Student_t_ML}. 
Then, in Section~\ref{Sec:Student_t_estimators}, we recall existence and uniqueness of the (weighted) ML estimators for the location 
and the scatter parameter of the distribution, where we provide own proofs.
We propose an efficient algorithm for computing the ML estimates 
in Section~\ref{Sec:Algorithm} and 
prove its convergence.
In Section~\ref{Sec:Application}, we illustrate how the developed algorithm can be applied in the context of nonlocal (robust) image denoising
both for gray-value images and images with values in $\mathbb S^1$.
Conclusions  and directions of future research are addressed in Section \ref{sec:conclusions}.

\section{Student-$t$ and Projected Normal Distribution}\label{Sec:Student_t}

In this section, we introduce the multivariate Student-$t$ distribution and the related projected normal distribution
and collect some of their properties. 

\subsection{Multivariate Student-$t$ Distribution}
The probability density function (pdf) of the 
$d$-dimensional Student $t$-distribution $T_\nu(\mu,\Sigma)$ with $\nu>0$ degrees of freedom is given by
\begin{equation}\label{pdf}
f_\nu(x|\mu,\Sigma)  = 
\frac{\Gamma\left(\frac{d+\nu}{2}\right)}{\Gamma\left(\frac{\nu}{2}\right)\, (\pi \nu)^{\frac{d}{2}}
{\abs{\Sigma}}^{\frac{1}{2}}} \, \frac{1}{\left(1+\frac{1}{\nu}\delta \right)^{\frac{d+\nu}{2}}},
\end{equation}
where 
$$
\delta \coloneqq (x-\mu)^\tT \Sigma^{-1}(x-\mu)
$$
denotes the the \emph{Mahalanobis distance} 
between $x$ and a distribution with parameters $\mu,\Sigma$
and
$
\Gamma(s) \coloneqq\int_0^\infty t^{s-1}\e^{-t}\dx[t]
$
the \emph{Gamma function}.
The smaller the value of $\nu$ 
for fixed \emph{location}~$\mu$ and positive definite \emph{scatter matrix} $\Sigma$, 
the heavier are the tails of the $T_\nu(\mu,\Sigma)$ distribution.
Figure~\ref{Fig:different_nu} 
illustrates this behavior for the one-dimensional standard Student-$t$ distribution. 
For $\nu \to \infty$, 
we obtain since  
$$
\frac{
\Gamma(\frac{\nu + d}{2})}{ \Gamma(\frac{\nu}{2}) \left(\frac{\nu}{2}\right) ^{d/2} } \rightarrow 1, \qquad
(1+ \frac{1}{\nu} \delta ) ^{(\nu+d)/2} \rightarrow \frac12 \delta \quad \mathrm {as} \; \nu \rightarrow \infty
$$
see \cite{BC2017}, 
that 
$\lim_{\nu \rightarrow \infty}f_\nu(x|\mu,\Sigma) = {(2\pi)^{-\frac{d }{2}} \abs{\Sigma}^{-\frac{1}{2}} } \e^{-\frac12 \delta}$
so that
the Student-$t$ distribution $T_\nu(\mu,\Sigma)$ converges to the normal distribution $\NN(\mu,\Sigma)$.

\begin{figure}[thb]
\centering  
\centering  
{\includegraphics[width=0.4\textwidth]{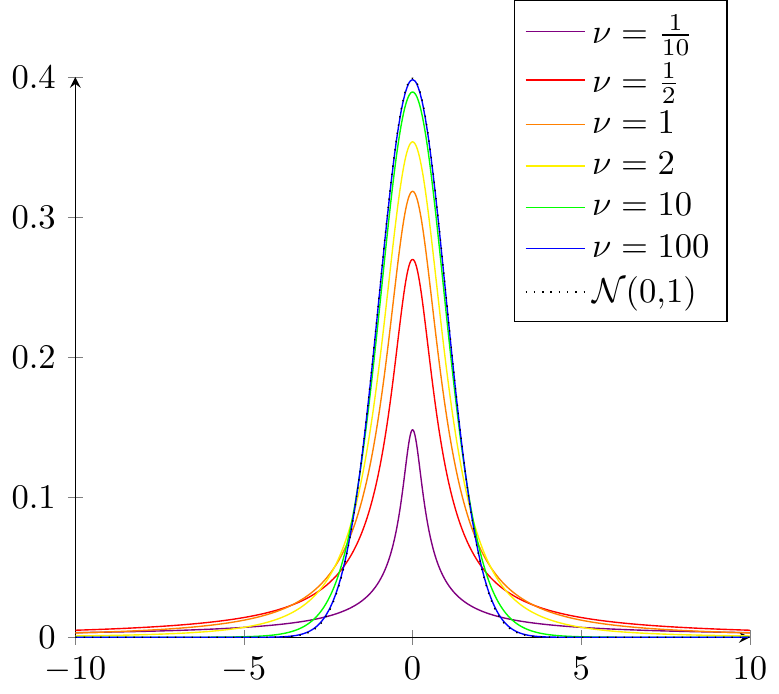}}
\caption{Standard Student-$t$ distribution $T_\nu(0,1)$ 
for different values of $\nu$ in comparison with the standard normal distribution $\NN(0,1)$.}\label{Fig:different_nu}
\end{figure}

The expectation of the Student-$t$ distribution is $\E(X) = \mu$ for $\nu > 1$ 
and the covariance matrix is given by $\Cov(X) =\frac{\nu }{\nu-2} \Sigma$ for $\nu > 2$, otherwise the quantities are undefined. 
As the normal distribution, the Student-$t$ distribution belongs to the class of \emph{elliptical distributions}.  
Some important properties that are needed later on are summarized in the next theorem~\cite{KN04}.
In the following, we denote by $\mathrm{Sym}(d)$ the space of symmetric $d \times d$ matrices and by 
$\SPD(d)$ the cone of symmetric positive definite matrices. 

\begin{Theorem}\label{Prop:Student_t}
\begin{enumerate}
	\item Let $\mu\in \R^d$ and $\Sigma\in \SPD(d)$. Further, let $Z\sim \NN(0,\Sigma)$ and $Y\sim \Gamma\left(\frac{\nu }{2},\frac{\nu }{2}\right)$ be independent, 
	where $\Gamma(\alpha,\beta)$ is the \emph{Gamma distribution} with parameters $\alpha,\beta>0$. 
	Then $X = \mu + \frac{Z}{\sqrt{Y}}~\sim T_\nu(\mu,\Sigma)$.
	\item Let $X\sim T_\nu(\mu,\Sigma)$, $A\in \R^{d\times d}$ be an invertible matrix and $b\in \R^d$. Then	$AX + b\sim T_\nu(A\mu + b, A\Sigma A^\tT)$. 
\end{enumerate}		
\end{Theorem}

\subsection{Projected Normal Distribution}
For the limiting case $\nu \rightarrow 0$, the pdf $f_\nu$ in \eqref{pdf} converges pointwise to zero, i.e.
$$
\lim_{\nu \rightarrow 0} f_\nu(x|\mu,\Sigma)  
= 
\lim_{\nu \rightarrow 0}
\frac{1}{\Gamma \left(\frac{\nu}{2}\right)\, (1+ \frac{1}{\nu} \delta)^{ \frac{\nu}{2} }  } \;
\frac{ \Gamma\left(\frac{d}{2}\right)}{\pi^{\frac{d}{2}} \, \abs{\Sigma}^{\frac{1}{2}} \, \delta^{\frac{d}{2}}}
= 0
$$
However, replacing the first factor by $\frac12$ the surface measure 
$\omega_d =  2 \pi^{\frac{d}{2}}/\Gamma\left(\frac{d}{2}\right)$
of the sphere $\SP^{d-1} \subset \mathbb R^d$ comes into the play
and setting $\mu \coloneqq 0$,
we obtain the pdf 
\begin{equation} \label{pdf_projected_normal}
f_0(x|\Sigma) =
\frac{ \Gamma \left(\frac{d}{2}\right) }{ 2 \pi^\frac{d}{2} } \frac{1}{ \abs{\Sigma}^\frac{1}{2} \, \delta^\frac{d}{2} }
\end{equation}
of the \emph{projected normal distribution} $\Pi_\NN(0,\Sigma)$ on $\SP^{d-1}$ on $\SP^{d-1}$.
In the rest of this paper, we will refer to this setting as case $\nu = 0$.
More precisely, if  $X\sim \NN(0,\Sigma)$, then
$\frac{X}{\lVert X\rVert_2} \sim \Pi_\NN(0,\Sigma)$.
Note that for $X\sim \NN(\mu,\Sigma)$ with $\mu \not = 0$ we have
again $\frac{X-\mu}{\lVert X-\mu\rVert_2} \sim \Pi_\NN(0,\Sigma)$, 
but $\frac{X}{\lVert X\rVert_2} \sim \Pi_\NN(\mu,\Sigma)$
with a more sophisticated pdf, see, e.g., \cite{frahm2004}.
This distribution is also called 
\emph{angular Gaussian distribution}~\cite{HBW2017,MJ09,Watson1983},
\emph{off-set normal distribution}~\cite{Mardia1972} 
or 
\emph{displaced normal distribution}~\cite{Kendall1974}.
It is important to mention that $f_0(x|\Sigma) = f_0(x|\lambda \Sigma)$ for any $\lambda >0$, 
so that the positive definite matrix $\Sigma$ is only identifiable up to a positive factor. 

\paragraph{Wrapped Cauchy Distribution.}
For $d=2$, there is a relation of the projected normal distribution 
to the \emph{wrapped Cauchy distribution}~\cite{KT88,MJ09,WG2013,WG2014}
which we will use for our applications in  Section~\ref{Sec:Application}.
The density of the real-valued  Cauchy distribution $C(a,\gamma)$ is given by
\begin{equation*}
g(\theta|a,\gamma) = \frac{1}{\pi} \frac{\gamma}{\gamma^2 + (\theta-a)^2}, \qquad a\in \R, \; \gamma>0.
\end{equation*}
The wrapped Cauchy distribution $C(a,\gamma)$ is obtained by wrapping it around the circle, i.e.\ for $\theta \in [-\pi,\pi)$ we have
\begin{align*}
g_w(\theta|a,\gamma) 
&=  \sum_{k\in \Z} \frac{1}{\pi}\frac{\gamma}{\gamma^2 + (\theta + 2k\pi - a)^2}\\
& = 
\frac{1}{2\pi}\frac{1 -\rho^2}{1+\rho^2-2\rho\cos(\theta-a)} , \qquad a \in [-\pi,\pi),
\end{align*}
where $\rho = \e^{-\gamma}$ and the second formula follows by Poisson's summation formula
using that $\hat g(\omega) = \e^{-\gamma |\omega| + \imag a \omega}$ is the characteristic function of $g$.
We rewrite the density as 
\begin{align}
g_w(\theta|a,\gamma)
&= \frac{1}{2\pi}\frac{1}{\frac{1+\rho^2}{1 -\rho^2}-\frac{2\rho}{1 -\rho^2}\bigl(\cos(a)\cos(\theta) + \sin(a)\sin(\theta)\bigr)}\\
&= \frac{1}{2\pi} \frac{\sqrt{1-\xi_1^2-\xi_2^2}}{1-\xi_1\cos(\theta) - \xi_2\sin(\theta)}, \label{WC}
\end{align}
where $\xi_1 = \frac{2\rho}{1 +\rho^2}\cos(a)$ and $\xi_2 = \frac{2\rho}{1 +\rho^2}\sin(a)$.

\begin{Lemma}[Relation between Projected Normal and Wrapped Cauchy Distribution\label{Prop:wrapped_Cauchy}]\hfill 
\begin{enumerate}
	\item 	Let $d=2$ and $X \sim \Pi_\NN(0,\Sigma)$ with $ \Sigma = \begin{pmatrix}
	\sigma_{11} & \sigma_{12}\\
	\sigma_{12} & \sigma_{22}
	\end{pmatrix}$		
	be a random variable in $\mathbb{S}^1$ with parameterization 
		$X = \begin{pmatrix}
		\cos(\Phi)\\
		\sin(\Phi)
		\end{pmatrix}$. 		
		Then 
	$\Theta \coloneqq  (2\Phi) \modulo 2\pi \sim C_w(a,\rho)$ 
	with parameters 
	\begin{align}
	\rho &= \e^{-\gamma} = \left(\frac{\tr(\Sigma)-2\sqrt{|\Sigma|}}{\tr(\Sigma)+2\sqrt{|\Sigma|}}\right)^{\frac{1}{2}}, \label{rho}\\
	a &=  \begin{cases}
	-\pi                                                                         &\text{if} \; \sigma_{11} - \sigma_{22} = 0,\\
	\operatorname{arctan} \left(\frac{2 \sigma_{12}}{\sigma_{11} - \sigma_{22}} \right)       &\text{if }   \sigma_{11} - \sigma_{22} > 0,\\
	\operatorname{arctan} \left(\frac{2 \sigma_{12}}{\sigma_{11} - \sigma_{22}}\right) + \pi  &\text{if }   \sigma_{11} - \sigma_{22} < 0   \text{ and }   \sigma_{12} \ge 0,\\
	\operatorname{arctan} \left(\frac{2 \sigma_{12}}{\sigma_{11} - \sigma_{22}}\right) - \pi  &\text{if }   \sigma_{11} - \sigma_{22} < 0   \text{ and }   \sigma_{12} <0.
	\end{cases}
	  \label{a}
	\end{align}
	\item  Let $\Theta \sim C_w(a,\rho)$ and let $\Xi $ be a discrete random variable with $\P(\Xi= -1) =\P(\Xi =1) = \frac{1}{4}$, $\P(\Xi =0) = \frac{1}{2}$ that is independent from $\Theta$.
	Then		$		\Phi = \left(\frac{\Theta}{2}+\pi \Xi \right)\modulo 2\pi~\sim \Pi_\NN(0,\Sigma)	$,
	where (up to a positive factor)
	\begin{equation*}
	\Sigma = \begin{pmatrix}
	\frac{1}{2} + \frac{\rho}{1+\rho^2}\cos(a) & \frac{\rho}{1 + \rho^2} \sin(a)\\
	\frac{\rho}{1 + \rho^2} \sin(a) & \frac{1}{2} - \frac{\rho}{1+\rho^2}\cos(a) 
	\end{pmatrix}.
	\end{equation*}
\end{enumerate} 
\end{Lemma}

The proof of this Lemma can be found in the appendix.

Observe that the wrapped Cauchy distribution is unimodular, 
whereas the projected normal distribution is antipodally symmetric, which  causes the $\modulo$ operation. 
A similar relation as in statement (i) of the lemma can be found, e.g., in \cite{MJ09} without proof.

\section{Weighted Likelihood Functions}\label{Sec:Student_t_ML}
%
In this section, we provide weighted log-likelihood functions for the Student-$t$ and projected normal distributions
together with the equations characterizing their critical points.

\paragraph{Student-$t$ distribution.} Let $\nu > 0$. For $x_i \in \mathbb R^d$, $i=1,\ldots,n$, 
the likelihood function of the the Student-$t$ distribution is given by
\begin{equation*}
\LL(\mu,\Sigma|x_1,\ldots,x_n)
= \frac{\Gamma\left(\frac{d+\nu}{2}\right)^n}{\Gamma\left(\frac{\nu}{2}\right)^n(\pi \nu)^{\frac{nd}{2}}\abs{\Sigma}^{\frac{n}{2}} } 
\prod_{i=1}^n \frac{1}{\bigl(1+\frac{1}{\nu}\delta_i\bigr)^{\frac{d+\nu}{2}}}
\end{equation*}
and the log-likelihood function by
\begin{align}
\ell(\mu,\Sigma|x_1,\ldots,x_n)
=& \, n \, \log\left(\Gamma\left(	\frac{d+\nu}{2}\right)\right) - n \log \left(\Gamma\left(\frac{\nu}{2}\right)\right)-\frac{nd}{2}\log(\pi\nu) \\
&- \frac{n}{2}\log \abs{\Sigma} - \frac{d+\nu}{2} \sum_{i=1}^n \log\left(1+\frac{1}{\nu}\delta_i \right),
\end{align}
where 
$$\delta_i \coloneqq  (x_i-\mu)^\tT \Sigma^{-1} (x_i-\mu).$$
Ignoring constants, maximizing $\ell$ is equivalent to minimizing the negative, weighted function
\begin{align} \label{L}
L(\mu,\Sigma)& \coloneqq (d+\nu) \sum_{i=1}^n w_i \log(\nu + \delta_i )+ \log \abs{\Sigma}
\end{align}
for uniform weights $w_i = \frac{1}{n}$. 
Note that we   allow for different weightings of the summands by introducing weights in the open probability simplex
\begin{equation*}
\mathring \Delta_n \coloneqq  \left\{w = (w_1,\ldots,w_n) \in \mathbb R_{>0}^n: \sum_{i=1}^n w_i = 1 \right\}.
\end{equation*}
Using the relations 
\begin{equation*}
\frac{\partial \log(\abs{X})}{\partial X} = X^{-1},\qquad
\frac{\partial a^\tT X^{-1}b }{\partial X} =- {(X^{-\tT})}a b^\tT {(X^{-\tT})},
\end{equation*}
see~\cite{PP08}, the derivatives of $L$ with respect to $\mu$ and $\Sigma$ are given by
\begin{align*}
\frac{\partial L}{\partial \mu}(\mu,\Sigma) 
& = -2(d+\nu )\sum_{i=1}^n w_i \frac{ \Sigma^{-1}(x_i-\mu)}{\nu + \delta_i},\\
\frac{\partial L}{\partial \Sigma}(\mu,\Sigma)	
& = - (d+\nu ) \sum_{i=1}^n w_i \frac{ \Sigma^{-1}(x_i-\mu)(x_i-\mu)^\tT \Sigma^{-1} }{\nu + \delta_i}+\Sigma^{-1}.
\end{align*}
Setting them to zero results in the equations
\begin{align}
0 &= \sum_{i=1}^n w_i \frac{x_i-\mu}{\nu+\delta_i},\label{mult_cond_a}\\
I &=	(d+\nu)\sum_{i=1}^n w_i \frac{\Sigma^{-\frac{1}{2}}(x_i-\mu)(x_i-\mu)^\tT {\Sigma^{-\frac{1}{2}}} }{\nu+\delta_i} \label{mult_cond_S}
\end{align}
characterizing the \emph{critical points} of $L$.
Computing the trace of both sides of~\eqref{mult_cond_S} and using the linearity and permutation invariance of the trace operator, we obtain
\begin{align}
d& = \tr(I) 
=
(d+\nu)\sum_{i=1}^n w_i \frac{\tr\bigl(\Sigma^{-\frac{1}{2}}(x_i-\mu)(x_i-\mu)^\tT {\Sigma^{-\frac{1}{2}}}\bigr)}{\nu+\delta_i}  
 = (d+\nu)\sum_{i=1}^n w_i \frac{\delta_i}{\nu+\delta_i},
\end{align}
which yields after division by $\nu > 0$ the relation
\begin{equation}
	1= (d+\nu)	\sum_{i=1}^n w_i \frac{1}{\nu+\delta_i}\label{trace_1}.
\end{equation}

\paragraph{Projected normal distribution.}
Let $\nu = 0$ and $d\ge 2$.
Similarly as above, we obtain for $x_i \in \SP^{d-1}$, $i=1,\ldots,n$, 
the negative, weighted likelihood function of the $\Pi_\NN(0,\Sigma)$ distribution as
\begin{equation} \label{tyler}
L_0(\Sigma) \coloneqq d \sum_{i=1}^n w_i \log(\delta_i )+ \log \abs{\Sigma}.
\end{equation}
The critical points of $L_0$ are given by the solution of
\begin{equation} \label{tyler_1}
I =	d \sum_{i=1}^n w_i \frac{\Sigma^{-\frac{1}{2}} x_i x_i^\tT {\Sigma^{-\frac{1}{2}}} }{\delta_i}.
\end{equation}
Note that again $L_0(\lambda \Sigma) = L_0(\Sigma)$, $\lambda > 0$ and if $\Sigma$ fulfills \eqref{tyler_1} then also
$\lambda \Sigma$ does.
From the statistical point of view, \eqref{tyler} is only a likelihood function for samples on  $\SP^{d-1}$.
However, later we will also consider the function for arbitrary nonzero points $x_i \in \mathbb R^d$.

\section{Weighted Maximum Likelihood Estimators}\label{Sec:Student_t_estimators}
%
In this section, we are interested in the minimizers of $L$ and $L_0$.
We prove 
\begin{itemize}
\item for $\nu > 0$ and fixed $\mu$ that $L$ has a unique critical point $\Sigma \in \SPD(d)$ and this point is a minimizer;
\item for $\nu \ge 1$ that $L$ has a unique critical point $(\mu,\Sigma) \in \mathbb R^d \times \SPD(d)$ 
and this point is a minimizer;
\item for $\nu =0$ ($\mu = 0$) that $L_0$  has a unique critical point $\Sigma \in \SPD(d)$ with $\tr \Sigma = 1$ and this point is a minimizer.
 Moreover, all critical points are given by $\lambda \Sigma$, $\lambda > 0$.
\end{itemize}
Under several assumptions, these results are known even in the more general context of $M$-estimator for ellipical distributions.
In~\cite{Mar76}, Maronna established sufficient conditions for the existence and uniqueness of a joint minimizer 
for uniform weights and $M$-estimators whose cost function fulfills certain properties. 
More results can be found in~\cite{KT91}. 
For the projected normal distribution,  the claims were proved under certain assumptions
by Tyler in~\cite{tyler1987a}, see also \cite{duembgen1998,DT2005,tyler1987b}.
For an overview we refer to the survey paper of Dümbgen et al.~\cite{DPS2015}.

In this paper, we incorporate different weights into the log-likelihood function which allows for multiple samples
and is moreover useful in the nonlocal denoising approach in Section~\ref{Sec:Application}.
We give direct existence and uniqueness proofs  in order to make the paper self-contained.

\subsection{Estimation of Scatter}
First, we consider the estimation of  the scatter matrix $\Sigma$ only. We start with the case $\nu>0$, 
where the location parameter~$\mu$ is known. 
If $\mu \neq 0$, 
we might transform the samples to $y_i = x_i -\mu$, $i=1,\ldots,n$, 
so that we can assume w.l.o.g.\ $\mu=0$ and
use the notation $L(\Sigma)=L(0,\Sigma)$.

We make the following assumption on the samples $x_1,\ldots,x_n\in \R^d$ and weights $w \in \mathring \Delta_n$:
\begin{Assumption}\label{Ass:lin_ind} 
\begin{enumerate}
	\item Any subset of $ \le d$ samples  $x_i$, $i \in \{1,\ldots,n\}$ is linearly independent.
	\item $(d-1)w_{\text{max}}<\frac{\nu + d-1}{ \nu +d }$, where $w_{\text{max}} \coloneqq \max\{w_i:i=1,\ldots,n\}$.
\end{enumerate}
\end{Assumption}
The linear independence assumption (i) holds $\lambda^d$-a.s.\ when sampling from a continuous distribution. 
The interpretation behind the constraints is that the mass of the (empirical) distribution determined by $x_1,\ldots,x_n$ is  
not allowed to be concentrated in lower dimensional subspaces, which would cause the resulting distribution to be degenerated.
\begin{Lemma}\label{Lem:bounded_weights}
Let $x_i\in\R^d$, , $i=1,\ldots,n$ and $w \in \mathring \Delta_n$ fulfill Assumption~\ref{Ass:lin_ind}. 
Further, let $V\subset \R^d$ be a linear subspace with $0\leq \dim(V)\leq d-1$ 
and
$\II_V \coloneqq \bigl\{i\in \{1,\ldots,n\}\colon x_i\in V\bigr\}$. Then it holds 
\begin{equation}
	\sum_{i\in \II_V} w_i < \frac{\nu + \dim(V)}{ \nu + d}\label{weights_subspace}
\end{equation}
and $n \ge d$.
\end{Lemma}

\begin{proof}
If $V=\{0\}$, we have $\II_V = \emptyset$, so that the statement holds true.
Next, let $d \ge 2$ and $1 \le k=\dim(V)\leq d-1$. By Assumption \ref{Ass:lin_ind} (i), it holds $|\II_V|\leq k$. 
Using that by assumption $w_{\text{max}}<\frac{\nu + d-1}{(d-1)(\nu +d)}$,   we obtain
\begin{equation*}
	\sum_{i\in \II_V} w_i <k\frac{\nu + d-1}{(d-1)(\nu +d)} = \frac{\frac{k}{d-1}\nu + k}{\nu + d}\leq \frac{\nu + k}{\nu + d}=\frac{\nu + \dim(V)}{\nu+d}.
\end{equation*}
Assume that $n<d$ and let $V= \lspan\{x_1,\ldots,x_n\}$. 
Then $\dim(V)\leq d-1$, and we obtain the contradiction
\begin{equation*}
1 = \sum_{i\in \II_V} w_i \leq \frac{\nu + \dim(V)}{\nu + d}\leq \frac{\nu + d-1}{d + \nu}<1. \qedhere
\end{equation*}
\end{proof}

Next, we show the existence of a minimizer of $L$. 
Since $\SPD(d)$ is an open cone, any minimizer of $L$ is also critical point of $L$.

\begin{Theorem}[Existence of Scatter, $\nu > 0$]\label{Theo:scatter_existence}
Let $x_i\in\R^d$, $i=1,\ldots,n$ and $w \in \mathring \Delta_n$ fulfill Assumption~\ref{Ass:lin_ind}. 
Then it holds 
\begin{equation*}
	\argmin_{\Sigma\in \SPD(d)} L(\Sigma)\neq \emptyset.
\end{equation*}
\end{Theorem}

\begin{proof}
Let $\{\Sigma_r\}_{r\in \N}\subseteq \SPD(d)$ be a sequence in $\SPD(d)$, where  
$\lambda_{1r}\geq\ldots\geq \lambda_{dr}>0$ 
denote the eigenvalues of $\Sigma_r$ and  $e_{1r},\ldots,e_{dr}$ the  corresponding orthonormal eigenvectors.
We prove that $L(\Sigma_r)\overset{r\to \infty}{\to} +\infty$ if one of the following situations is met: 
\begin{enumerate}
	\item[(i)] $\lambda_{1r}\overset{r\to \infty}{\to} +\infty$ and $\lambda_{dr}\geq c>0$ for all $r\in \N$,
	\item[(ii)] $\lambda_{dr}\overset{r\to \infty}{\to} 0$.
\end{enumerate}
Then $L$ attains its minimum in $\SPD(d)$, and since $L$ is continuously differentiable, it is necessarily a critical point. 
By definition of $L$ we have
\begin{align*}
L(\Sigma_r)  = (d + \nu)\sum_{i=1}^{n} w_i \log\left(\nu +  x_i^\tT \Sigma_r^{-1} x_i\right) +  \sum_{j=1}^d \log(\lambda_{rj}).
\end{align*}
In case (i), the first sum is bounded below, while the second one tends to infinity, so that $L(\Sigma_r)\overset{r\to \infty}{\to} +\infty$.

In case (ii), let $0\leq p\leq d-1$ such that 
$\lambda_{1r}\geq \ldots\geq \lambda_{rp} \geq c>0$ for all $r\in \N$ 
and 
$\lambda_{r p+1}\geq \ldots\geq \lambda_{dr}\overset{r\to \infty}{\rightarrow}0$. If $p=0$, then all $\lambda_{rj}$ tend to zero. 
Since the sphere $\SP^{d-1}$ is compact, 
there exist  subsequences (w.l.o.g.\ again denoted by $e_{rj}$) 
such that 
$\lim\limits_{r\to \infty} e_{rj}=e_j \in \SP^{d-1}$ for $j=1,\ldots,d$. 
Set $S_0 = \{0\}$, and for $k=1,\ldots,d$ further
$S_k \coloneqq \lspan\{e_1,\ldots,e_k\}$ and
\begin{equation*}
W_k \coloneqq S_k\setminus S_{k-1} = \bigl\{y\in \R^d\colon \langle y,e_k\rangle \neq 0,\, \langle y,e_l\rangle = 0 \text{ for }l=k+1,\ldots,d\bigr\}.
\end{equation*}
By Assumption~\ref{Ass:lin_ind}(i) we know that $\dim(S_k) = k$.
Let 
\begin{align*}
\widetilde{I}_k \coloneqq \bigl\{i\in \{1,\ldots,n\}\colon x_i \in S_k\bigr\}\qquad\text{and}\qquad	I_k \coloneqq \bigl\{i\in \{1,\ldots,n\}\colon x_i \in W_k\bigr\}.
\end{align*}
Using $S_k = W_k \dot\cup S_{k-1} $, we obtain $\tilde{I}_k = I_k \dot \cup \tilde{I}_{k-1}$ for $k=1,\ldots,d$. 
According to  Assumption~\ref{Ass:lin_ind}(i) it holds $|I_k|\leq |\tilde{I}_k|\leq \dim(S_k)= k$ for $k=1,\ldots,d-1$.
Introducing the functions
\begin{equation*}
L_{j}(\Sigma_r) = (d+\nu)\sum_{i\in {I}_j} w_i \log\left(\nu + x_i^\tT \Sigma^{-1}_r x_i\right) + \log(\lambda_{rj}),
\end{equation*}
we can split
\begin{align*}
L(\Sigma_r) = \sum_{j=1}^p L_j(\Sigma_r)+\sum_{j=p+1}^d L_j(\Sigma_r).
\end{align*}
By assumption the first sum remains bounded from below as $r \rightarrow \infty$ so that we have to consider the second one.
We have
\begin{align}
&\sum_{j=p+1}^d L_{j}(\Sigma_r) 
= \sum_{j=p+1}^d (d+\nu)\sum_{i\in {I}_j} w_i \log\left(\nu +x_i^\tT \Sigma^{-1}_r x_i\right) + \log(\lambda_{rj})\\
&=	 
\sum_{j=p+1}^d \biggl( (d+\nu)\sum_{i\in {I}_j} w_i \log\big(\lambda_{rj} (\nu + x_i^\tT \Sigma^{-1}_r x_i) \big) 
+ \big(1-(d+\nu)\sum_{i\in {I}_j}w_i\big) \log(\lambda_{rj})\Big). \label{doch}
\end{align}
Since
\begin{equation*}
y^\tT \Sigma^{-1}_r y = \sum_{j=1}^d \frac{1}{\lambda_{rj}} \langle y,e_{rj}\rangle^2   \geq \frac{1}{\lambda_{rk}}\langle y,e_{rk}\rangle^2,
\end{equation*}
and 
$\langle y,e_{rk}\rangle \overset{r\to \infty}{\rightarrow} \langle y,e_k\rangle \neq 0
$ for $y\in W_k$,
we obtain for all $r$ sufficiently large
\begin{equation}
\liminf_{r\to\infty} \lambda_{rk} \, y^\tT \Sigma^{-1}_r y \geq \langle y,e_{k}\rangle ^2 >0.\label{bounded_below_ratio}
\end{equation}
This implies that the first summand in \eqref{doch} is bounded from below. 
Concerning the second summand in \eqref{doch} we prove for sufficiently large $r$ by induction for
$k \ge p+1$ that

\begin{equation}
\sum_{j=k}^d \biggl( 1- (d+\nu)\sum_{i\in {I}_j}w_i\biggr) \log(\lambda_{rj})
\geq 
\biggl((d+\nu)\sum_{i\in \tilde{I}_{k-1}} w_i - ( \nu + k-1)\biggr) \log(\lambda_{rk}).\label{lower_bound_lambda_inv}
\end{equation}

Then, for $k=p+1$, we conclude by Lemma~\ref{Lem:bounded_weights} that the factor of $\log(\lambda_{rp+1})$ on the right-hand side
is negative, so that the whole sum tends to $+\infty$ as $r \rightarrow \infty$.

Based on the relation $\tilde{I}_k = I_k\cup \tilde{I}_{k-1}$ we write
\begin{equation}
\sum_{i\in \tilde{I}_k} w_i -\sum_{i\in {I}_k} w_i  =	\sum_{i\in \tilde{I}_{k-1}} w_i.\label{weight_relation}
\end{equation}
For the induction basis $k=d$, since $\lambda_{dr} < 1$ for sufficiently large $r$, we have  to show
\begin{equation*}
1- (d+\nu)\sum_{i\in {I}_d}w_i  \leq  (d+\nu)\sum_{i \in \tilde{I}_{d-1}}w_i - (\nu + d-1).
\end{equation*}
This follows directly from~\eqref{weight_relation}, since
\begin{align*}
	1- (d+\nu)\sum_{i\in {I}_d}w_i 
	& = 1- (d+\nu)\biggl( \sum_{i\in \tilde{I}_{d}} w_i- \sum_{i\in \tilde{I}_{d-1}} w_i\biggr) 
 = 1- (d+\nu)\biggl( 1-\sum_{i\in \tilde{I}_{d-1}} w_i  \biggr)\\
	& =   (d+\nu)\sum_{i \in \tilde{I}_{d-1}}w_i - (\nu + d-1).
\end{align*}
Now, assume that~\eqref{lower_bound_lambda_inv} holds for some $k+1$ with $d\geq k+1>p+1$, i.e.,
\begin{equation*}
\sum_{j={k+1}}^d \biggl(1 -  (d+\nu)\sum_{i\in {I}_j} w_i\biggr)\log(\lambda_{rj} )
\ge 
(d+\nu) \biggl(  \sum_{i\in \tilde{I}_{k}} w_i - \frac{ \nu + k}{d+\nu}\biggr)\log(\lambda_{r{k+1}}).
\end{equation*}
Then we get
\begin{align*}
	&\sum_{j={k}}^d \biggl(1 - (d+\nu)\sum_{i\in {I}_j} w_i\biggr) \log(\lambda_{rj})\\
	&=\sum_{j={k+1}}^d \biggl(1- (d+\nu)\sum_{i\in {I}_j}w_i \biggr) \log(\lambda_{rj}) 
                + \biggl(1- (d+\nu)\sum_{i\in {I}_k}w_i\biggr) \log(\lambda_{rk})\\
	& \geq 
(d+\nu) \biggl(  \sum_{i\in \tilde{I}_{k}} w_i - \frac{ \nu + k}{d+\nu}\biggr) \log(\lambda_{r{k+1}}) 
+ \biggl(1- (d+\nu)\sum_{i\in {I}_k}w_i \biggr) \log(\lambda_{r k})
\end{align*}		                  
and by Lemma~\ref{Lem:bounded_weights} and since $\lambda_{r k+1} \leq \lambda_{rk} < 1$ finally	
\begin{align*}
&\geq 
(d+\nu) \biggl(  \sum_{i\in \tilde{I}_{k}} w_i - \frac{ \nu + k}{d+\nu}\biggr) \log(\lambda_{r{k}}) 
+ \biggl(1- (d+\nu)\sum_{i\in {I}_k}w_i \biggr) \log(\lambda_{rk})\\
& = \biggl( (d+\nu) \sum_{i\in \tilde{I}_{k-1}} w_i - ( \nu + k-1)\biggr) \log(\lambda_{rk}).
\end{align*}
This finishes the proof.
\end{proof}

The end of the proof of Theorem~\ref{Theo:scatter_existence} reveals that the condition on the weights stated in~\eqref{weights_subspace} 
is sufficient for the existence of a minimizer. 
The next lemma shows that the non strong inequality is necessary for the existence of a  critical point.

\begin{Lemma}\label{Lem:Ass_lin_neccessary}
Let $x_i\in\R^d$,  $i=1,\ldots,n$ fulfill Assumption~\ref{Ass:lin_ind} (i) and  assume there exists a critical point of $L$. 
Then, for all linear subspaces $V\subset \R^d$ with $0\leq \dim(V)\leq d-1$ it holds
\begin{equation*}
\sum_{i\in \II_V} w_i \leq  \frac{\nu + \dim(V)}{d + \nu},
\end{equation*}
where $\II_V = \bigl\{i\in \{1,\ldots,n\}\colon x_i\in V\bigr\}$.
\end{Lemma}

\begin{proof}
Condition \eqref{mult_cond_S} with $\mu = 0$ can be alternatively written as
\begin{equation}\label{7a}
I =	(d+\nu)\sum_{i=1}^n w_i \frac{ R^{-1} x_i \, x_i^\tT R^{-\tT} }{\nu+\delta_i} 
\end{equation}
with the Cholesky decomposition $\Sigma = R R^\tT$.
W.l.o.g.\ we might assume that $\Sigma = I$ is the critical point, 
otherwise we can transform the samples to $y_i = R^{-1}x_i$, where $RR^\tT = \Sigma$.	
The idea of the proof is to project the samples onto the orthogonal complement of $V$.
More precisely, let $k=\dim(V) < d$ and choose an orthonormal basis $v_1,\ldots,v_k$ of $V$. 
Set 
$W = (v_1,\ldots,v_k)$ so that $P = WW^\tT$ is the orthogonal projection onto $V$.
Now, for $\Sigma = I$ and $\mu = 0$, equation~\eqref{mult_cond_S} reads as
\begin{equation*}
I = (d+\nu)\sum_{i=1}^n w_i \frac{x_i x_i^\tT}{\nu+x_i^\tT x_i}.
\end{equation*}
Multiplying both sides with $I-P$ and taking the trace, yields
\begin{equation*}
	d - k = (d+\nu)\sum_{i=1}^n w_i \frac{ x_i^\tT(I-P)x_i}{\nu+x_i^\tT x_i}. 
\end{equation*}
We split the sum into a sum over $i\in \II_V$ and $i\in \II_V^c .= \{1,\ldots,n\}\setminus \II_V$ and get
\begin{align*}
d - k &= (d+\nu)\sum_{i\in \II_V}w_i \frac{ x_i^\tT(I-P)x_i}{\nu+x_i^\tT x_i}+(d+\nu)\sum_{i\in \II_V^c} w_i\frac{ x_i^\tT(I-P)x_i}{\nu+x_i^\tT x_i}\\
&=(d+\nu)\sum_{i\in \II_V^c} w_i\frac{ x_i^\tT(I-P)x_i}{\nu+x_i^\tT x_i}. 
\end{align*}
With $x_i^\tT(I-P)x_i\leq x_i^\tT x_i$ we obtain further
\begin{align*}
d - k 
&=(d+\nu)\sum_{i\in \II_V^c} w_i\frac{ x_i^\tT(I-P)x_i}{\nu+x_i^\tT x_i}
\leq (d+\nu)\sum_{i\in \II_V^c} w_i\frac{ x_i^\tT x_i}{\nu+x_i^\tT x_i}
\leq  (d+\nu)\sum_{i\in \II_V^c} w_i\\	
& = d + \nu -(d+\nu) \sum_{i\in \II_V} w_i.
\end{align*}
Rearranging yields
\begin{equation*}
	(d+\nu) \sum_{i\in \II_V} w_i\leq \nu + k
\end{equation*}
and finally
\begin{equation*}
\sum_{i\in \II_V} w_i \leq  \frac{\nu + k}{d + \nu}=\frac{\nu + \dim(V)}{d + \nu}.
\end{equation*}\qedhere
\end{proof}

Now, we turn to the question whether $L$ has a unique critical point. 
If this is the case, then clearly the unique critical point of $L$ is the minimizer of $L$.

\begin{Theorem}[Uniqueness of Scatter, $\nu > 0$]\label{Theo:scatter_uniqueII}
Let $x_i\in\R^d$, $i=1,\ldots,n$ and $w \in \mathring \Delta_n$ fulfill Assumption~\ref{Ass:lin_ind}. 
Then $L$ has a unique critical point.
\end{Theorem}

\begin{proof}
Let $\Sigma_1$ and $\Sigma_2$ fulfill~\eqref{7a}.  
W.l.o.g. we can assume that $\Sigma_1=I$, since otherwise we  can use the Cholesky factorization 
$\Sigma_1 = R_1 R_1^\tT$,  transform the data to $y_i = R_1^{-1}x_i$ and replace $\Sigma_2$ by $R_1^{-1} \Sigma_2 R_1^{-\tT}$.

Let $\lambda_1 \ge \ldots \ge \lambda_d$ denote the eigenvalues of $\Sigma_2$.
We show that $\lambda_1 \le 1$ and $\lambda_d \ge 1$ which implies $I = \Sigma_2$.
Assume that $\lambda_1 > 1$.
By~\eqref{7a} it holds
\begin{align}
\Sigma_2 & = (d+\nu ) \sum_{i=1}^n w_i \frac{x_i x_i^\tT}{\nu +  x_i^\tT \Sigma_2^{-1}x_i}.\label{cond_Sigma_short}
\end{align}
Using the Courant-Fisher min-max principle we have
\begin{equation*}
x_i^\tT \Sigma_2^{-1}x_i \geq	{\lambda_1^{-1}}x_i^\tT x_i,
\end{equation*}
and with $\lambda_1>1$ we can  estimate
\begin{align*}
\frac{d+\nu}{\nu + x_i^\tT \Sigma_2^{-1}x_i} \leq \frac{d+\nu}{\nu + \lambda_1^{-1}x_i^\tT x_i}	 
= 
\lambda_1\frac{d+\nu}{\lambda_1 \nu +x_i^\tT x_i} <\lambda_1\frac{d+\nu}{{\nu} +x_i^\tT x_i} , \qquad i=1,\ldots,n.
\end{align*}	 
Inserting this in~\eqref{cond_Sigma_short} and regarding Assumption~\ref{Ass:lin_ind}(i) yields the contradiction
\begin{equation*}
\Sigma_2   = (d+\nu ) \sum_{i=1}^n w_i \frac{x_i x_i^\tT}{\nu +  x_i^\tT \Sigma_2^{-1}x_i} 
\prec \lambda_1 (d+\nu) \sum_{i=1}^n w_i \frac{x_i x_i^\tT}{\nu +  x_i^\tT  x_i}= \lambda_1 I,
\end{equation*}
where $A \prec B$ means that $B-A$ is positive definite.
Similarly we can show that $\lambda_d\geq 1$ and we are done.
\end{proof}

Now we turn to the case $\nu = 0$, i.e. to the projected normal distribution and the function $L_0$ in \eqref{tyler}.

\begin{Theorem}[Existence of Scatter, $\nu = 0$]\label{Theo:scatter_unique_0}
Let $d \ge 2$ and $x_i\in \mathbb R^d$, $i=1,\ldots,n$ fulfill Assumption~\ref{Ass:lin_ind} (i).
Let $w \in \mathring \Delta_n$ satisfy $w_{\max} < 1/d$.
Then, it holds 
\begin{equation*}
\argmin_{\Sigma\in \SPD(d)} L_0(\Sigma)\neq \emptyset.
\end{equation*}
\end{Theorem}

Note that for $x_i \in \SP^{d-1}$, $i=1,\ldots,n$, the Assumption~\ref{Ass:lin_ind} i) is fulfilled 
if the points are pairwise distinct and not antipodal.
Further, we see that $w \in \mathring \Delta_n$ and $w_{\max} < 1/d$ imply $n > d$.

\begin{proof}
By definition of $L_0$ we can restrict  the domain of $L_0$ to the bounded set
\begin{equation*}
D \coloneqq \left\{ \Sigma \in \SPD(d)\colon \tr(\Sigma) = d\right\}\subset \Sym(d).
\end{equation*}	
Let $\{\Sigma_r\}_{r\in \N}\subseteq D$ be a sequence with $\lim\limits_{r\to \infty}\Sigma_r = \Sigma\in \del D$. 
We will show that $L_0(\Sigma_r)\overset{r\to \infty}{\rightarrow}\infty$. 
Let  $\lambda_{1r}\geq\ldots\geq \lambda_{dr}>0$ denote the eigenvalues of $\Sigma_r$. 
Further, let   $\lambda_1\geq \ldots\geq \lambda_p>0$, $\lambda_{p+1}=\ldots=\lambda_d = 0$ be the eigenvalues of $\Sigma$, 
where $p = \rang(\Sigma)<d$. 	Note that  $p\geq 1$ since $\tr(\Sigma) = d$. 
Now, the same argumentation as in the proof of case (ii) of Theorem~\ref{Theo:scatter_existence} yields
\begin{align*}
L_0(\Sigma_r)& =\sum_{j=1}^d L_j(\Sigma_r) = \sum_{j=1}^p L_j(\Sigma_r)+\sum_{j=p+1}^d L_j(\Sigma_r)\\
& \geq 
\sum_{j=1}^p L_j(\Sigma_r) + d \sum_{j=p+1}^d \sum_{i \in I_j }w_i \log(\lambda_{rj} x_i^\tT \Sigma_r^{-1} x_i)
+ \biggl(d\sum_{i\in \tilde{I}_{p}} w_i - p \biggr) \log(\lambda_{r p+1}^{-1}),
\end{align*}
which tends to $+\infty$ for $r\to \infty$.
Consequently, since  $\{\Sigma_r\}_{r\in \N}$ is bounded, it contains a convergent subsequence, 
whose limit is by the above argumentation an inner point.
\end{proof}

The next lemma shows, that the solution of~\eqref{tyler_1} is no longer unique. 
However,  there exists a unique solution $\Sigma$ with trace~$d$, 
and all other solutions are of the form $\lambda \Sigma$ with $\lambda>0$.

\begin{Theorem}[Uniqueness of Scatter up to a factor, $\nu = 0$]\label{Prop:scatter_nonunique}
Let $d \ge 2$ and $x_i\in \mathbb R^d$, $i=1,\ldots,n$ fulfill Assumption~\ref{Ass:lin_ind} (i).
Let $w \in \mathring \Delta_n$ satisfy $w_{\max} < 1/d$.
Then, it holds 
\begin{equation}
\sum_{i=1}^{n}w_i \frac{\Sigma^{-\frac{1}{2}} x_i x_i^\tT \Sigma^{-\frac{1}{2}}}{x_i^\tT \Sigma^{-1}x_i} 
=\sum_{i=1}^{n}w_i \frac{S^{-\frac{1}{2}} x_i x_i^\tT S^{-\frac{1}{2}}}{x_i^\tT \Sigma^{-1}x_i}\label{equality_up_to_factor}
\end{equation}
if and only if $S = \lambda \Sigma$ for some $\lambda >0$. 
The critical points of $L_0$ are unique up to a positive factor.
\end{Theorem}

\begin{proof}
Clearly, if $S = \lambda \Sigma$ then~\eqref{equality_up_to_factor} holds true. \\
To show the reverse direction, we may as in the proof of Theorem~\ref{Theo:scatter_uniqueII} w.l.o.g.\ assume that $\Sigma=I$. 
Let $\lambda_1$ denote the largest eigenvalue of $S^{-1}$ and assume that it has multiplicity $k < d$. 
Let $e_1,\ldots,e_k$ be the corresponding orthonormal eigenvectors. 
Further, let
$
P = \sum\limits_{i=1}^k e_i e_i^\tT 
$
be the associated orthogonal projector onto $\lspan\{e_1,\ldots,e_k\}$. 
Equality~\eqref{equality_up_to_factor}  reads for $\Sigma=I$ as
\begin{equation*}
\sum_{i=1}^n w_i \frac{x_i x_i^\tT}{x_i^\tT x_i} = 	\sum_{i=1}^n w_i \frac{S^{-\frac{1}{2}}x_i x_i^\tT S^{-\frac{1}{2}}}{x_i^\tT S^{-1} x_i}.
\end{equation*}
Multiplying both sides with $P$ and taking the trace gives
\begin{equation}
\sum_{i=1}^n w_i \frac{x_i^\tT P x_i }{x_i^\tT x_i} =\lambda_1	\sum_{i=1}^n w_i \frac{x_i^\tT P x_i }{x_i^\tT S^{-1} x_i}.\label{proj_trace}
\end{equation}
Since $\frac{1}{\lambda_1} x_i^\tT S^{-1} x_i \leq x_i^\tT x_i$ with equality if and only if $Px_i = x_i$, 
this  implies $Px_i = 0$ or $Px_i = x_i$ for all $i=1,\ldots,n$. 
This contradicts the assumptions on the points $x_i$ and $w_{\max}$ unless $P=I$. Thus, $S = \lambda_1 I$. 
\end{proof}

\subsection{Estimation of Location and Scatter}
In order to show the existence and  uniqueness of a joint minimizer of the likelihood function $L(\mu,\Sigma)$, 
we use the established technique to consider the location 
and scatter estimation problem as a higher dimensional centered  scatter only problem.
We need the following auxiliary lemma, see, e.g. \cite{DPS2015}.

\begin{Lemma}\label{Lem:higher_dim}
Let $\lambda > 0$, $\mu \in \mathbb R^d$ and $\Sigma \in \SPD(d)$.
Then it holds
\begin{equation} \label{form}
A_\lambda = \lambda 
\begin{pmatrix}
	\Sigma + \mu \mu^\tT  &  \mu  \\
	\mu^\tT  & 1
\end{pmatrix} \in \SPD(d+1).
\end{equation}
Conversely, every $A \in \SPD(d+1)$ can be written in the form \eqref{form} with uniquely determined $\lambda > 0$, 
$\mu \in \mathbb R^d$ and $\Sigma \in \SPD(d)$.
\end{Lemma}

The inverse of the matrix $A_\lambda$ in \eqref{form} is given by
\begin{equation*}
A_\lambda^{-1} = \frac{1}{\lambda}\begin{pmatrix}
	\Sigma^{-1} & -\Sigma^{-1}\mu   \\
	-\mu^\tT \Sigma^{-1} & 1 + \mu^\tT \Sigma^{-1} \mu
\end{pmatrix}.
\end{equation*}
It fulfills
\begin{equation} \label{quadratic_form}
(x^\tT \, 1) \, A_\lambda^{-1} \begin{pmatrix} x \\1 \end{pmatrix} =  \frac{1}{\lambda}\bigl(1 + (x-\mu)^\tT \Sigma^{-1} (x-\mu)\bigr)
\end{equation}
and
$|A| = \lambda^{d+1} |\Sigma|$. 
Using the last two relations  
and setting 
\begin{equation} \label{form_1}
A \coloneqq
\begin{pmatrix}
	\Sigma + \mu \mu^\tT  &  \mu  \\
	\mu^\tT  & 1
\end{pmatrix},\quad
z_i \coloneqq \begin{pmatrix} x_i \\1 \end{pmatrix}, \quad i=1,\ldots,n,
\end{equation}
we can rewrite the function $L(\mu,\Sigma)$ with $\nu > 1$ in \eqref{L} as
\begin{align}
L(\mu,\Sigma) 
&= (d+ \nu)\sum_{i=1}^n w_i \log\bigl(\nu-1 + z_i^\tT A^{-1}z_i\bigr)+ \log \abs{A}\\
& = (\tilde{d} + \tilde{\nu}) \sum_{i=1}^n w_i \log\bigl(\tilde{\nu} + z_i^\tT A^{-1}z_i\bigr)+ \log \abs{A } 
\eqqcolon  \tilde L(A), \label{gleich}
\end{align}
where $\tilde d  \coloneqq d+1$, $\tilde \nu  \coloneqq \nu -1$ and the tilde on top of $L$ 
shows that the function is considered in dimension $d+1$ now.

The points $x_i \in \mathbb R^d$, $i=1,\ldots,k$ are called \emph{affinely independent} if
$\sum\limits_{i=1}^k \lambda_i x_i = 0$ and $\sum\limits_{i=1}^k \lambda_i = 0$ implies 
$\lambda_1 = \ldots = \lambda_k = 0$. It can be immediately seen that the points  $x_i \in \mathbb R^d$, $i=1,\ldots,k$
are affinely independent if and only if the points  
$z_i  =( x_i^\tT \, 1)^\tT \in \mathbb R^{d+1}$, $i=1,\ldots,k$
are linearly independent.

Let $\nu >1$.
By Theorem~\ref{Theo:scatter_uniqueII}
we know that $\tilde L(A)$ has a unique minimizer in $\SPD(d+1)$ if the following
conditions according to Assumption \ref{Ass:lin_ind} are fulfilled:

\begin{Assumption}\label{Ass:affin_ind}
\begin{enumerate}
	\item Any subset of $\le d+1$  samples $x_i$, $i \in \{1,\ldots,n\}$ is affinely independent.
	\item $d w_{\text{max}}< \frac{\nu + d -1 }{\nu +d}$.
\end{enumerate}
\end{Assumption}

Note that $w \in \mathring \Delta_n$ and (ii) imply $n \ge d+1$ for $\nu > 1$ 
and $n > d+1$ for $\nu = 1$.

By the next lemma, see, e.g. \cite{DPS2015},
we can restrict our attention to matrices of the form~\eqref{form_1} 
when minimizing $\tilde L$ over $\SPD(d+1)$.

\begin{Lemma} \label{lem:gut}
Let $\nu>1$ and $x_i\in\R^d$, $i=1,\ldots,n$ and $w \in \mathring \Delta_n$ fulfill Assumption~\ref{Ass:affin_ind}. 
For $z_i \coloneqq ( x_i^\tT \, 1)^\tT $, $i=1,\ldots,n$, let the function  $\tilde L$ be defined by \eqref{gleich}.
Then the critical point of $\tilde L$ has the $(d+1,d+1)$-th entry 1.
\end{Lemma}

Now the existence and uniqueness of a minimizer of $L$ follows immediately from Lemmata~\ref{Lem:higher_dim} and \ref{lem:gut}
and  \eqref{gleich}. 

\begin{Theorem}[Existence and Uniqueness of Location and Scatter, $\nu > 1$] \label{thm:main}
Let $x_i\in\R^d$, $i=1,\ldots,n$ and $w \in \mathring \Delta_n$ fulfill Assumption~\ref{Ass:affin_ind}. 
Then $L$ has a unique critical point and this point is a minimizer of $L$.
\end{Theorem} 

\begin{proof}
Let  $A$ be the critical point/minimizer of $\tilde L$. 
By Lemma \ref{lem:gut} it has the form \eqref{form_1}
and by \eqref{gleich}, we see that the corresponding $(\mu,\Sigma)$ is the minimizer of $L$ and thus a critical point of $L$.

Conversely, let $(\mu,\Sigma)$ be a critical point of $L$, i.e. fulfill
\eqref{mult_cond_a} and \eqref{mult_cond_S}. 
Define $A$ by \eqref{form_1}. We have to show that $A$ fulfills the critical point condition
\begin{equation} \label{hhelp_2}
A = (\tilde d + \tilde \nu) \sum_{i=1}^n w_i \frac{z_i z_i^\tT}{\nu - 1 + z_i^\tT A^{-1} z_i},
\end{equation}
where $z_i \coloneqq ( x_i^\tT \, 1)^\tT$, $i=1,\ldots,n$.
Since $\tilde L$ has only one critical point, this would imply the assertion.
By \eqref{quadratic_form} and definition of $z_i$, this can be written as
\begin{align} 
A = (d+\nu) \sum_{i=1}^n w_i \frac{\begin{pmatrix} x_i x_i^\tT & x_i\\ x_i^\tT & 1\end{pmatrix}}{\nu  + \delta_i}.
\end{align}
Considering the different blocks of this matrix, we obtain with \eqref{mult_cond_a}, \eqref{mult_cond_S} and \eqref{trace_1} that
\begin{align}
(d+\nu) \sum_{i=1}^n w_i \frac{x_i x_i^\tT}{\nu  + \delta_i}
&= 
\Sigma + (d+\nu) \sum_{i=1}^n w_i \frac{x_i \mu^\tT + \mu x_i^\tT - \mu \mu^tT}{\nu  + \delta_i}
=
\Sigma + \mu \mu^\tT,
\\
(d+\nu) \sum_{i=1}^n w_i \frac{x_i }{\nu  + \delta_i}
&= \mu, \quad \mathrm{and} \quad
(d+\nu) \sum_{i=1}^n w_i \frac{1 }{\nu  + \delta_i} = 1.
\end{align}
Thus, \eqref{hhelp_2} holds indeed true.
\end{proof}

	
The multivariate Cauchy distribution, i.e. the case $\nu=1$ requires a special consideration. 

\begin{Theorem}[Existence and Uniqueness of Location and Scatter, $\nu = 1$] \label{thm:main_0}
Let $x_i\in\R^d$, $i=1,\ldots,n$ and $w \in \mathring \Delta_n$ fulfill Assumption~\ref{Ass:affin_ind}. 
Then $L$ has a unique critical point and this point is a minimizer of $L$.
\end{Theorem}

\begin{proof}
As in \eqref{gleich} we see also for $\nu = 1$ with $A$ and $(\mu,\Sigma)$ related by \eqref{form_1} that
$L(\mu,\Sigma) =   \tilde L_0(A)$,
where $\tilde L_0$ resembles $L_0$ with respect to dimension $d+1$.
Let $A$ be the unique critical point of $\tilde L_0$ with $a_{d+1,d+1} = 1$. It is also a minimizer.
Then, by the above relation, $(\mu,\Sigma)$ obtained from $A$ by \eqref{form_1} is a minimizer of $L$
and also a critical point.
Conversely, let $(\mu,\Sigma)$ be any critical point of $L$.
Then we can show as in the proof of Theorem~\ref{thm:main} that the related $A$ in \eqref{form_1} is the unique critical point
of $\tilde L_0$ with $a_{d+1,d+1} = 1$. This finishes the proof.
\end{proof}

\section{Efficient Minimization Algorithm}\label{Sec:Algorithm}
{ In this section, we propose an efficient algorithm  to compute the ML estimates of the multivariate Student-$t$ distribution
and prove its convergence.
After finishing this paper we became aware that a Gauss-Seidel variant of our algorithm was already suggested 
by Kent et al.~\cite{KTV94} to speed up the classical EM algorithm without any convergence considerations.
In \cite{vanDyk1995}, see also \cite{MVD97}, 
van Dyk gave an interpretation of this algorithm from an EM theoretical point of view using a sophisticated choice of the hidden variables,
such that convergence is ensured by general results for the EM algorithm.
However, our derivation of the Jacobi variant of the algorithm and its convergence proof do not rely on an EM setting.}

We assume that $\nu \geq 0$ in case of estimating only $\Sigma$ and $\nu \geq 1$ when estimating both $\mu$  and $\Sigma$.
Conditions~\eqref{mult_cond_a} and~\eqref{mult_cond_S} can be reformulated as fixed-point equations
\begin{align}
\mu 
&= \frac{ \sum\limits_{i=1}^{n} w_i \frac{ 1}{\nu +   \delta_i}x_i}{ \sum\limits_{i=1}^{n} w_i  \frac{1}{\nu + \delta_i }}
\label{fp_mu},\\
\Sigma 
&=(d+\nu)
\sum\limits_{i=1}^{n} w_i \frac{ (x_i-\mu)(x_i-\mu)^\tT }{\nu + \delta_i}\label{fp_Sigma}.
\end{align}
Based on this, we propose the following Algorithm \ref{alg:myriad_mult_general} which is 
a Picard iteration for $\mu$
and a \emph{scaled  version} of the Picard iteration for $\Sigma$. 
For $d=1$, Algorithm \ref{alg:myriad_mult_general} coincides with the generalized myriad filtering considered in~\cite{LPS18}.

\begin{algorithm}[!ht]
\caption{Generalized Multivariate Myriad Filter (GMMF)} \label{alg:myriad_mult_general}
\begin{algorithmic}
	\State \textbf{Input:} $x_1,\ldots,x_n\in \R^d$, $w \in \mathring \Delta_n$ 
	\State \textbf{Initialization:} $\mu_0 =\frac{1}{n}\sum\limits_{i=1}^n x_i$, $\Sigma_0 =\frac{1}{n}\sum\limits_{i=1}^n (x_i-\mu_0)(x_i-\mu_0)^\tT$
	\For{$r=0,\ldots$}
	\begin{align*} 	
	\delta_{i,r} &=  (x_i-\mu_r)^\tT \Sigma_r^{-1} (x_i-\mu_r)\\
	\mu_{r+1}    
	&=   \frac{ \sum\limits_{i=1}^{n} w_i \frac{ 1}{\nu +   \delta_{i,r}}x_i}{ \sum\limits_{i=1}^{n} w_i  \frac{1}{\nu + \delta_{i,r} }} \\
	\Sigma_{r+1} 
	&=   \frac{\sum\limits_{i=1}^{n} w_i \frac{ (x_i-\mu_r)(x_i-\mu_r)^\tT }{\nu + \delta_{i,r}}}{\sum\limits_{i=1}^{n} w_i  \frac{1}{\nu + \delta_{i,r} }}
	\end{align*}
	\EndFor
\end{algorithmic}
\end{algorithm}

The classical EM algorithm for maximizing the log-likelihood function~\cite{LR95} with equal weights $w_i = \frac1n$, $i=1,\ldots,n$,
leads to the same iteration for $\mu$, 
but the iteration of $\Sigma$ reads as
\begin{align}
\Sigma_{r+1} 
&=   \frac{\nu+d}{n} \sum\limits_{i=1}^{n}   \frac{(x_i-\mu_{r+1})(x_i-\mu_{r+1})^\tT }{\nu + \delta_{i,r}}.
\end{align}
In the variant of Kent et al., this expression is divided by $\sum\limits_{i=1}^{n} w_i  \frac{1}{\nu + \delta_{i,r} }$.
Despite the nonuniform weights, our Algorithm \ref{alg:myriad_mult_general} differs from this update 
 by taking $\mu_r$ instead of $\mu_{r+1}$ on the right-hand side and can  therefore be seen as its Jacobi variant.

In the following, we focus on estimating both $\mu$  and $\Sigma$, where we assume that $\nu \geq 1$.
However, the algorithm without the iteration in $\mu$ and its convergence proof work also for estimating only $\Sigma$ 
with $\nu \geq 0$.
In case $\nu = 0$, we have to assume that $x_i \in \SP^{d-1}$, $i=1,\ldots,n$.
Then the iteration for $\Sigma$ is the well-known Tyler $M$-estimator~\cite{tyler1987a} and all iterates have trace 1. 

\subsection{Convergence Analysis}
Since Algorithm \ref{alg:myriad_mult_general} cannot be interpreted as an EM algorithm, we prove its convergence in the following.
We note that the following Lemmata \ref{Theo:descent_Sigma} and \ref{Theo:descent_mu}
on the monotone descent of the log likelihood function
could be also deduced from convergence results of the EM algorithm, but one obtains different estimates.
However, parts of those proofs are needed to prove Theorem \ref{Theo:descent_joint}.

Using again $\Sigma = R R^\tT$, we introduce the notation $y_i \coloneqq R^{-1}(x_i - \mu)$, $i=1,\ldots,n$ and
\begin{align}
S_0(\mu,\Sigma) &\coloneqq 	
(d+\nu)\sum_{i=1}^n w_i \frac{1}{\nu+ \delta_i},\label{S0_normalized} \\
S_1(\mu,\Sigma) &\coloneqq 	
(d+\nu)\sum_{i=1}^n w_i\frac{ y_i}{\nu+ \delta_i}, \label{S1_normalized}\\	
S_2(\mu,\Sigma) &\coloneqq 	
(d+\nu)\sum_{i=1}^n w_i \frac{y_i y_i^\tT} {\nu+ \delta_i}, \label{S2_normalized}	 
\end{align}
where the subscripts $0$, $1$ and $2$ stand for 'scalar', 'vector' and 'matrix', respectively.
Further we abbreviate $S_{jr} \coloneqq  S_j(\mu_r,\Sigma_r)$, $j=0,1,2$ and use this notation also if $\mu$ or $\Sigma$ is fixed.
By the Cholesky decomposition $\Sigma_r = R_r R_r^\tT$,
the iterations in Algorithm~\ref{alg:myriad_mult_general} read as
\begin{align}
\mu_{r+1} & = \mu_r + R_r \frac{S_{1r}}{S_{0r}} = \mu_r - \frac{1}{2S_{0r}}\Sigma_r \frac{\partial L }{\partial \mu}(\mu_r,\Sigma_r),\\
\Sigma_{r+1} & = R_{r+1} {R}^\tT_{r+1} = R_r  \frac{S_{2r}}{S_{0r}}R^\tT_{r}
= R_r  \frac{S_{2r}}{\frac{d+\nu}{\nu} - \frac{1}{\nu}\tr\bigl(S_{2r}\bigr)}R^\tT_{r}
= R_r  \frac{\nu S_{2r}}{d+\nu - \tr\bigl(S_{2r}\bigr)}R^\tT_{r} \\
&= \frac{1}{S_{0r} }\left(\Sigma_r - \Sigma_r\frac{\partial L }{\partial \Sigma}(\mu_r,\Sigma_r)\Sigma_r\right). \label{def_1}
\end{align}
We will need the difference of two iterates 
\begin{align}
L_\nu(\mu_{r+1},\Sigma_{r+1})- L_\nu(\mu_{r},\Sigma_{r})
& = (d+\nu)\sum_{i=1}^n w_i \log\Biggl(\frac{\nu + \delta_{i,r+1}}{ \nu + \delta_{i,r}} 
\frac{|\Sigma_{r+1}|^{\frac{1}{d+\nu}}}{|\Sigma_r|^{\frac{1}{d+\nu}}}\Biggr). \label{zielfkt}
\end{align}

We start with the convergence analysis of our algorithm for fixed $\mu$, where we might assume w.l.o.g. that $\mu = 0$

\begin{Lemma}\label{Theo:descent_Sigma}
Let $\nu \ge 0$, $\mu = 0$ and $x_i \not = 0$, $i=1,\ldots,n$.
Then $\{\Sigma_r\}_{r\in \N}$  defined by the iterations in Algorithm~\ref{alg:myriad_mult_general} 
satisfy
\begin{align}
L(\Sigma_{r+1})- L(\Sigma_{r}) &\leq 0 \quad \mathrm{for} \quad \nu > 0,\\
L_0(\Sigma_{r+1})- L_0(\Sigma_{r}) &\leq 0\quad \mathrm{for} \quad \nu = 0,
\end{align}
with equality if and only if $\Sigma_{r+1} = \Sigma_r$.
\end{Lemma}

\begin{proof} We consider $\nu > 0$. The proof for $L_0$ follows exactly the same lines with $\nu = 0$.
By concavity of the logarithm and \eqref{zielfkt} we have
\begin{align*}
L(\Sigma_{r+1})- L(\Sigma_{r})
&\leq  (d+\nu)\log\Biggl(\underbrace{\sum_{i=1}^n w_i \frac{\nu 
+ \delta_{i,r+1}}{ \nu 
+ \delta_{i,r}} \frac{|\Sigma_{r+1}|^{\frac{1}{d+\nu}}}{|\Sigma_r|^{\frac{1}{d+\nu}}}}_{=\Upsilon}\Biggr),
\end{align*}	
so that  it suffices to show that $\Upsilon\leq 1$.
Using  properties of the determinant it holds 
\begin{equation*}
\frac{|\Sigma_{r+1}|^{\frac{1}{d+\nu}}}{|\Sigma_r|^{\frac{1}{d+\nu}}} 
= \frac{\left|R_r \frac{S_{2r}}{S_{0r}} {R_r^\tT}\right|^{\frac{1}{d+\nu}}}{|R_r {R_r^\tT}|^{\frac{1}{d+\nu}}} 
= S_{0r}^{-\frac{d}{d+\nu}}|S_{2r}|^{\frac{1}{d+\nu}}.
\end{equation*}
Next, we consider the term 
\begin{equation*}
\sum_{i=1}^n w_i \frac{\nu + \delta_{i,r+1}}{ \nu + \delta_{i,r}} 
= \sum_{i=1}^n w_i \frac{\delta_{i,r+1}}{ \nu + \delta_{i,r}} + \frac{\nu }{d+\nu}S_{0r}.
\end{equation*}
Using 
\begin{align*}
\delta_{i,r+1} = \tr\Bigl(	x_i^\tT \Sigma_{r+1}^{-1}\, x_i\Bigr)  
 = \tr\left( x_i^\tT \, R_r^{-\tT} S_{0r} S_{2r}^{-1} R_r^{-1} \, x_i \right)
 = S_{0r}\tr\left(S_{2r}^{-1} R_r^{-1} \, x_i x_i^\tT \, R_r^{-\tT}\right)
\end{align*}
and the linearity of the trace, the sum simplifies to 
\begin{align*}
\sum_{i=1}^n w_i \frac{\delta_{i,r+1}}{  \nu+\delta_{i,r}} 
& = \sum_{i=1}^n w_i \frac{\tr\left(\delta_{i,r+1}\right)}{ \nu+\delta_{i,r}}
= {S_{0r}} \sum_{i=1}^n w_i \frac{\tr\left(S_{2r}^{-1} R_r^{-1} \, x_i x_i^\tT \, R_r^{-\tT}\right)}{ \nu+\delta_{i,r}}\\
& =  {S_{0r}}\tr\Biggl(S_{2r}^{-1}
\underbrace{\sum_{i=1}^n w_i \frac{\left( R_r^{-1} \, x_i x_i^\tT \, R_r^{-\tT}\right)}{  \nu+\delta_{i,r}}}_{=\frac{1}{d+\nu}S_{2r}}\Biggr)\\
& =  \frac{1}{d+\nu}S_{0r}  \tr(I) =  \frac{d}{d+\nu}S_{0r} .
\end{align*}
Thus, we obtain
\begin{align*}
\Upsilon &= \sum_{i=1}^n w_i \frac{\nu + \delta_{i,r+1}}{ \nu + \delta_{i,r}} 
\frac{|\Sigma_{r+1}|^{\frac{1}{d+\nu}}}{|\Sigma_r|^{\frac{1}{d+\nu}}}
 = S_{0r}S_{0r}^{-\frac{d}{d+\nu}} |S_{2r}|^{\frac{1}{d+\nu}} 
=  \left(S_{0r}^\nu \, |S_{2r}|\right)^{\frac{1}{d+\nu}}.
\end{align*}
We have $\nu S_{0r} + \tr(S_{2r}) = d + \nu $.
(If $\nu = 0$ and $n > d$, we are ready here since $\tr(S_{2r}) = d$ implies by the arithmetic-geometric mean property that
$|S_{2r}| \le (d/n)^n$.)
For $\nu > 0$, 
we can express $S_{0r}$ in terms of $\tr(S_{2r})$ as
\begin{equation*}
S_{0r} = \frac{d+\nu }{\nu} - \frac{1}{\nu}\tr(S_{2r}).
\end{equation*}
Next we consider maximizing the  function 
\begin{equation*}
g\colon \SPD(d)\to \R,\qquad g(X) = \left(\frac{d+\nu }{\nu} - \frac{1}{\nu}\tr(X)\right)^\nu |X|
\end{equation*}
under the constraint $0\leq \tr(X)\leq {d+\nu}$. 
Note that for $X\in \SPD(d)$ we always have $\tr(X),|X|>0$. 
Further, if $\tr(X) = 0$ or $\tr(X)=d+\nu$, 
we set $g(X) = 0$. 
Since $g(I) = 1$, the maximum is not attained at the boundary, 
but inside the open set. The derivative of $g$ with respect to $X$ is given by
\begin{align*}
\nabla g(X) &=- \nu\left(\frac{d+\nu }{\nu} - \frac{1}{\nu}\tr(X)\right)^{\nu-1} \frac{1}{\nu }I|X| 
+ \left(\frac{d+\nu }{\nu} - \frac{1}{\nu}\tr(X)\right)^\nu |X|X^{-1}\\
&  = \left(\frac{d+\nu }{\nu} - \frac{1}{\nu}\tr(X)\right)^{\nu-1} |X|\Biggl[\left(\frac{d+\nu }{\nu} 
- \frac{1}{\nu}\tr(X)\right) X^{-1} - I\Biggr].
\end{align*}
The necessary condition for a critical point $\hat{X}$ of $g$ reads as
\begin{equation*}
\left(\frac{d+\nu}{\nu}-\frac{1}{\nu}\tr(\hat{X})\right) I =  \hat{X}. 
\end{equation*}
To verify that $\hat{X} =I$ is a maximizer, 
we have a look at the Hessian $\nabla^2 g$ of $g$ and show that it is negative definite for $\hat{X}=I$. 
We compute
\begin{align*}
D(\nabla g)(X)[H]
& = |X|\left(\frac{d+\nu}{\nu} - \frac{1}{\nu}\tr(|X|)\right)^{\nu-2} \left[ \frac{\nu-1}{\nu} \tr(H)I \right.\\
& \;\;\;\; - \left(\frac{d+\nu}{\nu} - \frac{1}{\nu}\tr(|X|)\right)\left(\tr(X^{-1}H)I + \tr(H)X^{-1}\right)\\
& \;\;\;\; \left.+ \left(\frac{d+\nu}{\nu} - \frac{1}{\nu}\tr(|X|)\right)^{2}\tr(X^{-1}H)X^{-1} - X^{-1}HX^{-1} \right]
\end{align*}
and further
\begin{align*}
\langle \nabla^2g(X)[H],H\rangle 
& = |X|\left(\frac{d+\nu}{\nu} - \frac{1}{\nu}\tr(|X|)\right)^{\nu-2} \left[ \frac{\nu-1}{\nu} \tr(H)^2 \right.	\\ 	
& \;\;\;\; 
-2\left(\frac{d+\nu}{\nu} - \frac{1}{\nu}\tr(|X|)\right)\tr(H)\tr(X^{-1}H) \\
& \;\;\;\; \left.+ \left(\frac{d+\nu}{\nu} - \frac{1}{\nu}\tr(|X|)\right)^2 \bigl(\tr(X^{-1}H)^2 - \tr((X^{-1}H)^2)\bigr)\right].
\end{align*}
At the critical point $\hat{X}=I$ we obtain
\begin{align*}
\langle \nabla^2g(I)[H],H\rangle & = 	 \frac{\nu-1}{\nu} \tr(H)^2 - 2 \tr(H)^2 + \tr(H)^2 - \tr(H^2)\\
& = -\frac{1}{\nu}\tr(H)^2 - \tr(H^2)<0,	
\end{align*}
so that $\hat{X}=I$ is indeed a maximizer. The corresponding maximum of $g$ is given by 
$g(I) = \left(\frac{d+\nu}{\nu}-\frac{1}{\nu }\tr(I)\right)^\nu  |I| = 1 $ and therewith finally
\begin{equation*}
\Upsilon = \left(S_{0r}^\nu|S_{2r}|\right)^{\frac{1}{d+1}} = g(S_{2r})^{\frac{1}{d+1}} \leq 1.
\end{equation*}	 
By \eqref{def_1}, we have equality in  the theorem if and only if $\Sigma_{r+1} = \Sigma_r$.
\end{proof}

Next, we analyze the difference of two iterates for fixed $\Sigma$. 

\begin{Lemma}\label{Theo:descent_mu}
For fixed $\Sigma\in \SPD(d)$ and $\nu>0$,  let $\{\mu_r\}_{r\in \N}$ be defined by Algorithm~\ref{alg:myriad_mult_general}. Then it holds
\begin{equation*}
L(\mu_{r+1},\Sigma)- L(\mu_{r},\Sigma)\leq 0
\end{equation*}
with equality if and only if $\mu_{r+1} = \mu_r$.
\end{Lemma}

\begin{proof}
By concavity of the logarithm and \eqref{zielfkt} we have 
\begin{align*}
L(\mu_{r+1},\Sigma)- L(\mu_{r},\Sigma) 
&= (d+\nu)\sum_{i=1}^n w_i \log\Biggl(\frac{ \nu+\delta_{i,r+1}}{  \nu+\delta_{i,r}}\Biggr)\\
& \leq (d+\nu)\log\Biggl(\underbrace{\sum_{i=1}^n w_i\frac{ \nu+\delta_{i,r+1}}{\nu+\delta_{i,r}}}_{=\Upsilon}\Biggr),
\end{align*}	
so that  it suffices to show that $\Upsilon\leq 1$.
We  compute
\begin{small}
	\begin{align*}
\Upsilon& = \sum_{i=1}^n w_i\frac{ \nu+(x_i-\mu_{r}+\mu_{r}-\mu_{r+1})^\tT \Sigma^{-1}(x_i-\mu_{r}+\mu_{r}-\mu_{r+1})}{\nu+\delta_{i,r}}\\
& =  \sum_{i=1}^n w_i\frac{\nu+\delta_{i,r}}{\nu+\delta_{i,r}}
\\
& \quad + \frac{2\langle R^{-1}(x_i-\mu_{r}), (\overbrace{R^{-1}(\mu_r-\mu_{r+1})}^{-\frac{S_{1r}}{S_{0r}}})\rangle
+ (\mu_{r}-\mu_{r+1})^\tT \Sigma^{-1}(\mu_{r}-\mu_{r+1})}{\nu+\delta_{i,r}}\\
& = 1 -2\sum_{i=1}^n w_i\frac{\left\langle R^{-1}(x_i-\mu_{r}), \frac{S_{1r}}{S_{0r}}\right\rangle}{\nu+\delta_{i,r}} 
+ \sum_{i=1}^n w_i\frac{ \norm{\frac{S_{1r}}{S_{0r}}}{2}^2}{ \nu+\delta_{i,r}}\\
& = 1-2\left\langle \underbrace{\sum_{i=1}^n w_i\frac{ R^{-1}(x_i-\mu_{r})}{ \nu+\delta_{i,r}} }_{=\frac{1}{d+\nu}S_{1r}} ,
\frac{S_{1r}}{S_{0r}}  \right\rangle + \norm{\frac{S_{1r}}{S_{0r}}}{2}^2 
\underbrace{\sum_{i=1}^n w_i\frac{ 1}{ \nu+\delta_{i,r}}}_{=\frac{1}{d+\nu}S_{0r}}\\
& = 1- \frac{1}{d+\nu}\frac{\norm{S_{1r}}{2}^2}{S_{0r}}\leq 1,
\end{align*}
\end{small}
with equality if and only if $S_{1r}=0$, that is, $\mu_{r+1}=\mu_r$ and $\mu_r$ is a critical point of $L(\cdot,\Sigma)$. 
\end{proof}

Combining the results of Lemma~\ref{Theo:descent_mu} and~\ref{Theo:descent_Sigma} we obtain the following lemma.
\begin{Lemma}\label{Theo:descent_joint}
For $\nu>0$, 	 let $\{\mu_r,\Sigma_r\}_{r\in \N}$ be defined by Algorithm~\ref{alg:myriad_mult_general}. Then it holds
\begin{equation*}
	L(\mu_{r+1},\Sigma_{r+1})- L(\mu_{r},\Sigma_{r})\leq 0
\end{equation*}
with equality if and only if $(\mu_{r+1},\Sigma_{r+1}) = (\mu_r,\Sigma_r)$.
\end{Lemma}

\begin{proof}
By concavity of the logarithm and \eqref{zielfkt} we have  
\begin{align*}
L(\mu_{r+1},\Sigma_{r+1})- L(\mu_{r},\Sigma_{r})
& \leq(d+\nu)\log\Biggl(\underbrace{\sum_{i=1}^n w_i \frac{\nu + \delta_{i,r+1}}{ \nu + \delta_{i,r}} \frac{|\Sigma_{r+1}|^{\frac{1}{d+\nu}}}{|\Sigma_r|^{\frac{1}{d+\nu}}}}_{=\Upsilon}\Biggr),
\end{align*}
and  it suffices to show that $\Upsilon\leq 1$.
As in the proof of Lemma~\ref{Theo:descent_Sigma}	we have
\begin{equation*}
\frac{|\Sigma_{r+1}|^{\frac{1}{d+\nu}}}{|\Sigma_r|^{\frac{1}{d+\nu}}} =  S_{0r}^{-\frac{d}{d+\nu}}|S_{2r}|^{\frac{1}{d+\nu}}.
\end{equation*}	
Next, we consider the term 
\begin{equation*}
\sum\limits_{i=1}^n w_i \frac{\nu + \delta_{i,r+1}}{ \nu + \delta_{i,r}} = \sum\limits_{i=1}^n w_i \frac{ \delta_{i,r+1}}{ \nu + \delta_{i,r}} +  \frac{\nu }{d + \nu}S_{0r}.
\end{equation*}
Combining the computations in the proofs of Lemma~\ref{Theo:descent_mu} and Lemma~\ref{Theo:descent_Sigma}, we get 
\begin{small}
\begin{align*}
&	\sum\limits_{i=1}^n w_i \frac{ \delta_{i,r+1}}{ \nu + \delta_{i,r}} \\
& = \sum_{i=1}^n w_i \frac{\delta_{i,r} + 2\langle R_{r+1}^{-1}(x_i-\mu_r),R_{r+1}^{-1}(\mu_r-\mu_{r+1})\rangle 
+ \delta_{i,r+1}}{ \nu+\delta_{i,r}}\\
& = \sum\limits_{i=1}^n w_i \frac{S_{0r}\tr\left(S_{2r}^{-1} 
R_r^{-1}(x_i-\mu_r)(x_i-\mu_r)^\tT R_r^{-\tT}\right)}{ \nu+\delta_{i,r}}-2 
\sum_{i=1}^n w_i \frac{(x_i-\mu_r)^\tT R_r^{-\tT}S_{0r} S_{2r}^{-1} R_r^{-1} R_r\frac{S_{1r}}{S_{0r}}}{ \nu+\delta_{i,r}}\\
& \quad + \sum_{i=1}^n w_i \frac{\frac{S_{1r}^\tT}{S_{0r}}R_r^\tT R_r^{-\tT} S_{0r}S_{2r}^{-1} 
R_r^{-1} R_r\frac{S_{1r}}{S_{0r}}}{ \nu+\delta_{i,r}}\\
& = \frac{d}{d+\nu} S_{0r} - \frac{2}{d+\nu} S_{1r}^\tT S_{2r}^{-1} S_{1r} 
+  \frac{ S_{1r}^\tT S_{2r}^{-1} S_{1r}}{S_{0r}}\sum_{i=1}^n w_i \frac{1}{  \nu+\delta_{i,r}}\\
& = \frac{d}{d+\nu}S_{0r} -\frac{1}{d+\nu} S_{1r}^\tT S_{2r}^{-1} S_{1r}.
\end{align*}
\end{small}
Since $S_{2r}\in \SPD(d)$ and consequently also $S_{2r}^{-1}\in \SPD(d)$,  we  obtain
\begin{align*}
\Upsilon
& = \left(\left(\frac{d}{d+\nu}+\frac{\nu}{d+\nu}\right)S_{0r} 
-\frac{1}{d+\nu}\underbrace{S_{1r}^\tT S_{2r}^{-1} S_{1r}}_{\geq 0} \right)S_{0r}^{-\frac{d}{d+\nu}}|S_{2r}|^{\frac{1}{d+\nu}}
\leq \left(S_{0r}^\nu|S_{2r}|\right)^{\frac{1}{d+\nu}}\leq 1. \; \qedhere
\end{align*}
\end{proof}

\begin{Theorem}[Convergence of Algorithm~\ref{alg:myriad_mult_general}]\label{Theo:joint_convergence}
	\begin{enumerate}
		\item Let $\nu\ge 1$ and $x_i\in\R^d$, $i=1,\ldots,n$ and $w \in \mathring \Delta_n$ fulfill Assumption~\ref{Ass:affin_ind}.
		Then the sequence  $\{(\mu_r,\Sigma_r)\}_{r\in \N}$ generated by Algorithm~\ref{alg:myriad_mult_general} converges to the minimizer of $L$.
		\item Let $\nu>0$, $\mu$ fixed and $x_i\in\R^d$, $i=1,\ldots,n$ and $w \in \mathring \Delta_n$ fulfill Assumption~\ref{Ass:lin_ind}.
		Then the sequence  $\{\Sigma_r\}_{r\in \N}$ generated by Algorithm~\ref{alg:myriad_mult_general} converges to the minimizer of $L$.
		\item  Let $\nu=0$, $\mu = 0$ and $x_i \in \SP^{d-1}$, $i=1,\ldots,n$ be pairwise different non antipodal points.
		Then the sequence  $\{\Sigma_r\}_{r\in \N}$ generated by Algorithm~\ref{alg:myriad_mult_general} converges to the minimizer of $L_0$
		with trace 1.		
	\end{enumerate}
\end{Theorem}
%

\begin{proof}
We restrict our attention to (i), the other parts follows the same lines.
Let $\{(\mu_r,\Sigma_r)\}_{r\in \N}$ be the sequence of iterates generated by Algorithm~\ref{alg:myriad_mult_general}. 
Consider the mapping 
\begin{equation*}
T(\mu,\Sigma) 
= \left(\mu + \Sigma^{\frac{1}{2}} \frac{S_1(\mu,\Sigma)}{S_0(\mu,\Sigma) },  \Sigma^{\frac{1}{2}} \frac{S_1(\mu,\Sigma)}{S_0(\mu,\Sigma) } \Sigma^{\frac{1}{2}}\right).	
\end{equation*}
Then, according to~\eqref{fp_mu},~\eqref{fp_Sigma} and Theorem~\ref{thm:main}, 
$(\mu,\Sigma) = T(\mu,\Sigma)$ is a fixed point of $T$ if and only if it is the unique critical point/minimizer of $L$. 
Consider the case $(\mu_{r+1},\Sigma_{r+1}) \not = (\mu_r,\Sigma_r)$ for all $r \in \N$. 
We show  that the sequence $\{(\mu_r,\Sigma_r)\}_{r\in \N}$ is bounded: 
the update of $\mu_r$ is just a convex combination of the samples $x_1,\ldots,x_n$. 
By construction, $\Sigma_r$, 
is also a weighted average and the sequence remains bounded since the value
$\mu_r\in \conv\{x_1,\ldots,x_n\}$ stays bounded, 
$(x_i-\mu_r)(x_i-\mu_r)^\tT$, $i=1,\ldots,n$ is bounded as well, 
and consequently also $\Sigma_{r+1}$. 	
By Lemma~\ref{Theo:descent_joint}, we see that the sequence $L_r \coloneqq L(\mu_r,\Sigma_r)$ 
is a strictly decreasing, bounded below sequence such that it converges to some $\hat L$.
Further, $\{ (\mu_r,\Sigma_r)\}_{r\in \N}$ contains a convergent subsequence $\{ (\mu_{r_s},\Sigma_{r_s})\}_{s\in \N}$,
which converges to some $(\hat{\mu},\hat{\Sigma})$. 
By the continuity of $L$ and  $T$  we obtain
\begin{align*}
L (\hat{\mu},\hat{\Sigma}) 
&= \lim_{s\rightarrow \infty} L(\mu_{r_s}, \Sigma_{r_s}) = \lim_{s\rightarrow \infty} L_{r_s}=\lim_{s\rightarrow \infty} L_{r_s+1} \\
&= \lim_{s\rightarrow \infty} L (\mu_{r_s+1}, \Sigma_{r_s+1}) \\
&	= \lim_{s\rightarrow \infty} L\bigl(T(\mu_{r_s}, \Sigma_{r_s}) \bigr)
= L\bigl(T(\hat \mu,\hat \Sigma) \bigr).
\end{align*}
This implies $(\hat \mu,\hat \Sigma) = T(\hat \mu,\hat \Sigma)$, so that $(\hat \mu,\hat \Sigma)$ is a fixed point of $T$ and
consequently the critical point.
Since this point is unique, not only a subsequence, but the whole sequence $\bigl\{(\mu_r,\Sigma_r)\bigr\}_{r\in \N}$ 
converges to $(\hat \mu,\hat \Sigma)$, which finishes the proof.
\end{proof}

\subsection{Simulation Study}\label{sec:simu}
Next we evaluate the numerical performance, in particular the speed of convergence, of the proposed GMMF Algorithm~\ref{alg:myriad_mult_general}
compared to the EM algorithm.
Actually, the EM algorithm~\cite{LR95} in our notation reads as Algorithm \ref{alg:myriad_mult_general} except for the iteration with respect to
$\Sigma$ which is given by
\begin{equation*}
\Sigma_{r+1} 
=   \frac{1}{\nu + d} \sum\limits_{i=1}^{n} w_i \frac{ (x_i-\mu_r)(x_i-\mu_r)^\tT }{\nu + \delta_{i,r}}.
\end{equation*}
The convergence of the EM algorithm under quite general assumptions has been established in~\cite{Wu83}. However, it is well known that the EM algorithm might suffer from slow convergence, 
which is also what we observed in the  following Monte Carlo simulation: 
we draw $n=100$ i.i.d. samples of a $T_\nu(\mu,\Sigma)$ distribution for different degrees of freedom $\nu\in \{1,5,10,100\}$ 
and run Algorithm~\ref{alg:myriad_mult_general} respective the EM Algorithm to
compute the joint ML-estimate $(\hat{\mu},\hat{\Sigma})$. Both algorithms
are initialized with sample mean and sample covariance and we used 
the relative difference between two iterates  $(\mu_r,\Sigma_r)$ and $(\mu_{r+1},\Sigma_{r+1})$ as stopping criterion, 
that is 
\begin{equation*} 
\frac{\sqrt{\lVert \mu_{r+1}-\mu_r\rVert_2^2 
+ \lVert\Sigma_{r+1}-\Sigma_r\rVert_F^2}}{\sqrt{\lVert\mu_r\rVert_2^2 + \lVert\Sigma_r\rVert_F^2}}<10^{-6}.
\end{equation*}
This experiment is repeated $N = 10.000$ times
and afterward, we calculated the average number of
iterations $ \overline{\text{iter}}$  and $\overline{\text{iter}}_\text{EM}$ needed 
to reach the tolerance criterion together with their standard deviations. 
The results are given in Table~\ref{Tab:sampling_experiment}, 
where we chose $d=2$, $\mu = 0$  and different values for $\Sigma$. 
First, we notice that the average number of iterations is 
in general higher for the EM Algorithm, and further, 
it does merely not depend on $(\mu,\Sigma)$, but only on the degree of freedom $\nu$. 
Here, the smaller the value of $\nu$, the larger on the one hand the number of iterations for both algorithms, 
and on the other hand the larger the gain in speed of Algorithm~\ref{alg:myriad_mult_general} 
compared to the EM Algorithm. The fact that only very few iterations are needed for large $\nu$ 
can be explained by the fact that for $\nu \to \infty$ the Student-$t$ distribution converges to the normal distribution,  
so that on the one hand, the estimation becomes in some sense easier, and on the other hand, we can expect the initialization with sample mean 
and sample covariance matrix, which are the maximum likelihood estimates for the normal distribution, to be already close to the actual parameters to be estimated.
Additionally, we examined the speed of convergence for other location parameters $\mu$ as well as   fixing $\mu$ and estimating only $\Sigma$. 
Here, the results are qualitatively and quantitatively similar.

\begin{table*}[htbp]
\centering
\footnotesize
\caption{Comparison of GMMF Algorithm~\ref{alg:myriad_mult_general} and the EM Algorithm.}
\label{Tab:sampling_experiment}
\begin{tabular}{c|l|ll}
	$\Sigma$&	$\nu$	        & $\overline{\text{iter}}\pm \sigma(\overline{\text{iter}})$&  
	$\overline{\text{iter}}_\text{EM}\pm \sigma(\overline{\text{iter}}_\text{EM})$ \\ \cline{1-4}
	\multirow{5}{*}{$\begin{pmatrix}
		0.1 & 0 \\
		0 & 0.1
		\end{pmatrix}$ }& $1$ & $20.7582\pm 1.5430$ & $ 60.6318\pm 3.9313$\\
	& $2$ & $ 16.0843 \pm 1.1242 $ & $33.5389 \pm 2.0370 $ \\
	& $5$ & $11.166\pm 0.8121$ & $16.8973\pm 0.9948$ \\
	& $10$ & $8.5245\pm 0.6450$ & $11.1186\pm 0.6534$\\
	& $100$ & $4.1066\pm 0.3086$ & $ 4.9072\pm 0.2915$ \\\cline{1-4}
	\multirow{5}{*}{$\begin{pmatrix}
		1 & 0 \\
		0 & 1
		\end{pmatrix}$ }& $1$ & $20.3536\pm 1.5899$ & $ 60.8843\pm 3.9302$\\
	& $2$ & $ 15.7742\pm 1.1840  $ & $ 33.6515\pm 2.0373 $ \\
	& $5$ & $10.9528\pm 0.8513$  &  $16.9305\pm 0.9957$ \\
	& $10$ & $8.3487\pm 0.6646$ & $11.1186\pm 0.6534$\\
	& $100$ & $4.0654\pm 0.2472$ & $4.9040\pm 0.2953$ \\\cline{1-4}
	\multirow{5}{*}{$\begin{pmatrix}
		5 & 0 \\
		0 & 5
		\end{pmatrix}$ }& $1$ &  $20.2702\pm 1.6145$ & $60.9139\pm 3.9326$ \\
	& $2$ & $ 15.7136\pm 1.0099 $ & $ 33.6644\pm 2.0381  $ \\
	& $5$ & $10.9100\pm 0.8699$  & $16.9343\pm 0.9957$  \\
	& $10$ & $ 8.3181\pm 0.6738 $ & $11.1191\pm 0.6540$\\
	& $100$ & $4.0627\pm 0.2424$ & $4.9035\pm 0.2960$ \\\cline{1-4}
	\multirow{5}{*}{$\begin{pmatrix}
		10 & 0 \\
		0 & 10
		\end{pmatrix}$ }& $1$ & $20.2592\pm 1.6179$ & $ 28.0073\pm 2.0546$\\
	& $2$ & $15.7055 \pm 1.2136  $ & $ 33.6662\pm 2.0386 $ \\
	& $5$ &  $10.9050\pm 0.8725$ &  $16.9346\pm 0.9959$ \\
	& $10$ & $8.3137\pm 0.6757$ & $11.1195\pm 0.6537$\\
	& $100$ & $4.0623\pm 0.2417$ & $4.9036\pm 0.2958$ \\\cline{1-4}
	\multirow{4}{*}{$\begin{pmatrix}
		2 & -1 \\
		-1  & 2 
		\end{pmatrix}$ }& $1$ & $20.2091\pm 1.6384$ & $ 27.9290\pm 2.1314$\\
	& $2$ & $ 15.6265\pm 1.2344 $ & $ 33.6569\pm 2.0452 $ \\
	& $5$ &  $10.8407\pm 0.8841$ &  $16.9230\pm 0.9954$ \\
	& $10$ &  $8.2607\pm 0.6844$ &  $11.1092\pm 0.6534$ \\
	& $100$ & $4.0573\pm 0.2324$ & $4.8908\pm 0.3126 $ \\\cline{1-4}
\end{tabular}
\end{table*}

\subsection{ML Estimation of Wrapped Cauchy Distribution}

By the relation between the  wrapped Cauchy and projected normal distribution in Lemma~\ref{Prop:wrapped_Cauchy}, 
we can use our GMMF Algorithm~\ref{alg:myriad_mult_general} 
with $\mu_r\equiv\mu = 0$ for the ML estimation of the wrapped Cauchy distribution. 
By \eqref{rho} and \eqref{a}  and \eqref{nenner}, 
the iteration for $\Sigma$  in Algorithm~\ref{alg:myriad_mult_general} is of the form
\begin{align*}
\Sigma_{r+1} &= \begin{pmatrix}
\sigma_{11,r+1} & \sigma_{12,r+1}\\
\sigma_{12,r+1} & \sigma_{22,r+1}
\end{pmatrix} 
=  \frac{ \sum\limits_{i=1}^{n} w_i \frac{ x_i x_i^\tT }{x_i^\tT \Sigma_{r}^{-1}x_i}}{\sum\limits_{i=1}^{n} w_i \frac{ 1}{x_i^\tT \Sigma_{r}^{-1}x_i}} \\
&= 
\frac{ \sum\limits_{i=1}^{n} w_i \frac{1}{ \bigl(1-\xi_{1,r}\cos(2\phi_i) - \xi_{2,r} \sin(2\phi_i)\bigr) }
\begin{pmatrix}
		\cos^2(\phi_i) & \cos(\phi_i)\sin(\phi_i)\\
		\cos(\phi_i)\sin(\phi_i) & \sin^2(\phi_i)
		\end{pmatrix}
		}
		{\sum\limits_{i=1}^{n} w_i \frac{ 1}{\bigl(1-\xi_{1,r}\cos(2\phi_i) - \xi_{2,r} \sin(2\phi_i)\bigr)}}
\end{align*}
Using relations for trigonometric functions and $\theta_{i} = 2\phi_i$, we obtain
\begin{align*}
\zeta_{1,r+1} &= \frac{\sigma_{11,r+1} - \sigma_{22,r+1}}{ \sigma_{11,r+1} + \sigma_{22,r+1}}
=\frac{\sum\limits_{i=1}^n w_i 
\frac{\cos(\theta_i)}{1-\xi_{1,r}\cos(\theta_i)-\xi_{2,r}\sin(\theta_i)}}{\sum\limits_{i=1}^n w_i \frac{1}{1-\xi_{1,r}\cos(\theta_i)-\xi_{2,r}\sin(\theta_i)} }
,\\
\zeta_{2,r+1} &= \frac{2\sigma_{12,r+1}}{ \sigma_{11,r+1} + \sigma_{22,r+1}}
=
\frac{\sum\limits_{i=1}^n w_i 
\frac{\sin(\theta_i)}{1-\xi_{1,r}\cos(\theta_i)-\xi_{2,r}\sin(\theta_i)}}{\sum\limits_{i=1}^n w_i \frac{1}{1-\xi_{1,r}\cos(\theta_i)-\xi_{2,r}\sin(\theta_i)} }.
\end{align*}
This results in Algorithm~\ref{alg:wrapped_Cauchy}. From $\zeta_1$ and $\zeta_2$ we obtain the
desired estimation of parameters $(a,\gamma)$ of the wrapped Cauchy distribution via \eqref{rho} and \eqref{a}.

\begin{algorithm}[!ht]
\caption{ML estimation for the wrapped Cauchy distribution}\label{alg:wrapped_Cauchy}
\begin{algorithmic}
	\State \textbf{Input:} $\theta_1,\ldots,\theta_n\in [-\pi,\pi)$, $n \geq 3$, $w \in \mathring \Delta_n$ fulfilling Assumption~\ref{Ass:affin_ind}
	\State \textbf{Initialization:} $\zeta_{1,0} = \zeta_{2,0} =0$
	\For{$r=0,\ldots$}
	\begin{align*} 	
	\zeta_{1,r+1}& = \frac{\sum\limits_{i=1}^n w_i \frac{\cos(\theta_i)}{1-\zeta_{1,r}\cos(\theta_i)-\zeta_{2,r}\sin(\theta_i)}}{\sum\limits_{i=1}^n w_i \frac{1}{1-\zeta_{1,r}\cos(\theta_i)-\zeta_{2,r}\sin(\theta_i)} }\\
	\zeta_{2,r+1}& = \frac{\sum\limits_{i=1}^n w_i \frac{\sin(\theta_i)}{1-\zeta_{1,r}\cos(\theta_i)-\zeta_{2,r}\sin(\theta_i)}}{\sum\limits_{i=1}^n w_i \frac{1}{1-\zeta_{1,r}\cos(\theta_i)-\zeta_{2,r}\sin(\theta_i)} }\\
	\end{align*}
	\EndFor
\end{algorithmic}
\end{algorithm}

\section{Applications in Image Analysis}\label{Sec:Application}

In this section, we describe how our GMMF can be used to denoise images corrupted by different kinds of (additive) noise, 
in particular Cauchy, Gaussian and wrapped Cauchy noise.

\subsection{Nonlocal Denoising Approach}
Let $f\colon \GG\rightarrow \R$ be a noisy image, where $\GG = \{1,\ldots,n_1\}\times \{1,\ldots,n_2\}$ 
denotes the image domain. Here and in all subsequent cases we extend the image by mirroring  at the boundary. 
We assume that each pixel $f_i$, $i=(i_1,i_2)\in \GG$ is affected by noise in an independent and identical way.
Based on Theorem \ref{Prop:Student_t},this can be modeled as
\begin{equation} \label{noise_generation}
f_i = u_i + \sigma\frac{\eta}{\sqrt{y}}, \qquad i\in \GG,
\end{equation} 
where 
$u$ is the noise-free image we wish to reconstruct,
$\eta $ is a realization of $Z\sim\NN(0,1)$, and
$y$ is a realization of $Y\sim \Gamma\left(\frac{\nu}{2},\frac{\nu }{2}\right)$, where
$Z$ and $Y$ are independent.  
The given parameter $\nu\ge 1$  determines the amount of outliers 
and $\sigma>0$ their strength.
This results in independent realizations $f_i$ of $T_\nu(u_i,\sigma)$ distributed random variables.

\paragraph{Selection of Samples.}
The estimation of the noise-free image requires to select for each $i\in \GG$ a set of indices of samples $\SS(i)$ that are interpreted as i.i.d.\ 
realizations of $T_\nu(u_i,\sigma)$. 
We focus here on a nonlocal approach, which is based on an image self-similarity assumption stating that small patches of an image can be found several  times in the image. 
Then, the set $\SS(i)$ constitutes of the indices of the centers of $K$ patches that are most similar to the patch centered at $i\in \GG$. This  requires the selection of the patch size 
and an appropriate similarity measure, which need to be adapted to the noise statistic and the noise level and is described in the next paragraph.
In order to avoid a computational overload one typically restricts 
the search zone for similar patches to a $w\times w$ search window around $i\in \GG$. 
Having defined for each $i\in \GG$ a set of indices of samples $\SS(i)$, the noise-free image can be estimated as
\begin{align*}
\hat{u}_i &\in \argmin_{u_i}	\Bigl\{ L\left(u_i,\sigma|\{f_j\}_{j\in \SS(i)}\right)\Bigr\},\quad i\in \GG.
\end{align*}
However, by using such a pixelwise treatment we implicitly assume that the pixels of an image are independent of each other, which is in practice a rather unrealistic assumption; in fact, 
in natural images they are locally usually highly correlated. Taking the local dependence structure 
into account may improve the results of image restoration methods, 
which motivates to take  whole patches (and not only their centers) and estimate their parameters using Algorithm~\ref{alg:myriad_mult_general} as
 \begin{align*}
 ( \hat{\mu}_i, \hat{\Sigma}_i ) &\in \argmin_{\mu_i,\Sigma_i}	\Bigl\{ L\left(\mu_i,\Sigma_i|\{p_j\}_{j\in \SS(i)}\right)\Bigr\},\quad i\in \GG,
 \end{align*}
 where $p_j$ denotes a patch centered at $j\in \SS(i)$. 
If $\nu>2$, such that the covariance matrix of the Student-$t$ distribution exists, we take the estimated correlation into account 
and  restore a patch $\hat{p}_i$ as
\begin{equation*}
	\hat{p}_i = \hat{\mu}_i + \left(\hat{\Sigma}_i -\frac{\nu}{\nu-2} \sigma^2 I\right)\hat{\Sigma}_i^{-1}(p_i-\hat{\mu}_i),
\end{equation*}
 which is in this case the best linear unbiased estimator (BLUE) of $p_i$, see \cite{LNPS16}.
 If $\nu\le2$, we set $\hat p_i \coloneqq \hat{\mu}_i$.
Proceeding as above gives multiple estimates for each single image pixel that are averaged in the end to obtain the final image.

\paragraph{Patch Similarity.}
The selection of similar patches constitutes  a fundamental step in our nonlocal denoising approach.  At this point, the question arises how to compare noisy patches and numerical 
examples show that an adaptation of the similarity measure to the noise
distribution is essential for a robust similarity evaluation.
In~\cite{DDT12}, the authors formulated 
the similarity between  patches as a statistical hypothesis testing problem and proposed 
among other criteria a similarity measure based on a generalized likelihood test, 
which we use in the following. Further details on this approach can also be  found in~\cite{LPS18}. We briefly summarize the main ideas.

Modeling noisy images in a stochastic way allows to formulate the question whether two patches $p$ and $q$ are similar as a hypothesis test.
Two noisy patches $p,q$ are   considered to be similar  if they are
realizations  of  independent  random  variables $X\sim p_{\theta_1}$ and $Y\sim p_{\theta_2}$ 
that follow the same parametric distribution $p_\theta$, $\theta\in \Theta$ 
with a common parameter $\theta$ (corresponding to the underlying noise-free patch), i.e.\ $\theta_1 = \theta_2 \equiv \theta$. 
Therewith, the evaluation of the similarity between noisy patches can be
formulated as the following hypothesis test:
\begin{align}
\HH_0&\colon \theta_1 = \theta_2\qquad \text{vs.}\qquad 
\HH_1\colon \theta_1 \not = \theta_2.	\label{similarity_test}
\end{align}
In this context, a similarity measure $S$ maps a pair of noisy patches $p,q$ to a real value $c\in \R$. The larger this value $c$ is, the more the patches are considered to be similar. 

In the following, we describe how a similarity measure $S$ can be obtained based on a suitable test statistic for the hypothesis testing problem. 
In general, according to the Neyman-Pearson Theorem, see, e.g.~\cite{CB02}, the optimal test statistic 
(i.e.\ the one that maximizes the power for any given size $\alpha$) 
for single-valued hypotheses of the form
\begin{align*}
\HH_0&\colon \theta= \theta_0\qquad \text{vs.}\qquad 
\HH_1\colon \theta  = \theta_1	
\end{align*} 
is given by a likelihood ratio test. Note that single-valued testing problems correspond to a disjoint partition 
of the parameter space of the form $\Theta = \Theta_0\dot{\cup} \Theta_1$, where $\Theta_i = \{\theta_i\}$, $i=0,1$.	

Despite being a very strong theoretical result, the practical relevance of the Neyman-Pearson Theorem is limited due 
to the fact that $\Theta_0$ and $\Theta_1$ are in most applications not single-valued. 
Instead, the testing problem is a so called \emph{composite testing} problem, 
meaning that $\Theta_0$ and/or $\Theta_1$ contain more than one element.
It can be shown that for composite testing problems there does not exist a uniformly most powerful test. 
Now, the idea to generalize the Neyman-Pearson test to composite testing problems 
is to obtain first two candidates (or representatives) $\hat{\theta}_i$ of $\Theta_i$, and $i=0,1$ respectively, 
e.g.\ by maximum-likelihood estimation, 
and then to perform a Neyman-Pearson test using the computed candidates $\hat{\theta}_0$ and $\hat{\theta}_1$ 
in the definition of the test statistic. 
In case that an ML-estimation is used to determine $\hat{\theta}_0$ and $\hat{\theta}_1$, 
the resulting test is called \emph{Likelihood Ratio Test} (LR test). Although there are in general no theoretical guarantees concerning the power 
of LR tests, they usually perform very well in practice if the sample size used to estimate $\hat{\theta}_0$ and $\hat{\theta}_1$ is large enough.
This  is due to the fact that ML estimators are asymptotically efficient.	Several classical tests, e.g.\ one and two-sided $t$-tests, 
are either direct LR tests or equivalent to them.

In the sequel, we show how the above framework can be applied to our testing problem~\eqref{similarity_test}. 
First, let $x_1$ and $y_1$ be two single pixels for which we want to test whether they are realizations of the same distribution with unknown common parameter.
The LR statistic reads as
\begin{align} \label{FR_stat}
\lambda(x_1,y_1) 
&= \frac{\sup\limits_{\theta\in \HH_0}  
	\bigl\{{\cal L}(\theta|x_1,y_1)\bigr\}}{\sup\limits_{\theta }\bigl\{{\cal L}(\theta|x_1,y_1)\bigr\}}
 =\frac{\sup\limits_{\theta}\bigl\{{\cal L}(\theta|x_1){\cal L}(\theta|y_1)\bigr\}}{\sup\limits_{\theta}\bigl\{{\cal L}(\theta|x_1)\bigr\}
	\sup\limits_{\theta}\bigl\{{\cal L}(\theta|y_1)\bigr\}},
\end{align}
where in our situation ${\cal L} (\theta|x_1,y_1)$ denotes the likelihood function with respect to $\mu$ while $\nu$ and $\sigma$ are assumed to be known, 
and the notation $\theta\in \HH_0$ means that the supremum is taken over those parameters $\theta$ fulfilling $\HH_0$.
We use this statistic as similarity measure, i.e., 
$$S(x_1,y_1) \coloneqq \lambda(x_1,y_1).$$
More generally, since we assume the noise to affect each pixel in an independent and identical way, 
the similarity of two patches $p= (x_1,\ldots,x_t)$ and $q= (y_1,\ldots,y_t)$ is obtained as the product of the similarity of its pixels
\begin{equation*}
S(p,q) \coloneqq \prod_{i=1}^t S(x_i,y_i).
\end{equation*}
In case of the Student-$t$ distribution with $\nu>0$ degrees of freedom, the  similarity measure between two patches $p=(p_1,\ldots,p_t)$ 
and $q = (q_1,\ldots,q_t)$ can be computed as 
\begin{equation*}
S(p,q) = \prod_{i=1}^{t} \left(1+\frac{1}{\nu}\left(\frac{p_i-q_i}{2\sigma}\right)^2 \right)^{-(\nu+1)}.
\end{equation*}
In practice, we take a scaled  logarithm of $S$ in order to avoid numerical instabilities, resulting in the distance measure
\begin{equation*}
d(p,q) = \sum_{i=1}^t \log\left(\nu + \left(\frac{p_i-q_i}{2\sigma}\right)^2\right).
\end{equation*}

\subsection{Numerical Examples}

\paragraph{Cauchy Noise}
As mentioned in the introduction, the initial motivation for this work was the consideration of   
Cauchy noise in~\cite{LPS18,MDHY18} and thus we tested our approach on images corrupted by additive Cauchy noise ($\nu = 1$) with noise level $\sigma = 10$. 
Since the noise level is very high, we chose $n=50$ patches of size $5\times 5$ for the denoising. 
It turns out that the differences in terms of PSNR or SSIM compared to the current state of the art method~\cite{LPS18} 
are small and nearly not visible in images with much textured regions. 
However, the improvement is large in images with many constant or smoothly varying areas. 
This becomes in particular apparent in case of the test image given in Figure~\ref{Fig:examples_gamma_10}. 
Here, the top row displays the original image (left) together with its noisy version (right). The bottom row shows 
from left to right the results obtained using the variational method presented in~\cite{MDHY18}, the pixelwise~\cite{LPS18} 
and the patchwise nonlocal myriad filter. While in case of the variational method some of the outliers remain, the result of~\cite{LPS18} is rather grainy, 
which is much improved by our new approach. This is also reflected in the corresponding PSNR and SSIM values stated in the captions of the figure.
\begin{figure}[htbp]
	\centering	
	\begin{subfigure}[t]{0.32\textwidth}
		\centering
		\tikzsetnextfilename{test_orig_10}
		\begin{tikzpicture}[spy using outlines={%
			every spy in node/.append style={draw = none},%
			every spy on node/.append style={white},connect spies}]
		\node[anchor=south west,inner sep=0pt] (image) at (0,0){\includegraphics[width=.98\textwidth]{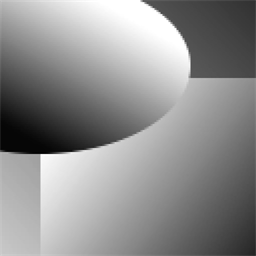}};
		\begin{scope}[x={($ (image.south east) - (image.south west) $ )},%
		y={( $ (image.north west) - (image.south west)$ )},%
		shift={(image.south west)}]
		\node (spy1) at (100/512,200/512) {};
		\coordinate (spy1to) at (0.25,-0.225);
		\spy [rectangle,width=0.45\textwidth, height=.35\textwidth,draw=white,magnification=1.6] on (spy1) in node at (spy1to);
		\node (spy2) at (370/512,120/512) {};
		\coordinate (spy2to) at (0.75,-0.225);
		\spy [rectangle,width=0.45\textwidth, height=.35\textwidth,draw=white,magnification=1.6] on (spy2) in node at (spy2to);
		\end{scope}
		\end{tikzpicture} 
		\caption{PSNR: $+\infty$, SSIM: $1$ }
	\end{subfigure}
	\begin{subfigure}[t]{0.32\textwidth}
		\centering
		\tikzsetnextfilename{test_noisy}
		\begin{tikzpicture}[spy using outlines={%
			every spy in node/.append style={draw = none},%
			every spy on node/.append style={white},connect spies}]
		\node[anchor=south west,inner sep=0pt] (image) at (0,0){\includegraphics[width=.98\textwidth]{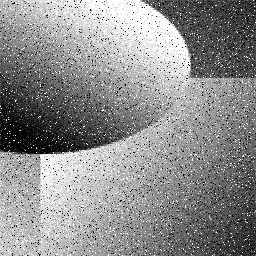}};
		\begin{scope}[x={($ (image.south east) - (image.south west) $ )},%
		y={( $ (image.north west) - (image.south west)$ )},%
		shift={(image.south west)}]
		\node (spy1) at (100/512,200/512) {};
		\coordinate (spy1to) at (0.25,-0.225);
		\spy [rectangle,width=0.45\textwidth, height=.35\textwidth,draw=white,magnification=1.6] on (spy1) in node at (spy1to);
		\node (spy2) at (370/512,120/512) {};
		\coordinate (spy2to) at (0.75,-0.225);
		\spy [rectangle,width=0.45\textwidth, height=.35\textwidth,draw=white,magnification=1.6] on (spy2) in node at (spy2to);
		\end{scope}
		\end{tikzpicture} 
		\caption{PSNR: $16.3270$, SSIM:$0.0870$ }
	\end{subfigure}	\\
	
	\vspace*{0.15cm}
	\begin{subfigure}[t]{0.32\textwidth}
		\centering
		\tikzsetnextfilename{test_jinmei_10}
		\begin{tikzpicture}[spy using outlines={%
			every spy in node/.append style={draw = none},%
			every spy on node/.append style={white},connect spies}]
		\node[anchor=south west,inner sep=0pt] (image) at (0,0){\includegraphics[width=.98\textwidth]{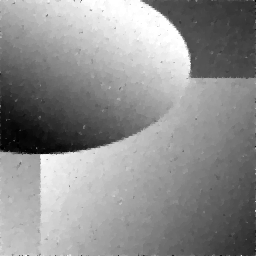}};
		\begin{scope}[x={($ (image.south east) - (image.south west) $ )},%
		y={( $ (image.north west) - (image.south west)$ )},%
		shift={(image.south west)}]
		\node (spy1) at (100/512,200/512) {};
		\coordinate (spy1to) at (0.25,-0.225);
		\spy [rectangle,width=0.45\textwidth, height=.35\textwidth,draw=white,magnification=1.6] on (spy1) in node at (spy1to);
		\node (spy2) at (370/512,120/512) {};
		\coordinate (spy2to) at (0.75,-0.225);
		\spy [rectangle,width=0.45\textwidth, height=.35\textwidth,draw=white,magnification=1.6] on (spy2) in node at (spy2to);
		\end{scope}
		\end{tikzpicture} 
		\caption{PSNR: $33.5452$, SSIM: $0.8651$}
	\end{subfigure}
	\begin{subfigure}[t]{0.32\textwidth}
		\centering
		\tikzsetnextfilename{test_our_10}
		\begin{tikzpicture}[spy using outlines={%
			every spy in node/.append style={draw = none},%
			every spy on node/.append style={white},connect spies}]
		\node[anchor=south west,inner sep=0pt] (image) at (0,0){\includegraphics[width=.98\textwidth]{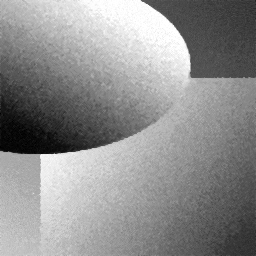}};
		\begin{scope}[x={($ (image.south east) - (image.south west) $ )},%
		y={( $ (image.north west) - (image.south west)$ )},%
		shift={(image.south west)}]
		\node (spy1) at (100/512,200/512) {};
		\coordinate (spy1to) at (0.25,-0.225);
		\spy [rectangle,width=0.45\textwidth, height=.35\textwidth,draw=white,magnification=1.6] on (spy1) in node at (spy1to);
		\node (spy2) at (370/512,120/512) {};
		\coordinate (spy2to) at (0.75,-0.225);
		\spy [rectangle,width=0.45\textwidth, height=.35\textwidth,draw=white,magnification=1.6] on (spy2) in node at (spy2to);
		\end{scope}
		\end{tikzpicture} 
		\caption{PSNR: $33.9144$, SSIM:$0.8410$ }
	\end{subfigure}	
	\begin{subfigure}[t]{0.32\textwidth}
		\centering
		\tikzsetnextfilename{test_our_10_new}
		\begin{tikzpicture}[spy using outlines={%
			every spy in node/.append style={draw = none},%
			every spy on node/.append style={white},connect spies}]
		\node[anchor=south west,inner sep=0pt] (image) at (0,0){\includegraphics[width=.98\textwidth]{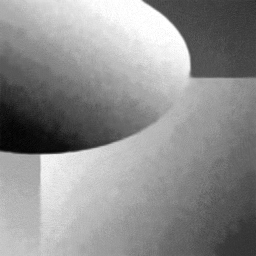}};
		\begin{scope}[x={($ (image.south east) - (image.south west) $ )},%
		y={( $ (image.north west) - (image.south west)$ )},%
		shift={(image.south west)}]
		\node (spy1) at (100/512,200/512) {};
		\coordinate (spy1to) at (0.25,-0.225);
		\spy [rectangle,width=0.45\textwidth, height=.35\textwidth,draw=white,magnification=1.6] on (spy1) in node at (spy1to);
		\node (spy2) at (370/512,120/512) {};
		\coordinate (spy2to) at (0.75,-0.225);
		\spy [rectangle,width=0.45\textwidth, height=.35\textwidth,draw=white,magnification=1.6] on (spy2) in node at (spy2to);
		\end{scope}
		\end{tikzpicture} 
		\caption{PSNR: $37.7263$, SSIM: $0.9398$ }
	\end{subfigure}		
	\caption[]{Denoising of the test image (top left) corrupted with additive Cauchy noise ($\nu=1$, $\sigma= 10$) (top right) 
	using the methods proposed in~\cite{MDHY18} (bottom left),~\cite{LPS18} (bottom middle) and our nonlocal GMMF (bottom right).}
	\label{Fig:examples_gamma_10}	
\end{figure}
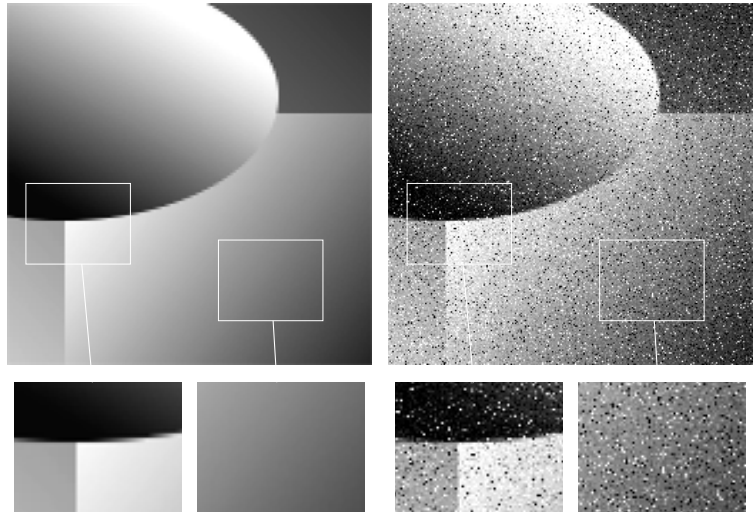
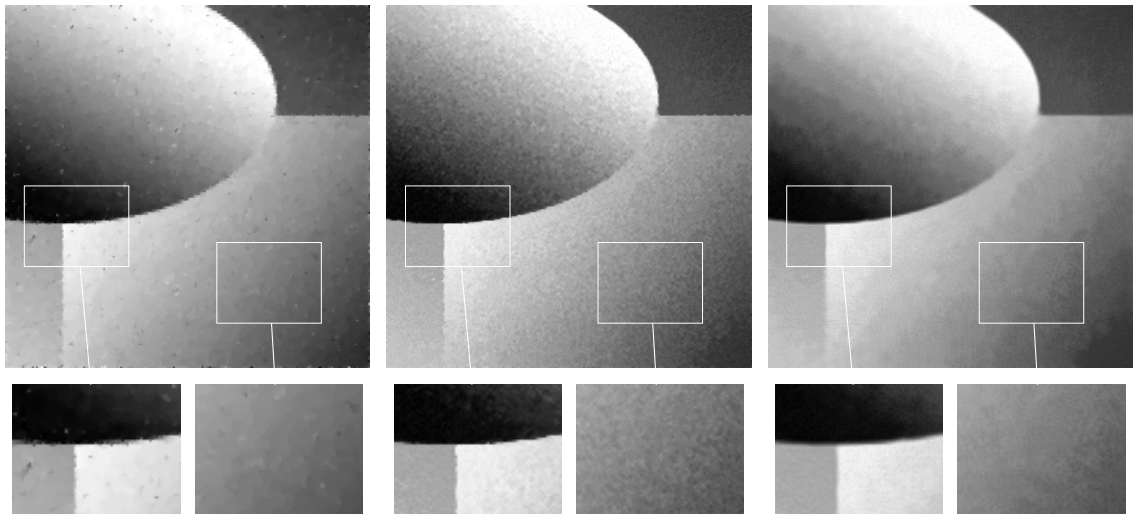

As to be expected, the larger the patch size, the smoother is the resulting image. While this is desirable in constant  
or smoothly varying regions, fine structures and details are lost when the patch size is too large. 
Thus, a natural idea is to adapt the filter to the local structure of the image, that is, using a multivariate myriad filter 
in constant areas and a pixelwise myriad filter in regions with many details. We tried the following naive approach: 
First, we computed the one-dimensional as well as the multivariate myriad filtered images. 
Then, for each pixel of the image we decide whether it belongs to a homogeneous area or not and restore the pixel accordingly. 
In order to detect constant regions we used the variance of the pixels in similar patches, which we expect to be low in regions without much details.   
Results of this approach are shown in Figures~\ref{Fig:examples_fused1} and~\ref{Fig:examples_fused2}, the corresponding PSNR and SSIM values 
are given in Table~\ref{Tab:fusing}, where we used  $5\times 5$ patches  a sample size of $n=40$ for all images.  
In all examples one nicely sees that with the adaptive approach both fine details as well as constant areas are very well reconstructed,  
which is not the case when applying only the one-dimensional or only the multivariate myriad filter.

\begin{figure}[htbp]
	\centering	
	\begin{subfigure}[t]{0.24\textwidth}
		\centering
		\tikzsetnextfilename{cameraman_orig}
		\begin{tikzpicture}[spy using outlines={%
			every spy in node/.append style={draw = none},%
			every spy on node/.append style={white},connect spies}]
		\node[anchor=south west,inner sep=0pt] (image) at (0,0){\includegraphics[width=.98\textwidth]{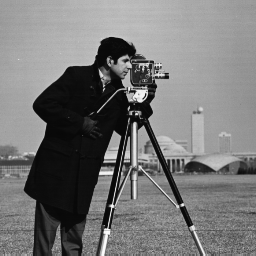}};
		\begin{scope}[x={($ (image.south east) - (image.south west) $ )},%
		y={( $ (image.north west) - (image.south west)$ )},%
		shift={(image.south west)}]
		\node (spy1) at (250/512,150/512) {};
		\coordinate (spy1to) at (0.25,-0.225);
		\spy [rectangle,width=0.45\textwidth, height=.35\textwidth,draw=white,magnification=2] on (spy1) in node at (spy1to);
		\node (spy2) at (420/512,385/512) {};
		\coordinate (spy2to) at (0.75,-0.225);
		\spy [rectangle,width=0.45\textwidth, height=.35\textwidth,draw=white,magnification=2] on (spy2) in node at (spy2to);
		\end{scope}
		\end{tikzpicture} 
	\end{subfigure}
	\begin{subfigure}[t]{0.24\textwidth}
		\centering
		\tikzsetnextfilename{cameraman_1d}
		\begin{tikzpicture}[spy using outlines={%
			every spy in node/.append style={draw = none},%
			every spy on node/.append style={white},connect spies}]
		\node[anchor=south west,inner sep=0pt] (image) at (0,0){\includegraphics[width=.98\textwidth]{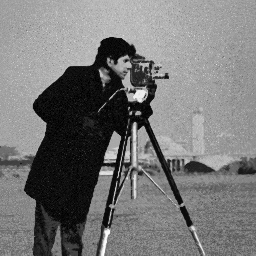}};
		\begin{scope}[x={($ (image.south east) - (image.south west) $ )},%
		y={( $ (image.north west) - (image.south west)$ )},%
		shift={(image.south west)}]
		\node (spy1) at (250/512,150/512) {};
		\coordinate (spy1to) at (0.25,-0.225);
		\spy [rectangle,width=0.45\textwidth, height=.35\textwidth,draw=white,magnification=2] on (spy1) in node at (spy1to);
		\node (spy2) at (420/512,385/512) {};
			\coordinate (spy2to) at (0.75,-0.225);
			\spy [rectangle,width=0.45\textwidth, height=.35\textwidth,draw=white,magnification=2] on (spy2) in node at (spy2to);
		\end{scope}
		\end{tikzpicture} 
	\end{subfigure}
	\begin{subfigure}[t]{0.24\textwidth}
		\centering
		\tikzsetnextfilename{cameraman_d}
		\begin{tikzpicture}[spy using outlines={%
			every spy in node/.append style={draw = none},%
			every spy on node/.append style={white},connect spies}]
		\node[anchor=south west,inner sep=0pt] (image) at (0,0){\includegraphics[width=.98\textwidth]{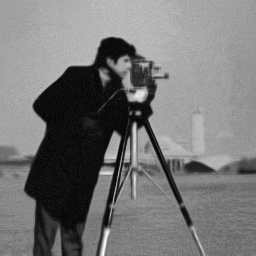}};
		\begin{scope}[x={($ (image.south east) - (image.south west) $ )},%
		y={( $ (image.north west) - (image.south west)$ )},%
		shift={(image.south west)}]
		\node (spy1) at (250/512,150/512) {};
		\coordinate (spy1to) at (0.25,-0.225);
		\spy [rectangle,width=0.45\textwidth, height=.35\textwidth,draw=white,magnification=2] on (spy1) in node at (spy1to);
		\node (spy2) at (420/512,385/512) {};
		\coordinate (spy2to) at (0.75,-0.225);
		\spy [rectangle,width=0.45\textwidth, height=.35\textwidth,draw=white,magnification=2] on (spy2) in node at (spy2to);
		\end{scope}
		\end{tikzpicture} 
	\end{subfigure}	
	\begin{subfigure}[t]{0.24\textwidth}
		\centering
		\tikzsetnextfilename{cameraman_fused}
		\begin{tikzpicture}[spy using outlines={%
			every spy in node/.append style={draw = none},%
			every spy on node/.append style={white},connect spies}]
		\node[anchor=south west,inner sep=0pt] (image) at (0,0){\includegraphics[width=.98\textwidth]{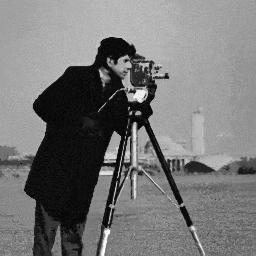}};
		\begin{scope}[x={($ (image.south east) - (image.south west) $ )},%
		y={( $ (image.north west) - (image.south west)$ )},%
		shift={(image.south west)}]
		\node (spy1) at (250/512,150/512) {};
		\coordinate (spy1to) at (0.25,-0.225);
		\spy [rectangle,width=0.45\textwidth, height=.35\textwidth,draw=white,magnification=2] on (spy1) in node at (spy1to);
		\node (spy2) at (420/512,385/512) {};
		\coordinate (spy2to) at (0.75,-0.225);
		\spy [rectangle,width=0.45\textwidth, height=.35\textwidth,draw=white,magnification=2] on (spy2) in node at (spy2to);
		\end{scope}
		\end{tikzpicture} 
	\end{subfigure}
	
	\vspace*{0.15cm}
	
	\begin{subfigure}[t]{0.24\textwidth}
		\centering
		\tikzsetnextfilename{boat_orig}
		\begin{tikzpicture}[spy using outlines={%
			every spy in node/.append style={draw = none},%
			every spy on node/.append style={white},connect spies}]
		\node[anchor=south west,inner sep=0pt] (image) at (0,0){\includegraphics[width=.98\textwidth]{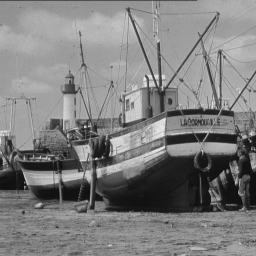}};
		\begin{scope}[x={($ (image.south east) - (image.south west) $ )},%
		y={( $ (image.north west) - (image.south west)$ )},%
		shift={(image.south west)}]
		\node (spy1) at (90/512,430/512) {};
		\coordinate (spy1to) at (0.25,-0.225);
		\spy [rectangle,width=0.45\textwidth, height=.35\textwidth,draw=white,magnification=2] on (spy1) in node at (spy1to);
		\node (spy2) at (400/512,270/512) {};
		\coordinate (spy2to) at (0.75,-0.225);
		\spy [rectangle,width=0.45\textwidth, height=.35\textwidth,draw=white,magnification=1.6] on (spy2) in node at (spy2to);
		\end{scope}
		\end{tikzpicture} 
	\end{subfigure}
	\begin{subfigure}[t]{0.24\textwidth}
		\centering
		\tikzsetnextfilename{boat_1d}
		\begin{tikzpicture}[spy using outlines={%
			every spy in node/.append style={draw = none},%
			every spy on node/.append style={white},connect spies}]
		\node[anchor=south west,inner sep=0pt] (image) at (0,0){\includegraphics[width=.98\textwidth]{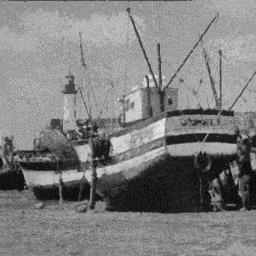}};
		\begin{scope}[x={($ (image.south east) - (image.south west) $ )},%
		y={( $ (image.north west) - (image.south west)$ )},%
		shift={(image.south west)}]
		\node (spy1) at (90/512,430/512) {};
		\coordinate (spy1to) at (0.25,-0.225);
		\spy [rectangle,width=0.45\textwidth, height=.35\textwidth,draw=white,magnification=2] on (spy1) in node at (spy1to);
		\node (spy2) at (400/512,270/512) {};
		\coordinate (spy2to) at (0.75,-0.225);
		\spy [rectangle,width=0.45\textwidth, height=.35\textwidth,draw=white,magnification=1.6] on (spy2) in node at (spy2to);
		\end{scope}
		\end{tikzpicture} 
	\end{subfigure}
	\begin{subfigure}[t]{0.24\textwidth}
		\centering
		\tikzsetnextfilename{boat_d}
		\begin{tikzpicture}[spy using outlines={%
			every spy in node/.append style={draw = none},%
			every spy on node/.append style={white},connect spies}]
		\node[anchor=south west,inner sep=0pt] (image) at (0,0){\includegraphics[width=.98\textwidth]{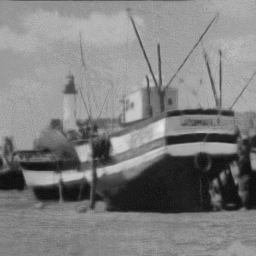}};
		\begin{scope}[x={($ (image.south east) - (image.south west) $ )},%
		y={( $ (image.north west) - (image.south west)$ )},%
		shift={(image.south west)}]
			\node (spy1) at (90/512,430/512) {};
			\coordinate (spy1to) at (0.25,-0.225);
			\spy [rectangle,width=0.45\textwidth, height=.35\textwidth,draw=white,magnification=2] on (spy1) in node at (spy1to);
			\node (spy2) at (400/512,270/512) {};
			\coordinate (spy2to) at (0.75,-0.225);
			\spy [rectangle,width=0.45\textwidth, height=.35\textwidth,draw=white,magnification=1.6] on (spy2) in node at (spy2to);
		\end{scope}
		\end{tikzpicture} 
	\end{subfigure}	
	\begin{subfigure}[t]{0.24\textwidth}
		\centering
		\tikzsetnextfilename{boat_fused}
		\begin{tikzpicture}[spy using outlines={%
			every spy in node/.append style={draw = none},%
			every spy on node/.append style={white},connect spies}]
		\node[anchor=south west,inner sep=0pt] (image) at (0,0){\includegraphics[width=.98\textwidth]{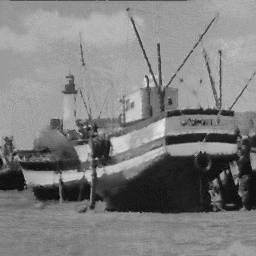}};
		\begin{scope}[x={($ (image.south east) - (image.south west) $ )},%
		y={( $ (image.north west) - (image.south west)$ )},%
		shift={(image.south west)}]
		\node (spy1) at (90/512,430/512) {};
		\coordinate (spy1to) at (0.25,-0.225);
		\spy [rectangle,width=0.45\textwidth, height=.35\textwidth,draw=white,magnification=2] on (spy1) in node at (spy1to);
		\node (spy2) at (400/512,270/512) {};
		\coordinate (spy2to) at (0.75,-0.225);
		\spy [rectangle,width=0.45\textwidth, height=.35\textwidth,draw=white,magnification=1.6] on (spy2) in node at (spy2to);
		\end{scope}
		\end{tikzpicture} 
	\end{subfigure}
	
	\vspace*{0.15cm}

	\begin{subfigure}[t]{0.24\textwidth}
		\centering
		\tikzsetnextfilename{house_orig}
		\begin{tikzpicture}[spy using outlines={%
			every spy in node/.append style={draw = none},%
			every spy on node/.append style={white},connect spies}]
		\node[anchor=south west,inner sep=0pt] (image) at (0,0){\includegraphics[width=.98\textwidth]{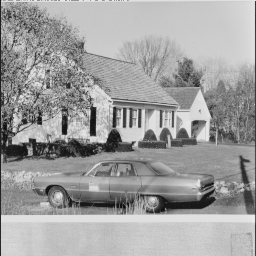}};
		\begin{scope}[x={($ (image.south east) - (image.south west) $ )},%
		y={( $ (image.north west) - (image.south west)$ )},%
		shift={(image.south west)}]
		\node (spy1) at (130/512,280/512) {};
		\coordinate (spy1to) at (0.25,-0.225);
		\spy [rectangle,width=0.45\textwidth, height=.35\textwidth,draw=white,magnification=1.6] on (spy1) in node at (spy1to);
		\node (spy2) at (370/512,400/512) {};
		\coordinate (spy2to) at (0.75,-0.225);
		\spy [rectangle,width=0.45\textwidth, height=.35\textwidth,draw=white,magnification=1.6] on (spy2) in node at (spy2to);
		\end{scope}
		\end{tikzpicture} 
	\end{subfigure}
	\begin{subfigure}[t]{0.24\textwidth}
		\centering
		\tikzsetnextfilename{house_1d}
		\begin{tikzpicture}[spy using outlines={%
			every spy in node/.append style={draw = none},%
			every spy on node/.append style={white},connect spies}]
		\node[anchor=south west,inner sep=0pt] (image) at (0,0){\includegraphics[width=.98\textwidth]{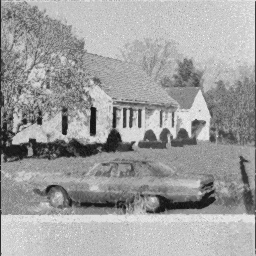}};
		\begin{scope}[x={($ (image.south east) - (image.south west) $ )},%
		y={( $ (image.north west) - (image.south west)$ )},%
		shift={(image.south west)}]
			\node (spy1) at (130/512,280/512) {};
			\coordinate (spy1to) at (0.25,-0.225);
			\spy [rectangle,width=0.45\textwidth, height=.35\textwidth,draw=white,magnification=1.6] on (spy1) in node at (spy1to);
			\node (spy2) at (370/512,400/512) {};
			\coordinate (spy2to) at (0.75,-0.225);
			\spy [rectangle,width=0.45\textwidth, height=.35\textwidth,draw=white,magnification=1.6] on (spy2) in node at (spy2to);
		\end{scope}
		\end{tikzpicture} 
	\end{subfigure}
	\begin{subfigure}[t]{0.24\textwidth}
		\centering
		\tikzsetnextfilename{house_d}
		\begin{tikzpicture}[spy using outlines={%
			every spy in node/.append style={draw = none},%
			every spy on node/.append style={white},connect spies}]
		\node[anchor=south west,inner sep=0pt] (image) at (0,0){\includegraphics[width=.98\textwidth]{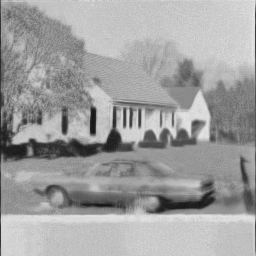}};
		\begin{scope}[x={($ (image.south east) - (image.south west) $ )},%
		y={( $ (image.north west) - (image.south west)$ )},%
		shift={(image.south west)}]
			\node (spy1) at (130/512,280/512) {};
			\coordinate (spy1to) at (0.25,-0.225);
			\spy [rectangle,width=0.45\textwidth, height=.35\textwidth,draw=white,magnification=1.6] on (spy1) in node at (spy1to);
			\node (spy2) at (370/512,400/512) {};
			\coordinate (spy2to) at (0.75,-0.225);
			\spy [rectangle,width=0.45\textwidth, height=.35\textwidth,draw=white,magnification=1.6] on (spy2) in node at (spy2to);
		\end{scope}
		\end{tikzpicture} 
	\end{subfigure}	
	\begin{subfigure}[t]{0.24\textwidth}
		\centering
		\tikzsetnextfilename{house_fused}
		\begin{tikzpicture}[spy using outlines={%
			every spy in node/.append style={draw = none},%
			every spy on node/.append style={white},connect spies}]
		\node[anchor=south west,inner sep=0pt] (image) at (0,0){\includegraphics[width=.98\textwidth]{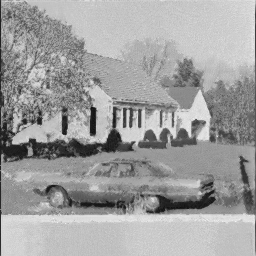}};
		\begin{scope}[x={($ (image.south east) - (image.south west) $ )},%
		y={( $ (image.north west) - (image.south west)$ )},%
		shift={(image.south west)}]
		\node (spy1) at (130/512,280/512) {};
		\coordinate (spy1to) at (0.25,-0.225);
		\spy [rectangle,width=0.45\textwidth, height=.35\textwidth,draw=white,magnification=1.6] on (spy1) in node at (spy1to);
		\node (spy2) at (370/512,400/512) {};
		\coordinate (spy2to) at (0.75,-0.225);
		\spy [rectangle,width=0.45\textwidth, height=.35\textwidth,draw=white,magnification=1.6] on (spy2) in node at (spy2to);
		\end{scope}
		\end{tikzpicture} 
	\end{subfigure}	
	\caption[]{Denoising of images (first column) corrupted with additive Cauchy noise ($\nu=1$, $\sigma = 10$) 
	using a one-dimensional myriad filter~\cite{LPS18} (second column), the multivariate myriad filter (third column) and a combination of both (fourth column).}
	\label{Fig:examples_fused1}	
\end{figure}

\begin{figure}[htbp]
	\centering

	\begin{subfigure}[t]{0.24\textwidth}
		\centering
		\tikzsetnextfilename{parrot_orig}
		\begin{tikzpicture}[spy using outlines={%
			every spy in node/.append style={draw = none},%
			every spy on node/.append style={white},connect spies}]
		\node[anchor=south west,inner sep=0pt] (image) at (0,0){\includegraphics[width=.98\textwidth]{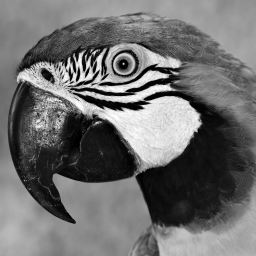}};
		\begin{scope}[x={($ (image.south east) - (image.south west) $ )},%
		y={( $ (image.north west) - (image.south west)$ )},%
		shift={(image.south west)}]
		\node (spy1) at (120/512,370/512) {};
		\coordinate (spy1to) at (0.25,-0.225);
		\spy [rectangle,width=0.45\textwidth, height=.35\textwidth,draw=white,magnification=1.6] on (spy1) in node at (spy1to);
		\node (spy2) at (330/512,70/512) {};
		\coordinate (spy2to) at (0.75,-0.225);
		\spy [rectangle,width=0.45\textwidth, height=.35\textwidth,draw=white,magnification=1.6] on (spy2) in node at (spy2to);
		\end{scope}
		\end{tikzpicture} 
	\end{subfigure}
	\begin{subfigure}[t]{0.24\textwidth}
		\centering
		\tikzsetnextfilename{parrot_1d}
		\begin{tikzpicture}[spy using outlines={%
			every spy in node/.append style={draw = none},%
			every spy on node/.append style={white},connect spies}]
		\node[anchor=south west,inner sep=0pt] (image) at (0,0){\includegraphics[width=.98\textwidth]{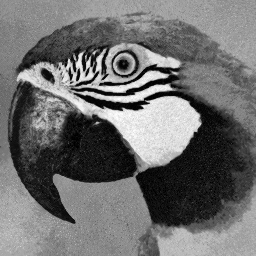}};
		\begin{scope}[x={($ (image.south east) - (image.south west) $ )},%
		y={( $ (image.north west) - (image.south west)$ )},%
		shift={(image.south west)}]
		\node (spy1) at (120/512,370/512) {};
		\coordinate (spy1to) at (0.25,-0.225);
		\spy [rectangle,width=0.45\textwidth, height=.35\textwidth,draw=white,magnification=1.6] on (spy1) in node at (spy1to);
		\node (spy2) at (330/512,70/512) {};
		\coordinate (spy2to) at (0.75,-0.225);
		\spy [rectangle,width=0.45\textwidth, height=.35\textwidth,draw=white,magnification=1.6] on (spy2) in node at (spy2to);
		\end{scope}
		\end{tikzpicture} 
	\end{subfigure}
	\begin{subfigure}[t]{0.24\textwidth}
		\centering
		\tikzsetnextfilename{parrot_d}
		\begin{tikzpicture}[spy using outlines={%
			every spy in node/.append style={draw = none},%
			every spy on node/.append style={white},connect spies}]
		\node[anchor=south west,inner sep=0pt] (image) at (0,0){\includegraphics[width=.98\textwidth]{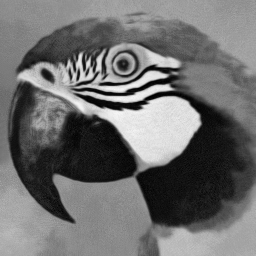}};
		\begin{scope}[x={($ (image.south east) - (image.south west) $ )},%
		y={( $ (image.north west) - (image.south west)$ )},%
		shift={(image.south west)}]
		\node (spy1) at (120/512,370/512) {};
		\coordinate (spy1to) at (0.25,-0.225);
		\spy [rectangle,width=0.45\textwidth, height=.35\textwidth,draw=white,magnification=1.6] on (spy1) in node at (spy1to);
		\node (spy2) at (330/512,70/512) {};
		\coordinate (spy2to) at (0.75,-0.225);
		\spy [rectangle,width=0.45\textwidth, height=.35\textwidth,draw=white,magnification=1.6] on (spy2) in node at (spy2to);
		\end{scope}
		\end{tikzpicture} 
	\end{subfigure}	
	\begin{subfigure}[t]{0.24\textwidth}
		\centering
		\tikzsetnextfilename{parrot_fused}
		\begin{tikzpicture}[spy using outlines={%
			every spy in node/.append style={draw = none},%
			every spy on node/.append style={white},connect spies}]
		\node[anchor=south west,inner sep=0pt] (image) at (0,0){\includegraphics[width=.98\textwidth]{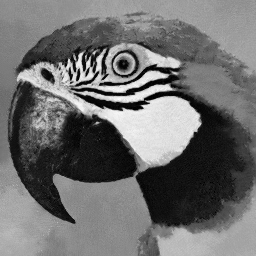}};
		\begin{scope}[x={($ (image.south east) - (image.south west) $ )},%
		y={( $ (image.north west) - (image.south west)$ )},%
		shift={(image.south west)}]
		\node (spy1) at (120/512,370/512) {};
		\coordinate (spy1to) at (0.25,-0.225);
		\spy [rectangle,width=0.45\textwidth, height=.35\textwidth,draw=white,magnification=1.6] on (spy1) in node at (spy1to);
		\node (spy2) at (330/512,70/512) {};
		\coordinate (spy2to) at (0.75,-0.225);
		\spy [rectangle,width=0.45\textwidth, height=.35\textwidth,draw=white,magnification=1.6] on (spy2) in node at (spy2to);
		\end{scope}
		\end{tikzpicture} 
	\end{subfigure}
		
		\vspace*{0.15cm}
	
	\begin{subfigure}[t]{0.24\textwidth}
		\centering
		\tikzsetnextfilename{plane_orig}
		\begin{tikzpicture}[spy using outlines={%
			every spy in node/.append style={draw = none},%
			every spy on node/.append style={white},connect spies}]
		\node[anchor=south west,inner sep=0pt] (image) at (0,0){\includegraphics[width=.98\textwidth]{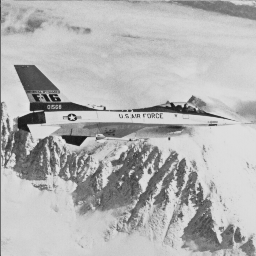}};
		\begin{scope}[x={($ (image.south east) - (image.south west) $ )},%
		y={( $ (image.north west) - (image.south west)$ )},%
		shift={(image.south west)}]
		\node (spy1) at (130/512,290/512) {};
		\coordinate (spy1to) at (0.25,-0.225);
		\spy [rectangle,width=0.45\textwidth, height=.35\textwidth,draw=white,magnification=1.6] on (spy1) in node at (spy1to);
		\node (spy2) at (370/512,400/512) {};
		\coordinate (spy2to) at (0.75,-0.225);
		\spy [rectangle,width=0.45\textwidth, height=.35\textwidth,draw=white,magnification=1.6] on (spy2) in node at (spy2to);
		\end{scope}
		\end{tikzpicture} 
	\end{subfigure}
	\begin{subfigure}[t]{0.24\textwidth}
		\centering
		\tikzsetnextfilename{plane_1d}
		\begin{tikzpicture}[spy using outlines={%
			every spy in node/.append style={draw = none},%
			every spy on node/.append style={white},connect spies}]
		\node[anchor=south west,inner sep=0pt] (image) at (0,0){\includegraphics[width=.98\textwidth]{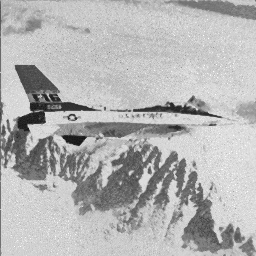}};
		\begin{scope}[x={($ (image.south east) - (image.south west) $ )},%
		y={( $ (image.north west) - (image.south west)$ )},%
		shift={(image.south west)}]
		\node (spy1) at (130/512,290/512) {};
		\coordinate (spy1to) at (0.25,-0.225);
		\spy [rectangle,width=0.45\textwidth, height=.35\textwidth,draw=white,magnification=1.6] on (spy1) in node at (spy1to);
		\node (spy2) at (370/512,400/512) {};
		\coordinate (spy2to) at (0.75,-0.225);
		\spy [rectangle,width=0.45\textwidth, height=.35\textwidth,draw=white,magnification=1.6] on (spy2) in node at (spy2to);
		\end{scope}
		\end{tikzpicture} 
	\end{subfigure}
	\begin{subfigure}[t]{0.24\textwidth}
		\centering
		\tikzsetnextfilename{plane_d}
		\begin{tikzpicture}[spy using outlines={%
			every spy in node/.append style={draw = none},%
			every spy on node/.append style={white},connect spies}]
		\node[anchor=south west,inner sep=0pt] (image) at (0,0){\includegraphics[width=.98\textwidth]{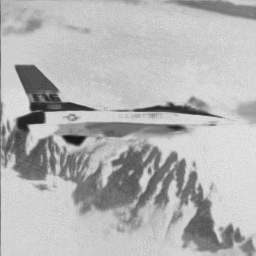}};
		\begin{scope}[x={($ (image.south east) - (image.south west) $ )},%
		y={( $ (image.north west) - (image.south west)$ )},%
		shift={(image.south west)}]
		\node (spy1) at (130/512,290/512) {};
		\coordinate (spy1to) at (0.25,-0.225);
		\spy [rectangle,width=0.45\textwidth, height=.35\textwidth,draw=white,magnification=1.6] on (spy1) in node at (spy1to);
		\node (spy2) at (370/512,400/512) {};
		\coordinate (spy2to) at (0.75,-0.225);
		\spy [rectangle,width=0.45\textwidth, height=.35\textwidth,draw=white,magnification=1.6] on (spy2) in node at (spy2to);
		\end{scope}
		\end{tikzpicture} 
	\end{subfigure}	
	\begin{subfigure}[t]{0.24\textwidth}
		\centering
		\tikzsetnextfilename{plane_fused}
		\begin{tikzpicture}[spy using outlines={%
			every spy in node/.append style={draw = none},%
			every spy on node/.append style={white},connect spies}]
		\node[anchor=south west,inner sep=0pt] (image) at (0,0){\includegraphics[width=.98\textwidth]{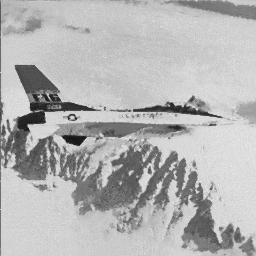}};
		\begin{scope}[x={($ (image.south east) - (image.south west) $ )},%
		y={( $ (image.north west) - (image.south west)$ )},%
		shift={(image.south west)}]
		\node (spy1) at (130/512,290/512) {};
		\coordinate (spy1to) at (0.25,-0.225);
		\spy [rectangle,width=0.45\textwidth, height=.35\textwidth,draw=white,magnification=1.6] on (spy1) in node at (spy1to);
		\node (spy2) at (370/512,400/512) {};
		\coordinate (spy2to) at (0.75,-0.225);
		\spy [rectangle,width=0.45\textwidth, height=.35\textwidth,draw=white,magnification=1.6] on (spy2) in node at (spy2to);
		\end{scope}
		\end{tikzpicture} 
	\end{subfigure}

	\vspace*{0.15cm}
	\begin{subfigure}[t]{0.24\textwidth}
		\centering
		\tikzsetnextfilename{leopard_orig}
		\begin{tikzpicture}[spy using outlines={%
			every spy in node/.append style={draw = none},%
			every spy on node/.append style={white},connect spies}]
		\node[anchor=south west,inner sep=0pt] (image) at (0,0){\includegraphics[width=.98\textwidth]{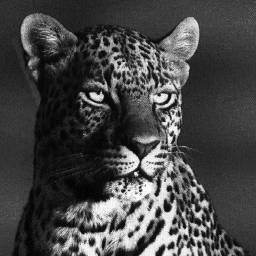}};
		\begin{scope}[x={($ (image.south east) - (image.south west) $ )},%
		y={( $ (image.north west) - (image.south west)$ )},%
		shift={(image.south west)}]
		\node (spy1) at (280/512,170/512) {};
		\coordinate (spy1to) at (0.25,-0.2);
		\spy [rectangle,width=0.45\textwidth, height=.3\textwidth,draw=white,magnification=1.6] on (spy1) in node at (spy1to);
		\node (spy2) at (390/512,390/512) {};
		\coordinate (spy2to) at (0.75,-0.2);
		\spy [rectangle,width=0.45\textwidth, height=.3\textwidth,draw=white,magnification=1.6] on (spy2) in node at (spy2to);
		\end{scope}
		\end{tikzpicture} 
	\end{subfigure}
	\begin{subfigure}[t]{0.24\textwidth}
		\centering
		\tikzsetnextfilename{leopard_1d}
		\begin{tikzpicture}[spy using outlines={%
			every spy in node/.append style={draw = none},%
			every spy on node/.append style={white},connect spies}]
		\node[anchor=south west,inner sep=0pt] (image) at (0,0){\includegraphics[width=.98\textwidth]{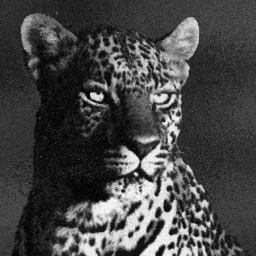}};
		\begin{scope}[x={($ (image.south east) - (image.south west) $ )},%
		y={( $ (image.north west) - (image.south west)$ )},%
		shift={(image.south west)}]
		\node (spy1) at (280/512,170/512) {};
		\coordinate (spy1to) at (0.25,-0.2);
		\spy [rectangle,width=0.45\textwidth, height=.3\textwidth,draw=white,magnification=1.6] on (spy1) in node at (spy1to);
		\node (spy2) at (390/512,390/512) {};
		\coordinate (spy2to) at (0.75,-0.2);
		\spy [rectangle,width=0.45\textwidth, height=.3\textwidth,draw=white,magnification=1.6] on (spy2) in node at (spy2to);
		\end{scope}
		\end{tikzpicture} 
	\end{subfigure}
	\begin{subfigure}[t]{0.24\textwidth}
		\centering
		\tikzsetnextfilename{leopard_d}
		\begin{tikzpicture}[spy using outlines={%
			every spy in node/.append style={draw = none},%
			every spy on node/.append style={white},connect spies}]
		\node[anchor=south west,inner sep=0pt] (image) at (0,0){\includegraphics[width=.98\textwidth]{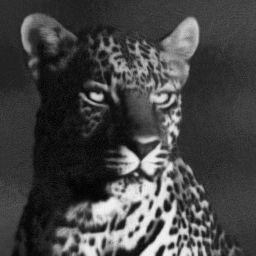}};
		\begin{scope}[x={($ (image.south east) - (image.south west) $ )},%
		y={( $ (image.north west) - (image.south west)$ )},%
		shift={(image.south west)}]
		\node (spy1) at (280/512,170/512) {};
		\coordinate (spy1to) at (0.25,-0.2);
		\spy [rectangle,width=0.45\textwidth, height=.3\textwidth,draw=white,magnification=1.6] on (spy1) in node at (spy1to);
		\node (spy2) at (390/512,390/512) {};
		\coordinate (spy2to) at (0.75,-0.2);
		\spy [rectangle,width=0.45\textwidth, height=.3\textwidth,draw=white,magnification=1.6] on (spy2) in node at (spy2to);
		\end{scope}
		\end{tikzpicture} 
	\end{subfigure}	
	\begin{subfigure}[t]{0.24\textwidth}
		\centering
		\tikzsetnextfilename{leopard_fused}
		\begin{tikzpicture}[spy using outlines={%
			every spy in node/.append style={draw = none},%
			every spy on node/.append style={white},connect spies}]
		\node[anchor=south west,inner sep=0pt] (image) at (0,0){\includegraphics[width=.98\textwidth]{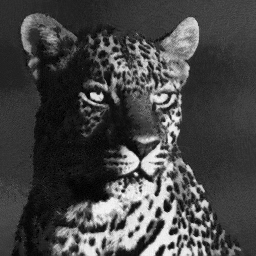}};
		\begin{scope}[x={($ (image.south east) - (image.south west) $ )},%
		y={( $ (image.north west) - (image.south west)$ )},%
		shift={(image.south west)}]
	\node (spy1) at (280/512,170/512) {};
	\coordinate (spy1to) at (0.25,-0.2);
	\spy [rectangle,width=0.45\textwidth, height=.3\textwidth,draw=white,magnification=1.6] on (spy1) in node at (spy1to);
	\node (spy2) at (390/512,390/512) {};
	\coordinate (spy2to) at (0.75,-0.2);
	\spy [rectangle,width=0.45\textwidth, height=.3\textwidth,draw=white,magnification=1.6] on (spy2) in node at (spy2to);
		\end{scope}
		\end{tikzpicture} 
	\end{subfigure}
	\caption[]{Denoising of images (first column) corrupted with additive Cauchy noise ($\nu=1$, $\sigma = 10$) 
	using a one-dimensional myriad filter~\cite{LPS18} (second column), the multivariate myriad filter (third column) and a combination of both (fourth column).}
	\label{Fig:examples_fused2}	
\end{figure}

\begin{table}[thb]
\begin{small}
	\centerline{
		\begin{tabular}{l@{\;}|c@{\;}c@{\;}c@{\;}c@{\;}c@{\;}c}
			\hline
			\hline
			 &  & PSNR &   &   & SSIM &  \\	\hline
			Image	 &  1d & multivariate &  adaptive  &  1d  & multivariate &  adaptive \\
			\hline
			\hline
			&&&$\sigma=10$& && \\
			\hline
			cameraman & 26.2721 & 25.5515  &  27.3726 \vline&0.7759 &  0.7376 & 0.7650 \\  
			boat & 26.2230 & 25.8486 & 26.6643  \vline & 0.7349&0.7209&0.7418\\  
				house &  25.1464 & 24.6241 &  25.8974 \vline & 0.7467  &0.7203 & 0.7689\\
			parrot & 26.4647 & 26.1817 &  27.4714 \vline & 0.7842  &0.7727&0.7858\\  
			plane &  25.8675  & 25.6023  & 27.1730 \vline  & 0.7694&0.7501& 0.7799 \\  
			leopard &  24.7906 &  24.1551 & 26.2640  \vline & 0.7535  &0.7348&0.7692\\  
			\hline				
		\end{tabular}
	}
\end{small}
\caption{Comparison of  PSNR and SSIM values of the one-dimensional, the multivariate and an adaptive myriad filter. }
\label{Tab:fusing}
\end{table}

\paragraph{Gaussian Noise}
For $\nu \to \infty$, the Student-$t$ distribution converges to the normal distribution.
Thus, for large $\nu$ we expect the nonlocal GMMF to be able to denoise images corrupted by Gaussian noise as well. 
This is illustrated in Figure~\ref{Fig:examples_Gaussian}, which shows the \emph{barbara} image (top left) corrupted by additive white Gaussian noise 
of standard deviation $\sigma = 10$ (top right) together with the denoising result using the minimum means squared error estimator proposed in the state-of-the-art 
nonlocal Bayes algorithm~\cite{LBM13} (bottom left) and our nonlocal GMMF with $\nu=1000$ (bottom right), this time applied with the BLUE estimator instead of the mean. 
Again, we chose $n=40$ samples, but since the noise is not very strong we used $ 3\times 3$ patches in this example. Both methods reconstruct the image very well 
and there are no visible differences. 
Note that we show the results obtained after one iteration of the respective algorithms 
in order to see the influence of the different estimators and no other effects. 
The complete denoising procedure proposed in~\cite{LBM13} uses this first-step denoised image as an oracle image for patch selection in a second iteration and applies 
further fine-tuning steps such as a special treatment of homogeneous areas and patch aggregation to obtain the final image. 

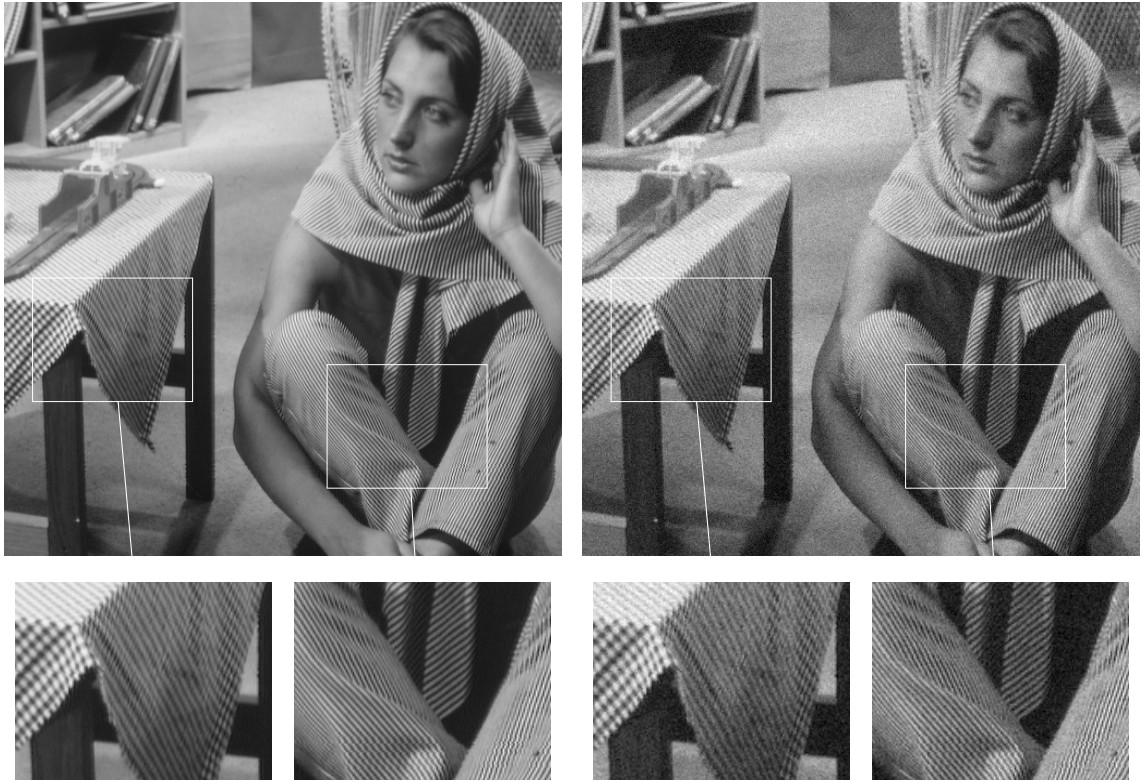
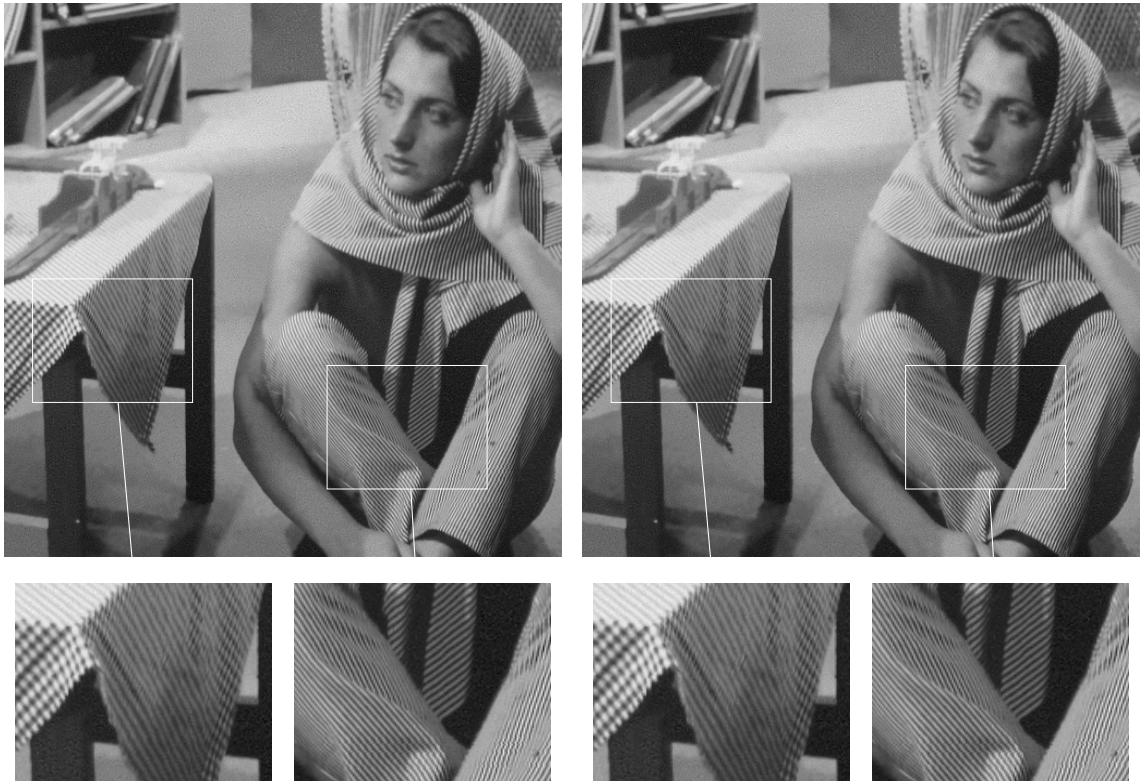
\begin{figure}[htbp]
	\centering	
	\begin{subfigure}[t]{0.49\textwidth}
		\centering
		\tikzsetnextfilename{barbara_orig}
		\begin{tikzpicture}[spy using outlines={%
			every spy in node/.append style={draw = none},%
			every spy on node/.append style={white},connect spies}]
		\node[anchor=south west,inner sep=0pt] (image) at (0,0){\includegraphics[width=.98\textwidth]{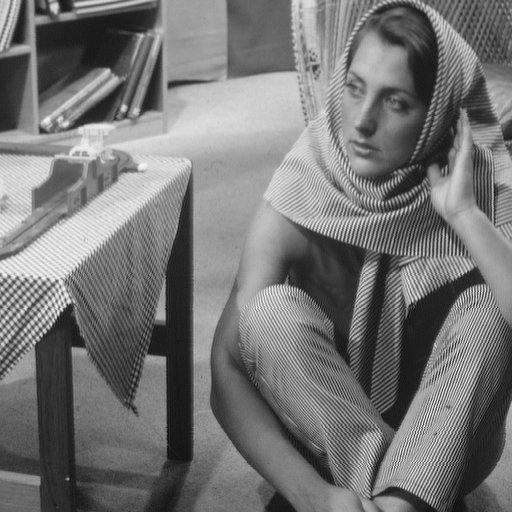}};
		\begin{scope}[x={($ (image.south east) - (image.south west) $ )},%
		y={( $ (image.north west) - (image.south west)$ )},%
		shift={(image.south west)}]
		\node (spy1) at (100/512,200/512) {};
		\coordinate (spy1to) at (0.25,-0.225);
		\spy [rectangle,width=0.45\textwidth, height=.35\textwidth,draw=white,magnification=1.6] on (spy1) in node at (spy1to);
		\node (spy2) at (370/512,120/512) {};
		\coordinate (spy2to) at (0.75,-0.225);
		\spy [rectangle,width=0.45\textwidth, height=.35\textwidth,draw=white,magnification=1.6] on (spy2) in node at (spy2to);
		\end{scope}
		\end{tikzpicture} 
		\caption{PSNR: $+\infty$, SSIM: $1$ }
	\end{subfigure}
	\begin{subfigure}[t]{0.49\textwidth}
		\centering
		\tikzsetnextfilename{barbara_noisy}
		\begin{tikzpicture}[spy using outlines={%
			every spy in node/.append style={draw = none},%
			every spy on node/.append style={white},connect spies}]
		\node[anchor=south west,inner sep=0pt] (image) at (0,0){\includegraphics[width=.98\textwidth]{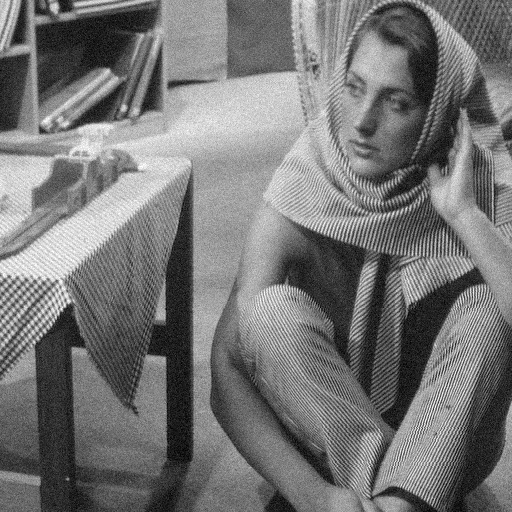}};
		\begin{scope}[x={($ (image.south east) - (image.south west) $ )},%
		y={( $ (image.north west) - (image.south west)$ )},%
		shift={(image.south west)}]
		\node (spy1) at (100/512,200/512) {};
		\coordinate (spy1to) at (0.25,-0.225);
		\spy [rectangle,width=0.45\textwidth, height=.35\textwidth,draw=white,magnification=1.6] on (spy1) in node at (spy1to);
		\node (spy2) at (370/512,120/512) {};
		\coordinate (spy2to) at (0.75,-0.225);
		\spy [rectangle,width=0.45\textwidth, height=.35\textwidth,draw=white,magnification=1.6] on (spy2) in node at (spy2to);
		\end{scope}
		\end{tikzpicture} 
		\caption{PSNR: $28.1241$, SSIM: $0.7140$ }
	\end{subfigure}	\\
	
	\vspace*{0.15cm}
	\begin{subfigure}[t]{0.49\textwidth}
		\centering
		\tikzsetnextfilename{barbara_normal}
		\begin{tikzpicture}[spy using outlines={%
			every spy in node/.append style={draw = none},%
			every spy on node/.append style={white},connect spies}]
		\node[anchor=south west,inner sep=0pt] (image) at (0,0){\includegraphics[width=.98\textwidth]{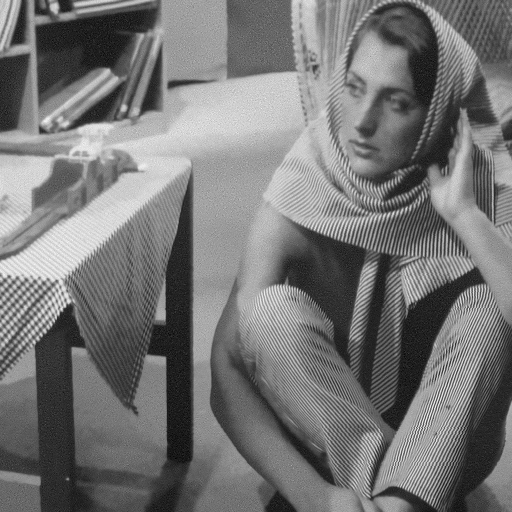}};
		\begin{scope}[x={($ (image.south east) - (image.south west) $ )},%
		y={( $ (image.north west) - (image.south west)$ )},%
		shift={(image.south west)}]
		\node (spy1) at (100/512,200/512) {};
		\coordinate (spy1to) at (0.25,-0.225);
		\spy [rectangle,width=0.45\textwidth, height=.35\textwidth,draw=white,magnification=1.6] on (spy1) in node at (spy1to);
		\node (spy2) at (370/512,120/512) {};
		\coordinate (spy2to) at (0.75,-0.225);
		\spy [rectangle,width=0.45\textwidth, height=.35\textwidth,draw=white,magnification=1.6] on (spy2) in node at (spy2to);
		\end{scope}
		\end{tikzpicture} 
		\caption{PSNR: $31.2445$, SSIM: $0.7873$}
	\end{subfigure}
	\begin{subfigure}[t]{0.49\textwidth}
		\centering
		\tikzsetnextfilename{barbara_our}
		\begin{tikzpicture}[spy using outlines={%
			every spy in node/.append style={draw = none},%
			every spy on node/.append style={white},connect spies}]
		\node[anchor=south west,inner sep=0pt] (image) at (0,0){\includegraphics[width=.98\textwidth]{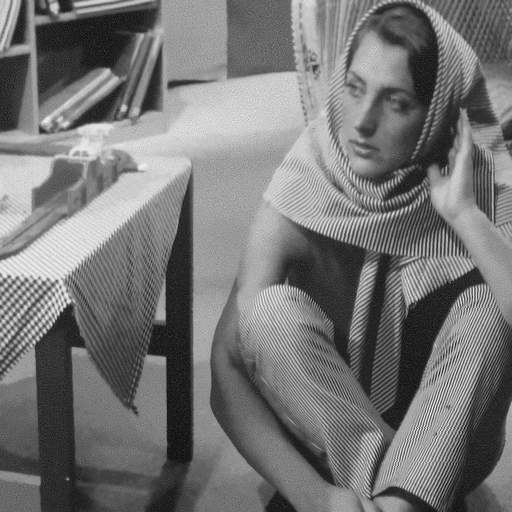}};
		\begin{scope}[x={($ (image.south east) - (image.south west) $ )},%
		y={( $ (image.north west) - (image.south west)$ )},%
		shift={(image.south west)}]
		\node (spy1) at (100/512,200/512) {};
		\coordinate (spy1to) at (0.25,-0.225);
		\spy [rectangle,width=0.45\textwidth, height=.35\textwidth,draw=white,magnification=1.6] on (spy1) in node at (spy1to);
		\node (spy2) at (370/512,120/512) {};
		\coordinate (spy2to) at (0.75,-0.225);
		\spy [rectangle,width=0.45\textwidth, height=.35\textwidth,draw=white,magnification=1.6] on (spy2) in node at (spy2to);
		\end{scope}
		\end{tikzpicture} 
		\caption{PSNR: $31.3024$, SSIM: $0.7898$ }
	\end{subfigure}	
	\caption[]{Denoising of the test image (top left) corrupted with additive Gaussian noise ($\nu=\infty$, $\sigma= 10$) (top right) 
		using the nonlocal MMSE algorithm~\cite{LBM13} (bottom left) and our nonlocal GMMF (bottom right).}
	\label{Fig:examples_Gaussian}	
\end{figure}

\begin{Remark}[Robustness of Parameters]
		Our algorithm depends on several parameters, namely  the patch width $s$, the size of the search
		window $w$, and the number of similar patches $n$. An extensive grid search revealed 
		that the results are rather robust towards small changes in the parameters, 
		and patch sizes between $3\times 3$ and $5\times 5$ as well as $n=40$ to $n=50$ samples yield results that are visually indistinguishable and have similar PSNR values. 
		The size of the search window needs to be adapted to the patch size and the number of samples; it has to be large enough 
		to guarantee that enough similar patches can be found. As a rule of thumb, the stronger the noise, the larger we choose the patch size.
\end{Remark}

\paragraph{Wrapped Cauchy Noise.}
Next, 
we apply Algorithm~\ref{alg:wrapped_Cauchy} to denoise $\SP^1$-valued images $f\colon \GG\to \SP^1$ corrupted by wrapped Cauchy noise, 
\begin{equation}
f_i =  (u_i + \gamma \eta) \modulo 2\pi,\qquad \eta\sim C(0,\gamma),\, \gamma>0,\quad i\in \GG,
\end{equation} 
where we chose $\gamma = 0.1$ corresponding to a moderate noise level, which yields $\rho = \e^{-\gamma} \approx 0.9048$. 
The original image as well as the noisy image are given in the top row of Figure~\ref{Fig:examples_wrapped_Cauchy}. The similarity measure to find similar patches 
is in this case derived from the density of the wrapped Cauchy distribution and it is given by
\begin{equation*}
S(p,q) =\frac{(1-\rho)^4}{\Bigl(1 + \rho^2 -2\rho \cos\left(\frac{p_i-q_i}{2}\right) \Bigr)^{2}}, 
\end{equation*}
which leads to the distance
\begin{equation*}
d(p,q) = \sum_{i=1}^t \log \Biggl(1+\rho^2-2\rho \cos\left(\frac{p_i-q_i}{2}\right)\Biggr).
\end{equation*}
Here, we chose $n=50$ patches of size $5\times 5$  as samples, but this time extracted only their centers, estimated the parameters and restored the image pixelwise. 
We compare our approach with the variational method using an $L_2$ data term and a first and second order TV-regularizer in~\cite{BLSW14} and the nonlocal denoising algorithm based 
on second order statistic~\cite{LNPS16} (NL-MMSE). The results together with the mean-squared reconstruction error are given in the bottom row of Figure~\ref{Fig:examples_wrapped_Cauchy}. 
Both the variational as well as the second order statistical method cannot cope with the impulsiveness of the wrapped Cauchy noise such that several wrong pixels remain, 
which is in particular visible in the background. Furthermore, the edges of the color squares and the transitions in the ellipse and in the circle are rather fringy. 
On the contrary, our method restores the image very well, if at all a slight grain can be observed in the color squares which is due to the pixelwise denoising that 
does not regard neighboring pixels appropriately.

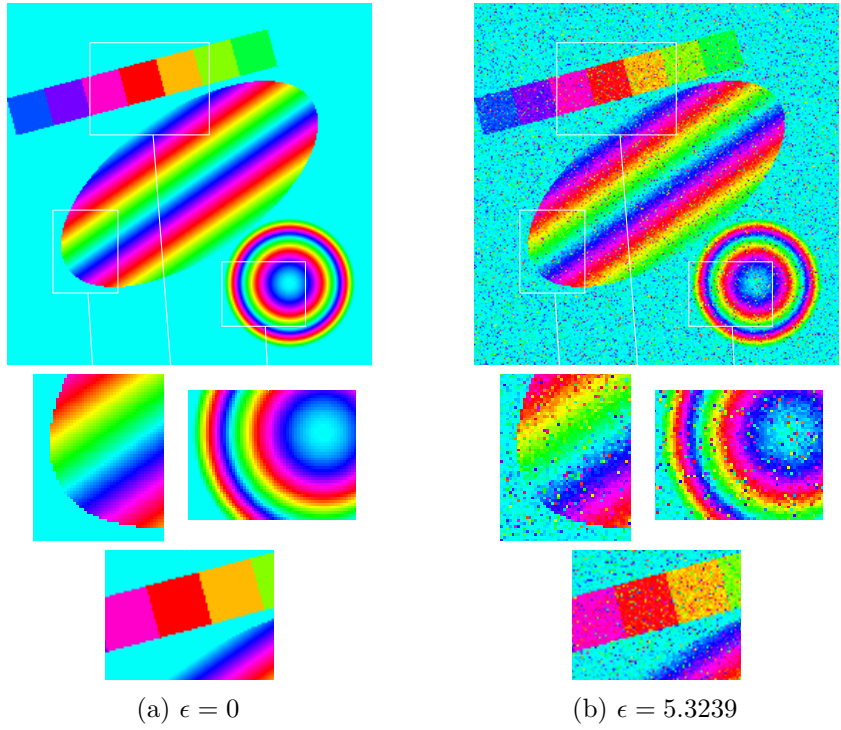
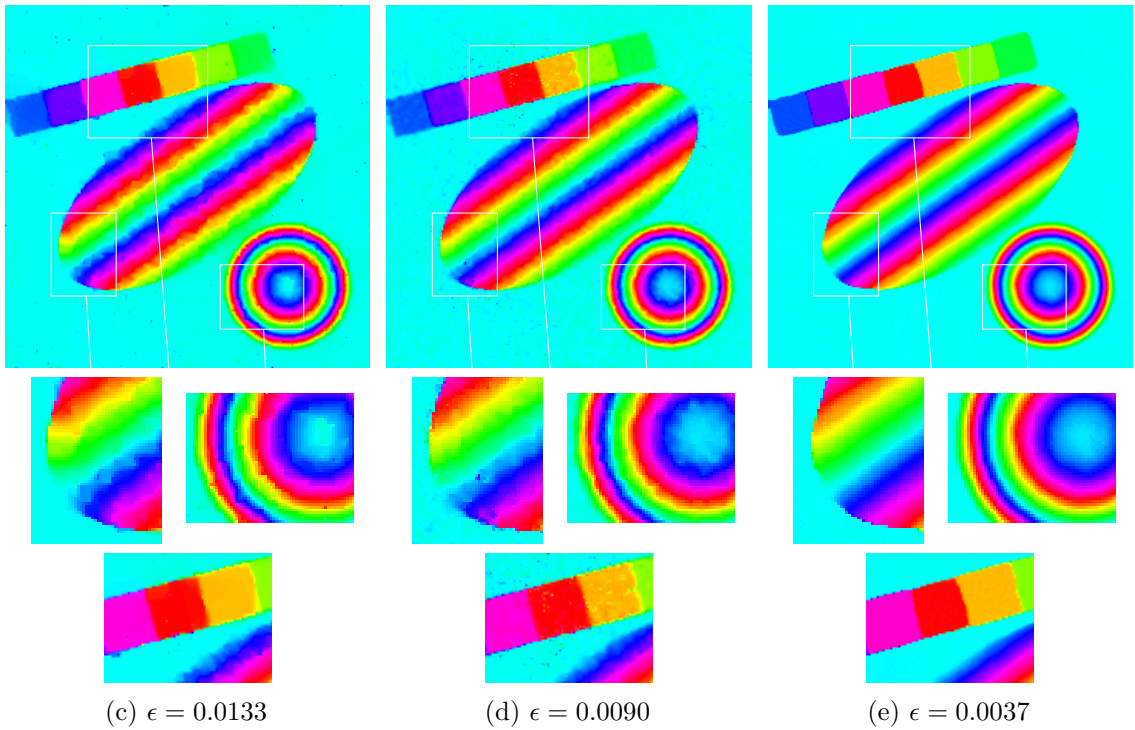
\begin{figure}[tp]
\centering	
\begin{subfigure}[t]{0.32\textwidth}
	\centering
	\tikzsetnextfilename{S1_orig}
	\begin{tikzpicture}[spy using outlines={%
		every spy in node/.append style={draw = none},%
		every spy on node/.append style={white},connect spies}]
	\node[anchor=south west,inner sep=0pt] (image) at (0,0){\includegraphics[width=.98\textwidth]{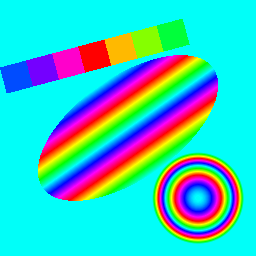}};
	\begin{scope}[x={($ (image.south east) - (image.south west) $ )},%
	y={( $ (image.north west) - (image.south west)$ )},%
	shift={(image.south west)}]
	\node (spy1) at (110/512,160/512) {};
	\coordinate (spy1to) at (0.25,-0.255);
	\spy [rectangle,width=0.35\textwidth, height=.45\textwidth,draw=white,magnification=2.0] on (spy1) in node at (spy1to);
	\node (spy2) at (360/512,100/512) {};
	\coordinate (spy2to) at (0.725,-0.25);
	\spy [rectangle,width=0.45\textwidth, height=.35\textwidth,draw=white,magnification=2.0] on (spy2) in node at (spy2to);
		\node (spy3) at (200/512,390/512) {};
		\coordinate (spy3to) at (0.5,-0.69);
		\spy [rectangle,width=0.45\textwidth, height=.35\textwidth,draw=white,magnification=1.4] on (spy3) in node at (spy3to);	
	\end{scope}
	\end{tikzpicture} 
	\caption{$\epsilon= 0$ }
\end{subfigure}	
\hspace*{1cm}	
\begin{subfigure}[t]{0.32\textwidth}
	\centering
	\tikzsetnextfilename{S1_noisy}
	\begin{tikzpicture}[spy using outlines={%
		every spy in node/.append style={draw = none},%
		every spy on node/.append style={white},connect spies}]
	\node[anchor=south west,inner sep=0pt] (image) at (0,0){\includegraphics[width=.98\textwidth]{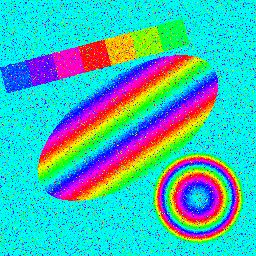}};
	\begin{scope}[x={($ (image.south east) - (image.south west) $ )},%
	y={( $ (image.north west) - (image.south west)$ )},%
	shift={(image.south west)}]
	\node (spy1) at (110/512,160/512) {};
	\coordinate (spy1to) at (0.25,-0.255);
	\spy [rectangle,width=0.35\textwidth, height=.45\textwidth,draw=white,magnification=2.0] on (spy1) in node at (spy1to);
	\node (spy2) at (360/512,100/512) {};
	\coordinate (spy2to) at (0.725,-0.25);
	\spy [rectangle,width=0.45\textwidth, height=.35\textwidth,draw=white,magnification=2.0] on (spy2) in node at (spy2to);
		\node (spy3) at (200/512,390/512) {};
		\coordinate (spy3to) at (0.5,-0.69);
		\spy [rectangle,width=0.45\textwidth, height=.35\textwidth,draw=white,magnification=1.4] on (spy3) in node at (spy3to);	
	\end{scope}
	\end{tikzpicture} 
		\caption{$\epsilon= 5.3239$ }
\end{subfigure}	\\

\vspace*{0.15cm}
\begin{subfigure}[t]{0.32\textwidth}
	\centering
	\tikzsetnextfilename{S1_TV}
	\begin{tikzpicture}[spy using outlines={%
		every spy in node/.append style={draw = none},%
		every spy on node/.append style={white},connect spies}]
	\node[anchor=south west,inner sep=0pt] (image) at (0,0){\includegraphics[width=.98\textwidth]{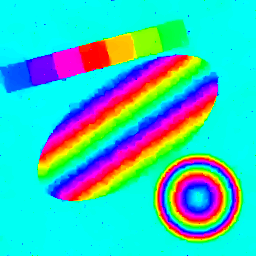}};
	\begin{scope}[x={($ (image.south east) - (image.south west) $ )},%
	y={( $ (image.north west) - (image.south west)$ )},%
	shift={(image.south west)}]
	\node (spy1) at (110/512,160/512) {};
	\coordinate (spy1to) at (0.25,-0.255);
	\spy [rectangle,width=0.35\textwidth, height=.45\textwidth,draw=white,magnification=2.0] on (spy1) in node at (spy1to);
	\node (spy2) at (360/512,100/512) {};
	\coordinate (spy2to) at (0.725,-0.25);
	\spy [rectangle,width=0.45\textwidth, height=.35\textwidth,draw=white,magnification=2.0] on (spy2) in node at (spy2to);
		\node (spy3) at (200/512,390/512) {};
		\coordinate (spy3to) at (0.5,-0.69);
		\spy [rectangle,width=0.45\textwidth, height=.35\textwidth,draw=white,magnification=1.4] on (spy3) in node at (spy3to);	
	\end{scope}
	\end{tikzpicture} 
		\caption{$\epsilon= 0.0133$ }
\end{subfigure}
\begin{subfigure}[t]{0.32\textwidth}
	\centering
	\tikzsetnextfilename{S1_MMSE}
	\begin{tikzpicture}[spy using outlines={%
		every spy in node/.append style={draw = none},%
		every spy on node/.append style={white},connect spies}]
	\node[anchor=south west,inner sep=0pt] (image) at (0,0){\includegraphics[width=.98\textwidth]{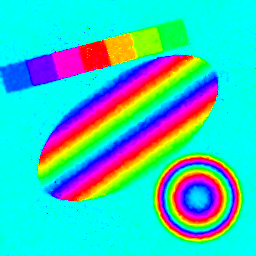}};
	\begin{scope}[x={($ (image.south east) - (image.south west) $ )},%
	y={( $ (image.north west) - (image.south west)$ )},%
	shift={(image.south west)}]
	\node (spy1) at (110/512,160/512) {};
	\coordinate (spy1to) at (0.25,-0.255);
	\spy [rectangle,width=0.35\textwidth, height=.45\textwidth,draw=white,magnification=2.0] on (spy1) in node at (spy1to);
	\node (spy2) at (360/512,100/512) {};
	\coordinate (spy2to) at (0.725,-0.25);
	\spy [rectangle,width=0.45\textwidth, height=.35\textwidth,draw=white,magnification=2.0] on (spy2) in node at (spy2to);
	\node (spy3) at (200/512,390/512) {};
	\coordinate (spy3to) at (0.5,-0.69);
	\spy [rectangle,width=0.45\textwidth, height=.35\textwidth,draw=white,magnification=1.4] on (spy3) in node at (spy3to);	
	\end{scope}
	\end{tikzpicture} 
		\caption{$\epsilon= 0.0090$ }
\end{subfigure}	
\begin{subfigure}[t]{0.32\textwidth}
	\centering
	\tikzsetnextfilename{S1_our}
	\begin{tikzpicture}[spy using outlines={%
		every spy in node/.append style={draw = none},%
		every spy on node/.append style={white},connect spies}]
	\node[anchor=south west,inner sep=0pt] (image) at (0,0){\includegraphics[width=.98\textwidth]{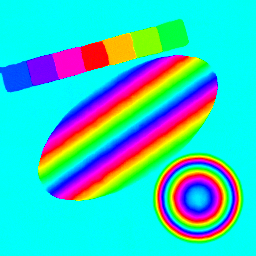}};
	\begin{scope}[x={($ (image.south east) - (image.south west) $ )},%
	y={( $ (image.north west) - (image.south west)$ )},%
	shift={(image.south west)}]
	\node (spy1) at (110/512,160/512) {};
	\coordinate (spy1to) at (0.25,-0.255);
	\spy [rectangle,width=0.35\textwidth, height=.45\textwidth,draw=white,magnification=2.0] on (spy1) in node at (spy1to);
	\node (spy2) at (360/512,100/512) {};
	\coordinate (spy2to) at (0.725,-0.25);
	\spy [rectangle,width=0.45\textwidth, height=.35\textwidth,draw=white,magnification=2.0] on (spy2) in node at (spy2to);		
	\node (spy3) at (200/512,390/512) {};
	\coordinate (spy3to) at (0.5,-0.69);
	\spy [rectangle,width=0.45\textwidth, height=.35\textwidth,draw=white,magnification=1.4] on (spy3) in node at (spy3to);		
	\end{scope}
	\end{tikzpicture} 
		\caption{$\epsilon=0.0037$ }
\end{subfigure}	

\caption[]{Denoising of an $\SP^1$-valued image (top left) corrupted with additive wrapped Cauchy noise ($a = 0$, $\rho= 0.1$) (top right) 
using the variational method~\cite{BLSW14} (bottom left), the nonlocal MMSE method~\cite{LNPS16} (bottom middle) 
and our nonlocal GMMF (bottom right).  }
\label{Fig:examples_wrapped_Cauchy}	
\end{figure}

\paragraph{Denoising of InSAR Data}
Finally, we  apply our denoising approach to a  real-world data set  given in~\cite{RPG97}\footnote{The data is
available online at %
\url{https://earth.esa.int/workshops/ers97/program-details/speeches/rocca-et-al/}%
}. It consists of an interferometric synthetic aperture radar (InSAR) image recorded in 1991 by the ERS-1 satellite capturing topographical information from the Mount Vesuvius. 
InSAR is a radar technique used in geodesy and remote sensing. Based on two or more synthetic
aperture radar (SAR) images, maps of surface deformation or digital elevation are generated,
using differences in the phase of the waves returning to an aircraft or satellite. The technique can
potentially measure millimeter-scale changes in deformation over spans of days to years. It has
applications in the geophysical monitoring of natural hazards, for example earthquakes, volcanoes
and landslides, and in structural engineering, in particular monitoring of subsidence and structural
stability. InSAR produces phase-valued images, i.e. in each pixel the measurement lies on the
circle $\mathbb{S}^1$. The recorded image is shown in Figure~\ref{Fig:examples_InSar} top left. 
While the contour lines of the   
Vesuvius are clearly visible, the image suffers in particular in the middle and at the boundaries  from strong noise. 
The top right image in Figure~\ref{Fig:examples_InSar}  depicts the denoising result   using the variational approach in~\cite{BLSW14}, 
while in the bottom row we show the results obtained with the nonlocal MMSE approach~\cite{LNPS16} (left) and with  our method (right), where we used  the parameters $\gamma = 0.1$, $n=50$ and $5\times 5$ patches. 
While both methods restore the outer contour lines very well, our approach yields a sharper image and preserves much more details,  
for instance the fine contour lines in the middle of the image. 

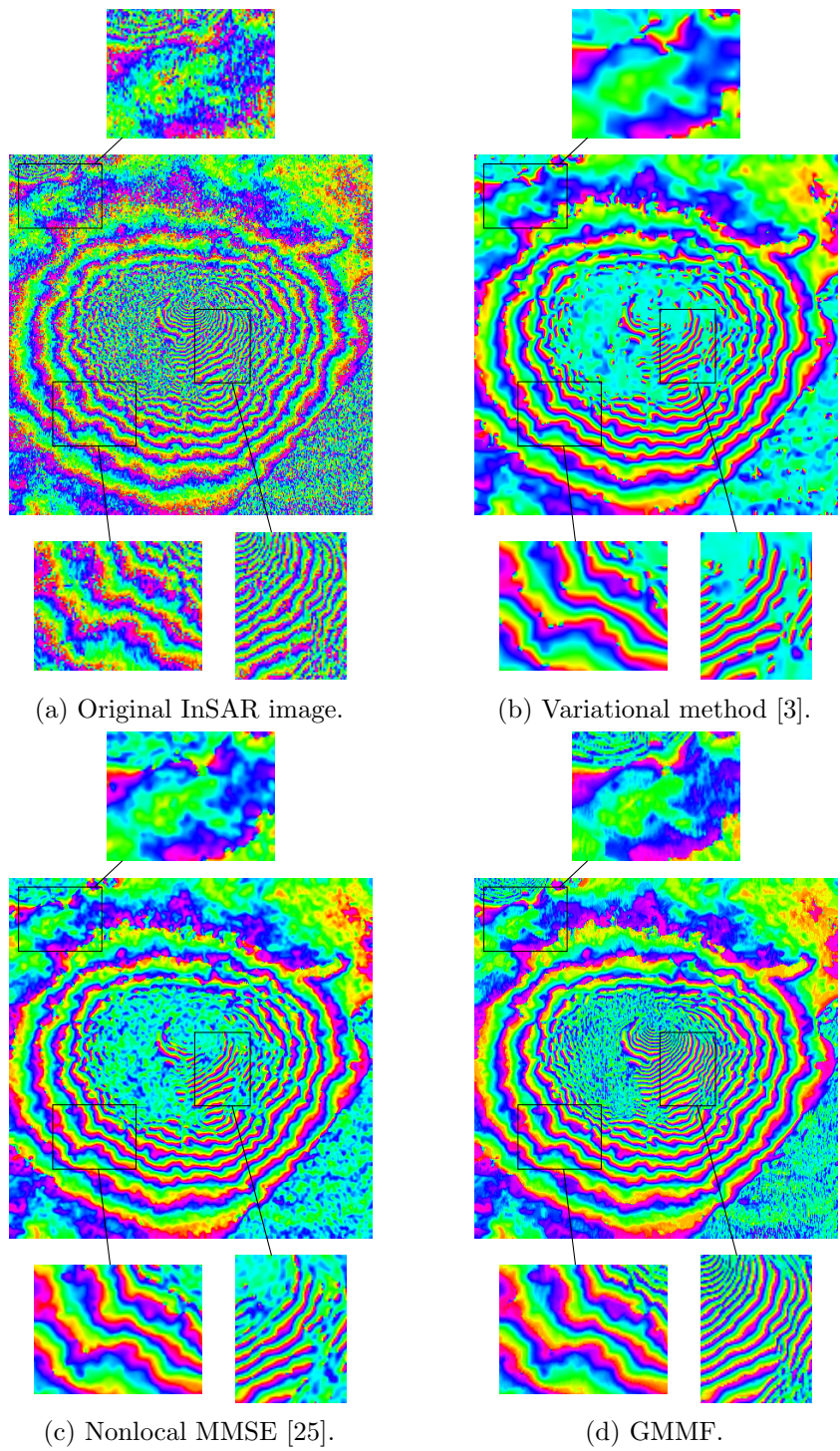
\begin{figure}[htp]
\centering	
\begin{subfigure}[t]{0.32\textwidth}
	\centering
	\tikzsetnextfilename{InSAR_orig}
	\begin{tikzpicture}[spy using outlines={%
		every spy in node/.append style={draw = none},%
		every spy on node/.append style={black},connect spies}]
	\node[anchor=south west,inner sep=0pt] (image) at (0,0){\includegraphics[width=.98\textwidth]{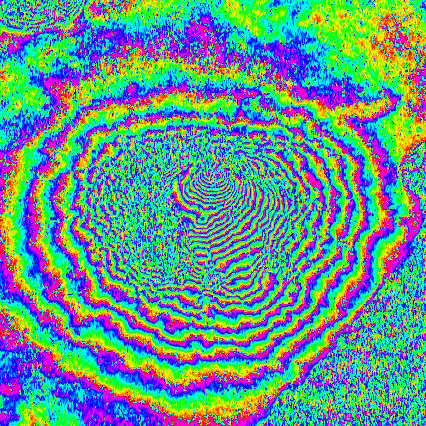}};
	\begin{scope}[x={($ (image.south east) - (image.south west) $ )},%
	y={( $ (image.north west) - (image.south west)$ )},%
	shift={(image.south west)}]
	\node (spy1) at (100/426,120/426) {};
	\coordinate (spy1to) at (0.3,-0.25);
	\spy [rectangle,width=0.45\textwidth, height=.35\textwidth,draw=black,magnification=2.0] on (spy1) in node at (spy1to);
	\node (spy2) at (60/426,377/426) {};
	\coordinate (spy2to) at (0.5,1.225);
	\spy [rectangle,width=0.45\textwidth, height=.35\textwidth,draw=black,magnification=2.0] on (spy2) in node at (spy2to);		
	\node (spy3) at (250/426,200/426) {};
	\coordinate (spy3to) at (0.775,-0.25);
	\spy [rectangle,width=0.3\textwidth, height=.4\textwidth,draw=black,magnification=2] on (spy3) in node at (spy3to);	
	\end{scope}
	\end{tikzpicture} 
		\caption{Original InSAR image.}
\end{subfigure}
\hspace*{1cm}
\begin{subfigure}[t]{0.32\textwidth}
	\centering
	\tikzsetnextfilename{S1_CPPS}
	\begin{tikzpicture}[spy using outlines={%
		every spy in node/.append style={draw = none},%
		every spy on node/.append style={black},connect spies}]
	\node[anchor=south west,inner sep=0pt] (image) at (0,0){\includegraphics[width=.98\textwidth]{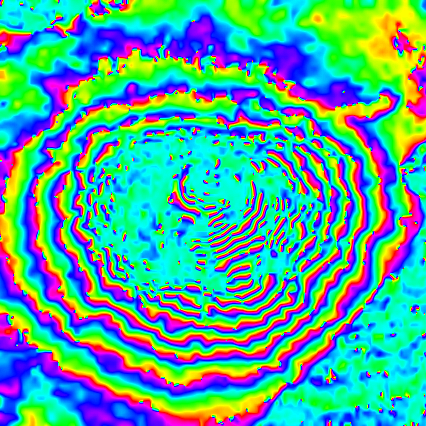}};
	\begin{scope}[x={($ (image.south east) - (image.south west) $ )},%
	y={( $ (image.north west) - (image.south west)$ )},%
	shift={(image.south west)}]
	\node (spy1) at (100/426,120/426) {};
	\coordinate (spy1to) at (0.3,-0.25);
	\spy [rectangle,width=0.45\textwidth, height=.35\textwidth,draw=black,magnification=2.0] on (spy1) in node at (spy1to);
	\node (spy2) at (60/426,377/426) {};
	\coordinate (spy2to) at (0.5,1.225);
	\spy [rectangle,width=0.45\textwidth, height=.35\textwidth,draw=black,magnification=2.0] on (spy2) in node at (spy2to);		
	\node (spy3) at (250/426,200/426) {};
	\coordinate (spy3to) at (0.775,-0.25);
	\spy [rectangle,width=0.3\textwidth, height=.4\textwidth,draw=black,magnification=2] on (spy3) in node at (spy3to);		
	\end{scope}
	\end{tikzpicture} 
		\caption{Variational method~\cite{BLSW14}.}
\end{subfigure}	

\hspace{1cm}

\begin{subfigure}[t]{0.32\textwidth}
	\centering
	\tikzsetnextfilename{InSAR_NLMMSE}
	\begin{tikzpicture}[spy using outlines={%
		every spy in node/.append style={draw = none},%
		every spy on node/.append style={black},connect spies}]
	\node[anchor=south west,inner sep=0pt] (image) at (0,0){\includegraphics[width=.98\textwidth]{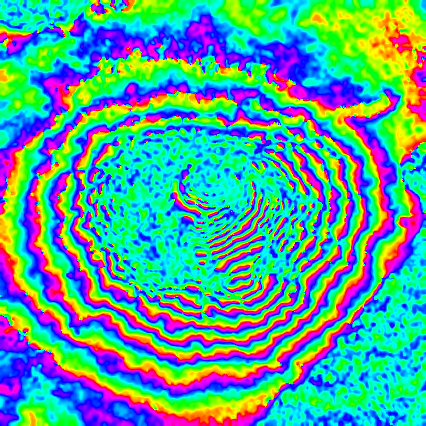}};
	\begin{scope}[x={($ (image.south east) - (image.south west) $ )},%
	y={( $ (image.north west) - (image.south west)$ )},%
	shift={(image.south west)}]
	\node (spy1) at (100/426,120/426) {};
	\coordinate (spy1to) at (0.3,-0.25);
	\spy [rectangle,width=0.45\textwidth, height=.35\textwidth,draw=black,magnification=2.0] on (spy1) in node at (spy1to);
	\node (spy2) at (60/426,377/426) {};
	\coordinate (spy2to) at (0.5,1.225);
	\spy [rectangle,width=0.45\textwidth, height=.35\textwidth,draw=black,magnification=2.0] on (spy2) in node at (spy2to);		
	\node (spy3) at (250/426,200/426) {};
	\coordinate (spy3to) at (0.775,-0.25);
	\spy [rectangle,width=0.3\textwidth, height=.4\textwidth,draw=black,magnification=2] on (spy3) in node at (spy3to);		
	\end{scope}
	\end{tikzpicture} 
			\caption{Nonlocal MMSE~\cite{LNPS16}.}
\end{subfigure}
\hspace*{1cm}
\begin{subfigure}[t]{0.32\textwidth}
	\centering
	\tikzsetnextfilename{InSAR_NGMMF}
	\begin{tikzpicture}[spy using outlines={%
		every spy in node/.append style={draw = none},%
		every spy on node/.append style={black},connect spies}]
	\node[anchor=south west,inner sep=0pt] (image) at (0,0){\includegraphics[width=.98\textwidth]{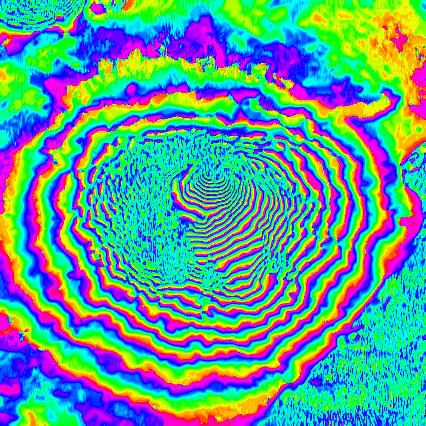}};
	\begin{scope}[x={($ (image.south east) - (image.south west) $ )},%
	y={( $ (image.north west) - (image.south west)$ )},%
	shift={(image.south west)}]
	\node (spy1) at (100/426,120/426) {};
	\coordinate (spy1to) at (0.3,-0.25);
	\spy [rectangle,width=0.45\textwidth, height=.35\textwidth,draw=black,magnification=2.0] on (spy1) in node at (spy1to);
	\node (spy2) at (60/426,377/426) {};
	\coordinate (spy2to) at (0.5,1.225);
	\spy [rectangle,width=0.45\textwidth, height=.35\textwidth,draw=black,magnification=2.0] on (spy2) in node at (spy2to);		
	\node (spy3) at (250/426,200/426) {};
	\coordinate (spy3to) at (0.775,-0.25);
	\spy [rectangle,width=0.3\textwidth, height=.4\textwidth,draw=black,magnification=2] on (spy3) in node at (spy3to);		
	\end{scope}
	\end{tikzpicture} 
			\caption{GMMF.}
\end{subfigure}	
\caption[]{Denoising of an InSAR image  of the Mount Vesuvius (top left)   using the variational method~\cite{BLSW14} (top right), the nonlocal MMSE approach~\cite{LNPS16} and our nonlocal GMMF (bottom right).  }
\label{Fig:examples_InSar}	
\end{figure}

In a second experiment, we examine the influence of the denoising on the reconstruction of the absolute phase. 
In order to do so, we use the phase unwrapping method PUMA proposed in~\cite{BV07}, 
which recasts the problem as a max flow-min cut problem and is among the current state-of-the-art algorithms for phase unwrapping. We used the code provided by the authors of~\cite{BV07} 
with the default parameters choices. The results of the corresponding images of Figure~\ref{Fig:examples_InSar} are given in Figure~\ref{Fig:unwrapping}. 
While the absolute phase of the original image is strongly affected by the noise, the results based on the TV-regularized image are  in some parts oversmoothed, 
so that the reconstruction does not provide any details in those areas. On the other hand, the reconstructions based on the NL-MMSE denoised image and based on our method provide much better results, 
where our method leads to a slightly finer resolution.

\begin{figure}[htp]
	\centering
		\begin{subfigure}[t]{0.49\textwidth}
		\centering
		\includegraphics[width=\textwidth]{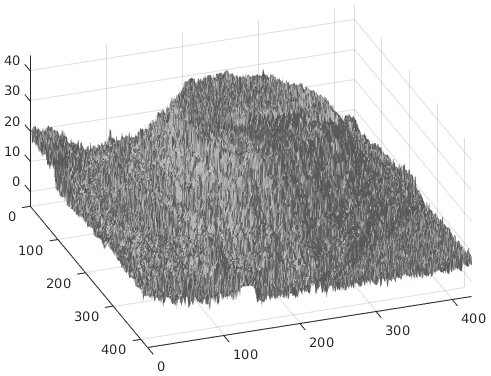}
			\caption{Original InSAR image.}		
	\end{subfigure}
		\begin{subfigure}[t]{0.49\textwidth}
			\centering
			\includegraphics[width=\textwidth]{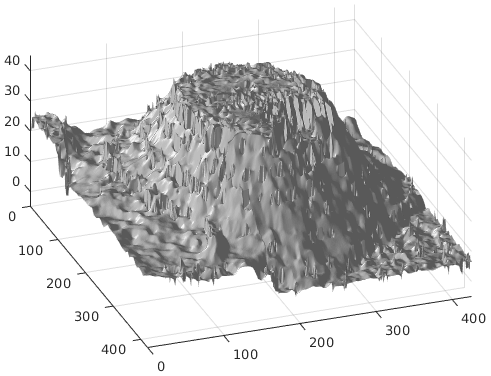}	
						\caption{Variational method~\cite{BLSW14}.}			
		\end{subfigure}

			\begin{subfigure}[t]{0.49\textwidth}
				\centering
				\includegraphics[width=\textwidth]{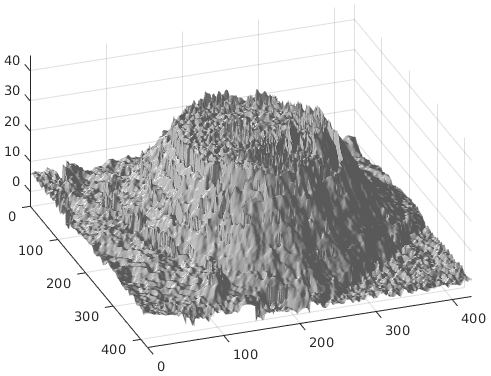}
							\caption{Nonlocal MMSE~\cite{LNPS16}.}				
			\end{subfigure}
					\begin{subfigure}[t]{0.49\textwidth}
				\centering
				\includegraphics[width=\textwidth]{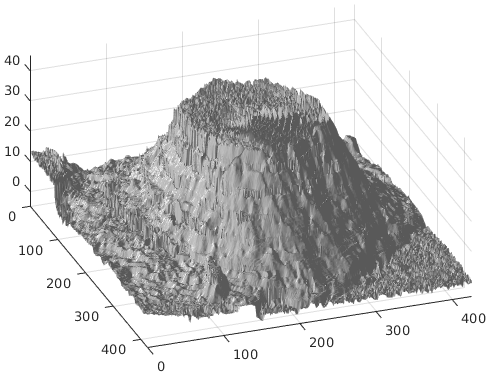}	
											\caption{GMMF.}					
			\end{subfigure}		
\caption[]{Reconstruction of the absolute phase images from Figure~\ref{Fig:examples_InSar} using the reconstruction algorithm PUMA~\cite{BV07}.  }\label{Fig:unwrapping}
\end{figure}

\section{Conclusion} \label{sec:conclusions}
We introduced a generalized multivariate myriad filter based on (weighted) ML estimation of the multivariate Student-$t$ distribution 
and illustrated  its usage in a nonlocal robust denoising approach. 
Furthermore, we showed how a special case of our algorithm 
can be related to projected normal distributions on the sphere $\SP^{d-1}$ and for $d=2$ further to the wrapped Cauchy distribution on the circle $\SP^1$, 
which gives rise to robust denoising strategies for $\mathbb S^1$-valued images.

There are different directions for future work: 
First, we would like to extend our analysis to the case 
that additionally the degrees of freedom parameter $\nu$ is unknown and to SMM. 
Although an EM algorithm has already been derived for this case in~\cite{LLT89}, to the best of our knowledge
there do not exist results concerning existence and/or uniqueness of the joint ML estimator. 

Concerning our denoising approach, fine tuning steps
as discussed in~\cite{LCBM12} such as aggregations of patches~\cite{SDB18},
the use of an oracle image or a variable patch size to better
cope   with textured and homogeneous image regions may   improve the denoising results. 
Further, in all our examples we used uniform weights, but
weights based for instance on spatial distance or similarity would make sense as well. 
Another question is how to incorporate linear operators
(blur, missing pixels) into the image restoration. To this end, SMM could be used as priors within variational models. 

\subsection*{Acknowledgments} 
Funding by the ANR-DFG project {SUPREMATIM} STE 571/14-1 is gratefully acknowledged. 
Furthermore, we thank J. Delon for fruitful discussions and 
J. M. Bioucas-Dias for providing the code of the phase unwrapping algorithm, see \cite{BV07}.

\appendix
\section{Appendix}\label{App1}
{\begin{small}
		Proof of Lemma \ref{Prop:wrapped_Cauchy}.
(i)	By definition of $\Theta$ in terms of  $\Phi$ we have for the corresponding probability density functions
	\begin{align*}
	f_\Theta (\theta) 
	&= f_{2\Phi} (\theta) + f_{2\Phi} \left(\theta- \mathrm{sgn}(\theta) 2 \pi \right) 
	= \frac12 \Biggl[ f_\Phi \left(\frac{\theta}{2} \right) + f_\Phi \left( \frac{\theta}{2} - \mathrm{sgn}(\theta) \pi \right) \Biggr]\\
	&= \frac12 \left[ f_\Phi \left(\frac{\theta}{2} \right) + f_\Phi \Biggl(\left(\frac{\theta}{2} + \pi\right)\modulo 2\pi \Biggr)\right] = f_\Phi \left( \frac{\theta}{2} \right).
	\end{align*}
	Therefore we have to show that $f_\Phi \left( \frac{\theta}{2} \right) = g_w (\theta|a,\gamma)$. We parametrize 
	$
	x = \begin{pmatrix}
	\cos(\phi)\\
	\sin(\phi)
	\end{pmatrix}
	$
	and using relations of trigonometric functions, we obtain
	\begin{align*}
	x^\tT \Sigma^{-1}x 
	& = \frac{1}{\sigma_{11}\sigma_{22} - \sigma_{12}^2}\bigl(\sigma_{22} x_1^2 - 2\sigma_{12} x_1 x_2 + \sigma_{11} x_2^2\bigr)\\
	& = \frac{1}{\sigma_{11}\sigma_{22} - \sigma_{12}^2} \bigl(\sigma_{22} \cos^2(\phi) + \sigma_{11} \sin^2(\phi) - 2\sigma_{12} \cos(\phi) \sin(\phi) \bigr)\\
	& = \frac{1}{\sigma_{11}\sigma_{22} - \sigma_{12}^2} \left(\frac{1}{2} (\sigma_{11} + \sigma_{22}) - \frac{1}{2}(\sigma_{11}-\sigma_{22})\cos(2\phi) - \sigma_{12} \sin(2\phi)\right)\\
	& = \frac{\frac{1}{2} (\sigma_{11} + \sigma_{22})}{\sigma_{11}\sigma_{22}- \sigma_{12}^2}
	\left(1 - \frac{\sigma_{11}-\sigma_{22}}{\sigma_{11} + \sigma_{22}}\cos(2\phi) - \frac{2\sigma_{12}}{\sigma_{11} + \sigma_{22}} \sin(2\phi)\right)\\
	& = \frac{\frac{1}{2} \tr(\Sigma)}{|\Sigma|}
	\left(1 - \frac{\sigma_{11}-\sigma_{22}}{\sigma_{11} + \sigma_{22}}\cos(2\phi) - \frac{2\sigma_{12}}{\sigma_{11} + \sigma_{22}} \sin(2\phi)\right).
	\end{align*}
	Setting 
	$\zeta_1 = \frac{\sigma_{11}-\sigma_{22}}{\sigma_{11} + \sigma_{22}}$ and $\zeta_2 =\frac{2\sigma_{12}}{\sigma_{11}+\sigma_{22}}$
	such that
	\begin{equation*}
	\sqrt{1-\zeta_1^2-\zeta_2^2} = \frac{2\sqrt{\sigma_{11}\sigma_{22}-\sigma_{12}^2}}{\sigma_{11} + \sigma_{22}}
	= \frac{\sqrt{|\Sigma|}}{\frac{1}{2}\tr(\Sigma)}
	\end{equation*}
	we get
	\begin{equation}\label{nenner}
	x^\tT \Sigma^{-1}x  
	= 
	\frac{1-\zeta_1 \cos(2\phi) -\zeta_2 \sin (2\phi)}{ \sqrt{|\Sigma|} \sqrt{1-\zeta_1^2-\zeta_2^2} }
	\end{equation}
	and
	\begin{equation*}
	f_\Phi\left( \frac{\theta}{2} \right) = \frac{1}{2\pi}  \frac{\sqrt{1-\zeta_1^2-\zeta_2^2}}{1-\zeta_1\cos(\theta) - \xi_2\sin(\theta)}.
	\end{equation*}
	This has exactly the form \eqref{WC} of $g_w$ if we identify
	\begin{align}
	\xi_1 = \frac{2\rho}{1+\rho^2} \cos(a) = \zeta_1 = \frac{\sigma_{11}-\sigma_{22}}{\sigma_{11} + \sigma_{22}},\quad
	\xi_2 = \frac{2\rho}{1+\rho^2} \sin(a) = \zeta_2 = \frac{2\sigma_{12}}{\sigma_{11}+\sigma_{22}}.
	\end{align}
	Squaring and adding these equations leads by observing that $\rho < 1$ to \eqref{rho} and by devision of the equations to \eqref{a}.
	\\[1ex]
(ii) We first compute the density of $\frac{\Theta}{2}+ \pi \Xi$ as
	\begin{align*}
		f_{\frac{\Theta}{2}+ \pi \Xi}(\theta) = \frac{1}{4} 
		f_{\frac{\Theta}{2}}(\theta + \pi)\mathbbm{1}_{\left[-\frac{3\pi }{2},-\frac{\pi }{2}\right)}(\theta)+\frac{1}{2}f_{\frac{\Theta}{2}}(\theta  )
		\mathbbm{1}_{\left[-\frac{ \pi }{2}, \frac{\pi }{2}\right]}(\theta)+\frac{1}{4} f_{\frac{\Theta}{2}}(\theta - \pi)\mathbbm{1}_{\left[ \frac{\pi }{2}, \frac{3\pi }{2}\right)}(\theta).
	\end{align*}
	Therewith, the density of $\left(\frac{\Theta}{2}+ \pi \Xi\right)_{2\pi}:=\left(\frac{\Theta}{2}+ \pi \Xi\right)\modulo 2\pi$ is given by
	$$
		f_{\left(\frac{\Theta}{2}+ \pi \Xi\right)_{2\pi}}(\theta)
	 =  	{f_{\frac{\Theta}{2}+ \pi \Xi}(\theta)}\mathbbm{1}_{\left[ -\pi, \pi\right)}(\theta)+ 	\underbrace{f_{\frac{\Theta}{2}+ \pi \Xi}(\theta+2\pi)\mathbbm{1}_{\left[ -\pi,-\frac{\pi }{2}\right)}(\theta)}_{(\ast)}+	\underbrace{f_{\frac{\Theta}{2}+ \pi \Xi}(\theta-2\pi )\mathbbm{1}_{\left[ \frac{\pi }{2}, \pi\right)}(\theta)}_{(\ast\ast)}.
	$$
	We calculate
	\begin{align*}
		(\ast) 
		=& \tfrac{1}{4} f_{\frac{\Theta}{2}}(\theta+ 2\pi + \pi)\mathbbm{1}_{\left[-\frac{3\pi }{2},-\frac{\pi }{2}\right)}(\theta+ 2\pi)\mathbbm{1}_{\left[ -\pi,-\frac{\pi }{2}\right)}(\theta)
		 +\tfrac{1}{2}f_{\frac{\Theta}{2}}(\theta + 2\pi  )\mathbbm{1}_{\left[-\frac{ \pi }{2}, \frac{\pi }{2}\right)}(\theta+ 2\pi)\mathbbm{1}_{\left[ -\pi,-\frac{\pi }{2}\right)}(\theta)\\
		&+
		\tfrac{1}{4} f_{\frac{\Theta}{2}}(\theta+ 2\pi - \pi)\mathbbm{1}_{\left[ \frac{\pi }{2}, \frac{3\pi }{2}\right)}(\theta+ 2\pi)\mathbbm{1}_{\left[ -\pi,-\frac{\pi }{2}\right)}(\theta)
		 = 
		 \tfrac{1}{4} f_{\frac{\Theta}{2}}(\theta+ \pi)\mathbbm{1}_{\left[ -\pi,-\frac{\pi }{2}\right)}(\theta)
	\\
		(\ast \ast) =& \tfrac{1}{4} f_{\frac{\Theta}{2}}(\theta- 2\pi + \pi)\mathbbm{1}_{\left[-\frac{3\pi }{2},-\frac{\pi }{2}\right)}(\theta- 2\pi)\mathbbm{1}_{\left[ \frac{\pi }{2},\pi\right)}(\theta)
		 +\tfrac{1}{2}f_{\frac{\Theta}{2}}(\theta - 2\pi  )\mathbbm{1}_{\left[-\frac{ \pi }{2}, \frac{\pi }{2}\right)}(\theta- 2\pi)\mathbbm{1}_{\left[ \frac{\pi }{2},\pi\right)}(\theta)\\
		&+\tfrac{1}{4} f_{\frac{\Theta}{2}}(\theta- 2\pi - \pi)\mathbbm{1}_{\left[ \frac{\pi }{2}, \frac{3\pi }{2}\right)}(\theta- 2\pi)\mathbbm{1}_{\left[ \frac{\pi }{2},\pi\right)}(\theta)
	 = 	\tfrac{1}{4} f_{\frac{\Theta}{2}}(\theta-  \pi ) \mathbbm{1}_{\left[ \frac{\pi }{2},\pi\right)}(\theta),
		\end{align*}
	which results in
	\begin{align*}
		 	f_\Phi(\theta) &= f_{\left(\frac{\Theta}{2}+ \pi \Xi\right)\modulo 2\pi}(\theta) \\
		 	&= \frac{1}{4} f_{\frac{\Theta}{2}}(\theta + \pi)\mathbbm{1}_{\left[-\pi,-\frac{\pi }{2}\right)}(\theta)+\frac{1}{2}f_{\frac{\Theta}{2}}(\theta  )\mathbbm{1}_{\left[-\frac{ \pi }{2}, \frac{\pi }{2}\right]}(\theta)+\frac{1}{4} f_{\frac{\Theta}{2}}(\theta - \pi)\mathbbm{1}_{\left[ \frac{\pi }{2},\pi\right)}(\theta)\\
		 	& +\frac{1}{4} f_{\frac{\Theta}{2}}(\theta+ \pi)\mathbbm{1}_{\left[ -\pi,-\frac{\pi }{2}\right)}(\theta)+	\frac{1}{4} f_{\frac{\Theta}{2}}(\theta-  \pi ) \mathbbm{1}_{\left[ \frac{\pi }{2},\pi\right)}(\theta)\\
		 	& = \frac{1}{2} f_{\frac{\Theta}{2}}(\theta + \pi)\mathbbm{1}_{\left[-\pi,-\frac{\pi }{2}\right)}(\theta)+\frac{1}{2}f_{\frac{\Theta}{2}}(\theta  )\mathbbm{1}_{\left[-\frac{ \pi }{2}, \frac{\pi }{2}\right]}(\theta)+\frac{1}{2} f_{\frac{\Theta}{2}}(\theta - \pi)\mathbbm{1}_{\left[ \frac{\pi }{2},\pi\right)}(\theta)\\
		 	& = \frac{1}{2}f_{\frac{\Theta}{2}}(\theta  )\mathbbm{1}_{\left[-\frac{ \pi }{2}, \frac{\pi }{2}\right)}(\theta)+\frac{1}{2} f_{\frac{\Theta}{2}}\bigl((\theta + \pi)\modulo 2\pi\bigr)\mathbbm{1}_{\left[-\pi ,\pi\right)\setminus\left[- \frac{\pi }{2},\frac{\pi }{2}\right)}(\theta)\\
		 	& =f_{{\Theta}}(2\theta  )\mathbbm{1}_{\left[-\frac{ \pi }{2}, \frac{\pi }{2}\right)}(\theta)+f_{{\Theta}}\bigl((2\theta + \pi)\modulo 2\pi)\bigr)\mathbbm{1}_{\left[-\pi ,\pi\right)\setminus\left[- \frac{\pi }{2},\frac{\pi }{2}\right)}(\theta) .
	\end{align*}
	With the same identifications as in part (i) we obtain the assertion. \hfill $\Box$
\end{small}

\bibliographystyle{abbrv}
\bibliography{Student_t}

\end{document}